\documentclass[oneside,a4paper,english, 11pt]{amsart}

\input{artmymacros.sty}

\title[Potentially crystalline deformation rings and Serre weight conjectures]{Potentially crystalline deformation rings and Serre weight conjectures: \\ \textit{Shapes and Shadows}}

\author{Daniel Le}
\address{Department of Mathematics,
University of Toronto,
40 St. George Street,
Toronto, ON M5S 2E4, Canada}
\email{le@math.toronto.edu}

\author{Bao V. Le Hung}
\address{School of Mathematics,
Institute for Advanced Studies, 
1 Einstein Drive, 
Princeton NJ 08540, USA}
\email{lhvietbao@googlemail.com}

\author{Brandon Levin}
\address{Department of Mathematics,
University of Arizona, 
617 N. Santa Rita Avenue,
P.O. Box 210089,
Tucson, Arizona 85721, USA}
\email{bwlevin@math.arizona.edu}

\author{Stefano Morra}
\address{Institut Montpelli\'erain A. Grothendieck,
Universit\'e de Montpellier,
Cc 051, Place E. Bataillon,
34095 Montpellier Cedex, France}
\email{stefano.morra@umontpellier.fr}

\subjclass[2010]{11F33, 11F80.}

\begin{document}

\begin{abstract}  We prove the weight part of Serre's conjecture in generic situations for forms of $U(3)$ which are compact at infinity and split  at places dividing $p$ as conjectured by \cite{herzig-duke}.  We also prove automorphy lifting theorems in dimension three.   The key input is an explicit description of tamely potentially crystalline deformation rings with Hodge-Tate weights $(2,1,0)$ for $K/\Qp$ unramified combined with patching techniques.  Our results show that the (geometric) Breuil-M\'ezard conjectures hold for these deformation rings. 
\end{abstract}

\maketitle

\tableofcontents

\section{Introduction}

Let $p$ be a prime. Serre's modularity conjecture (\cite{serre-inv}) predicts that any continuous, irreducible, odd Galois representation $\overline{r}:G_{\mathbb{Q}}\defeq \Gal(\overline{\mathbb{Q}}/\Q)\rightarrow \GL_2(\overline{\F}_p)$ is modular.
In other words, there exists a cuspidal modular form $f=\sum_{n>0}a_nq^n$, which is an eigenvector for the Hecke operators, such that $\mathrm{tr}(\overline{r}\left(\mathrm{Frob}_{\ell}\right))\equiv a_{\ell}\,\mod{p}$ for all primes $\ell\neq p$ not dividing the level of $f$.  In \cite{serre-duke}, Serre made his conjecture more precise, by specifying the \emph{minimal weight} (for prime to $p$ level) of such a modular form. More precisely, if $\overline{r}$ is modular, then the set of weights in which a modular form $f$  associated to $\overline{r}$ occurs is determined in an explicit way from the local datum $\rhobar\defeq \rbar\vert_{\Gal(\overline{\Q}_p/\Q_p)}$. 
In generalizations of Serre's conjecture beyond $\GL_2(\Q)$, there is no obvious notion of minimality, and it is more natural to attach to $\rhobar$ a set of irreducible mod $p$ representations, or \emph{Serre weights}, of the rational points of a general linear group over a finite field.
This is referred to as the weight part of Serre's conjecture or more succinctly, the \emph{Serre weight conjectures}. 

There has been considerable progress on generalizations of Serre's weight conjecture in dimension two (the Buzzard-Diamond-Jarvis conjecture) (\cite{BDJ}, \cite{schein-crelle}, \cite{gee-kisin}, \cite{GLS15}, \cite{newton13},...), leading to complete results for 2-dimensional Galois representations. A key insight in \cite{gee-kisin} connects the weight part of Serre's conjecture to the Breuil-M\'ezard conjecture (\cite{BM}, \cite{BM11}), and its geometrization (\cite{EG}, \cite{EGstacks}), which predicts the multiplicities of the special fibers of deformation spaces (or, more generally, moduli stacks) of local Galois representations when $\ell = p$ in terms of Serre weights of general linear groups. In particular, a good understanding of the geometry of local Galois deformation spaces leads naturally to modularity lifting results, Breuil-M\'ezard, and the weight part of Serre's conjecture, via the patching techniques of Kisin-Taylor-Wiles. 

In dimension two,  potentially Barsotti-Tate (BT) deformation rings were studied via moduli of finite flat group schemes (\cite{kisin-annals}, \cite{breuil-annals}) leading to explicit presentations when $K/\Qp$ is unramified (\cite{BM11, EGS}). The geometry of these (potentially) BT-deformation rings is a key input into the proof of the weight part of Serre's conjecture in \cite{gee-kisin} and provides evidence for mod $p$ local Langlands. However, a satisfactory understanding of the $n$-dimensional analogue, potentially crystalline deformation rings with Hodge-Tate weights $(n-1, n-2, \ldots, 0)$, seemed out of reach, due to the difficulty of understanding the monodromy operator in the theory of Breuil-Kisin modules.

In this paper, we overcome this difficulty in dimension 3 to give a description of the local deformation rings $R^{(2,1,0), \tau}_{\rhobar}$ for $K/\Qp$ unramified and $\tau$ a generic tame inertial type. We thereby obtain the first examples in dimension greater than 2 of Galois deformation rings which are neither ordinary nor Fontaine-Laffaille.  Our results are consistent with the Breuil-M\'ezard conjecture and lead to improvements in modularity lifting. 

\paragraph{\emph{\textbf{Results on local deformation spaces.}}}

Let  $K/\Qp$ be a finite unramified extension. We fix a sufficiently large finite extension $E/\Qp$, $\cO$ its ring of integers and $\F$ its residue field (the \emph{rings of coefficients} for our representations). 

Let $\tau: I_{\Qp}\ra \GL_3(\cO)$ be a tame inertial type  and $\rhobar: G_K\ra\GL_3(\F)$ be a continuous Galois representation. In Definitions \ref{gencond} and \ref{defn n-gen}, we introduce a mild condition on the inertial weights of $\tau$ and $\rhobar$, which we call \emph{genericity}.
Our main local results  (cf. Corollary  \ref{dring}, Theorem  \ref{thm:dringBC}  in the main body of the paper) are a detailed description of the framed potentially crystalline deformation ring $R^{(2,1,0), \tau}_{\rhobar}$ (if it is nonzero) in terms of the notion of \emph{shape} attached to the pair $(\rhobar,\tau)$ (cf. Definition \ref{rhobarshape}).  The \emph{shape} is an element of length $\leq 4$ in the Iwahori-Weyl group of $\GL_3$ and arises from the study of moduli of Kisin modules with descent datum in \S \ref{KM with dd} (inspired by work of \cite{breuil-buzzati, BM11, CDM1, EGS} and further pursued in \cite{CL}); it generalizes the notion of \emph{genre} which is crucial in \cite{breuil-buzzati} in describing tamely Barsotti-Tate deformation rings for $\GL_2$.

\begin{thm} 
\label{thm1Intro}
Let $\rhobar:G_{\Qp} \ra \GL_3(\F)$ be a continuous Galois representation.  Let $\tau$ be a strongly generic tame inertial type. 
Then the framed potentially crystalline deformation ring $R^{(2,1,0), \tau}_{\rhobar}$ with Hodge-Tate weights $(2,1,0)$ has connected generic fiber and its special fiber is as predicted by the geometric Breuil-M\'ezard conjecture.
 
If $R^{(2,1,0), \tau}_{\rhobar}$ is nonzero and the shape of $(\rhobar,\tau)$ has length at least 2, then $R^{(2,1,0), \tau}_{\rhobar}$ has an explicit presentation given in Table $\ref{table:withmon}$. If the shape of $(\rhobar,\tau)$ has length $\leq1$, then the special fiber of $R^{(2,1,0), \tau}_{\rhobar}$ is described in Section $\ref{sec:badcases}$.
\end{thm} 

The first step towards Theorem \ref{thm1Intro} is a detailed study of the moduli space of Kisin modules with descent datum.  The shapes of Kisin modules which arise from reductions of potentially crystalline representations with Hodge-Tate weights $(2,1,0)$ are indexed by  $(2,1,0)$-admissible elements $(\mathrm{Adm}(2,1,0))$ in the Iwahori-Weyl group of $\GL_3$ as defined by Kottwitz and Rapoport (cf. \cite[(9.17)]{PZ}).  For generic $\tau$, the Kisin variety is trivial, and so we can associate a shape to a pair $(\rhobar,\tau)$.

There are 25 elements in $\mathrm{Adm}(2,1,0)$ (cf. Table \ref{Table admissible elements}).  Due to an additional symmetry, we are able to reduce our analysis to nine cases.  The shorter the length of the shape the more complicated the deformation ring is. In seven cases (length $\geq 2$), the deformation ring admits a simple description (see Table \ref{table:withmon}).  The remaining two cases require separate analysis undertaken in \S \ref{sec:badcases}.  Our strategy is as follows:

\begin{enumerate} 
\item Classify all Kisin modules of shape $\widetilde{w} \in \mathrm{Adm}(2,1,0)$ over $\overline{\mathbb{F}}_p$ (Theorem \ref{thm:classification});
\item For $\overline{\fM}$ of shape $\widetilde{w}$, construct the universal deformation space with height conditions (Theorem \ref{univfh}); 
\item Impose monodromy condition on the universal family (\S 5).  
\end{enumerate}
 
Steps (1) and (2) generalize techniques of \cite{breuil-buzzati, CDM1, EGS} used to compute tamely Barsotti-Tate deformation rings for $\GL_2$. Step (2) amounts to constructing local coordinates for the Pappas-Zhu local model for $(\GL_3, \mu = (2,1,0)$, Iwahori level) (cf. \cite{CL}) and requires a more systematic approach to the $p$-adic convergence algorithm employed by \cite{breuil-buzzati, CDM1, LM}. 
 
Step (3) requires a genuinely new method not present in the tamely Barsotti-Tate case where the link between moduli of finite flat groups schemes and Galois representations is stronger.  Kisin \cite{KisinFcrys} characterized when a torsion-free Kisin module $\fM$ over $\Zp$ comes from a crystalline representation in terms of the poles of a monodromy operator $N_{\fM^{\rig}}$ which is naturally defined on the extension $\fM^{\rig}$ of  $\fM$ to the rigid analytic unit ball.  This condition on the poles of the monodromy operator is a subtle analogue of Griffiths transversality in $p$-adic Hodge theory. 
While one cannot compute $N_{\fM^{\rig}}$ completely, it is possible to give an explicit approximation using the genericity condition on $\tau$. The \emph{error term} turns out to be good enough to understand the geometry of the deformation rings. 

\vspace{2mm}

\paragraph{\emph{\textbf{Global applications.}}}

Using Kisin-Taylor-Wiles patching methods, the local information on the Galois deformation spaces leads to new modularity results and the Serre weight conjectures.  To state these results, we fix a global setup (cf. \S \ref{subsec:glsetup}) and remark that the weight part of Serre's conjecture is expected to be independent of the global setup.  Our proofs only use the existence of \emph{patching functors} in the sense of \cite{EGS,GHS} verifying certain axioms  (Definition \ref{minimalpatching}) and so our results should hold in other situations as well.   

Let $F/\Q$ be a CM field with totally real subfield $F^+$.  Assume that $p$ splits completely in $F$.  Let $\rbar:\Gal(\overline{F}/F) \ra \GL_3(\overline{\F})$ be a continuous irreducible representation.   Let $G$ be a unitary group over $F^+$ which is isomorphic to $U(3)$ at each infinite place and split above $p$. Attached to this data, there is a well known notion of modularity for $\rbar$ (cf. Definition \ref{definition modularity}). Roughly speaking, we can find a prime-to-$p$ level $U^p$ in the finite ad\`elic points $G(\A_F^{\infty, p})$ of $G$ and a maximal ideal $\mathfrak{m}_{\rbar}$ in the Hecke algebra $\T$ acting on the space of mod $p$ algebraic automorphic forms $S(U^p,\F)$ with infinite level at $p$, such that $S(U^p,\F)_{\mathfrak{m}_{\rbar}}\neq 0$. 

Rather than specify a minimal weight for which $\rbar$ is modular, it is natural to consider local systems attached to 
irreducible mod $p$-representations of $G(\cO_{F^+_p})$ (also called \emph{Serre weights}) on the locally symmetric spaces of $G$. We let $W(\rbar)$ be the set of modular weights (i.e., the set of Serre weights $V$ for which $\Hom_{G(\cO_{F^+_p})}(V,S(U^p,\F)_{\mathfrak{m}_{\rbar}})\neq 0$). We define what it means for Serre weights to be \emph{reachable} (cf. Definition \ref{definition genericity}, which is an explicit condition on the highest weight) and write $W_{elim}(\rbar)$ for the set of reachable modular weights. (The notion of reachable weight is due to the current weight elimination result when $\rbar$ is irreducible at a place above $p$, cf. \cite{EGH}, Theorem 5.2.5.)

If $\rbar$ is semisimple at each place above $p$, then there is a set of conjectural weights $W^{?}(\rbar)$ defined in \cite{herzig-duke, GHS} which only depends on the restriction of $\rbar$ to the inertia subgroups at the primes above $p$. 

\begin{thm}
\label{thm-intro2}
Let $\rbar:G_F\rightarrow \GL_3(\F)$ be a continuous Galois representation, verifying the Taylor-Wiles conditions $($cf. Definition $\ref{TWconditions})$. Assume that $\rbar|_{G_{F_{v}}}$ is semisimple and $8$-generic for all $v \mid p$ $($cf. Definition $\ref{defn n-gen})$, that $\rbar$ is automorphic of some reachable Serre weight, and that $\rbar$ has split ramification outside $p$. Then 
$$
W_{\mathrm{elim}}(\rbar) = W^{?}(\rbar).
$$
\end{thm}

When $\rbar$ is irreducible at each prime above $p$, this is proven in \cite{EGH} using the technique of weight cycling and without any Taylor-Wiles conditions. The inclusion $W_{elim}(\rbar) \subset W^{?}(\rbar)$ (weight elimination) is proven in \cite{EGH, HLM, MP}.  Recent improvements in weight elimination results show that $W_{elim}(\rbar)$ can be replaced by $W(\rbar)$ and `automorphic of some reachable Serre weight' with just `automorphic', see Remark \ref{rem:WE} for a precise discussion.

If $[F^+:\Q] =d$, there are $9^d$ conjectural weights appearing in $W^{?}(\rbar)$, $6^d$ of which are called \emph{obvious} weights since they are directly related to the Hodge-Tate weights of ``obvious'' crystalline lifts of $(\rbar|_{G_{F_v}})_{v \mid p}$.  The precise relation between Serre weights of $\rbar$ and Hodge-Tate weights of crystalline lifts of $\rhobar$ was first made precise in \cite{gee-annalen} and the obvious weights are shown to be modular in \cite{BLGG} using global methods (namely, modularity lifting techniques) under the assumption that $\rbar$ is modular of a lower alcove weight.   

The remaining weights in $W^{?}(\rbar)$ are more mysterious and are referred to as \emph{shadow weights}. The modularity of the shadow weights lies deeper than that of the obvious weights, in part, because modularity of a shadow weight cannot be detected by modularity lifting alone but requires characteristic $p$ information.  It is at this point that the computation of the monodromy operator appears to play a critical role.  The proof of Theorem \ref{thm-intro2} builds on the Breuil-M\'ezard philosophy introduced in \cite{gee-kisin}. The patching techniques of Gee-Kisin \cite{gee-kisin} and Emerton-Gee \cite{EG} connect the geometry of the local deformation rings to modularity questions. We use geometric information about the local deformation rings, especially the geometry of their special fibers, to prove the modularity of the shadow weights. 

Theorem \ref{thm-intro2} is stated only for $\rbar$ which are semisimple above $p$ because those are the only representations for which there is an explicit conjecture.  Our computations, together with work of \cite{HLM, MP}, suggest a set $W^{?}(\rhobar)$ for non-semisimple $\rhobar$ for which the Theorem should hold. We give one example in Proposition \ref{Prop-non-ss} and will return to this question in future work.  We also give counterexamples in Proposition \ref{geecounter} to Conjecture 4.3.2 (for some non-semisimple $\rhobar$) of \cite{gee-annalen}, which predicts Serre weights in terms of the existence of crystalline lifts.  

The information on the deformation rings (Theorem \ref{thm1Intro}), namely the connectedness of their generic fiber,  lets us deduce new modularity lifting theorems.
\begin{thm}
\label{thm-intro3}
Let $r: G_F\rightarrow \GL_3(\cO)$ be a Galois representation and write $\rbar: G_F\rightarrow \GL_3(\F)$ for its associated residual representation.

Assume that:
\begin{enumerate}
	\item $p$ splits completely in $F^+$;
	\item $r$ is unramified almost everywhere and satisfies $r^c\cong r^{\vee} \epsilon^{-2}$ $($where $c$ denotes the complex conjugation on $F/F^+$$)$;
	\item for all places $w\in \Sigma_p$, the representation $r|_{G_{F_w}}$ is potentially crystalline, with parallel Hodge type $(2,1,0)$ and with strongly generic tame inertial type $\tau_{\Sigma_p^+}=\otimes_{v\in \Sigma_p^+}\tau_v$ $($cf. Definition $\ref{gencond})$;
	\item $\rbar$ verifies the Taylor-Wiles conditions $($cf. Definition $\ref{TWconditions}$, in particular $\rbar$ is absolutely irreducible$)$ and $\rbar$ has split ramification;
	\item $\rbar\cong \rbar_{\imath}(\pi)$ for a RACSDC representation $\pi$ of $\GL_3(\bA_{F})$ with trivial infinitesimal character and such that $\otimes_{v\in \Sigma_p^+}\sigma(\tau_v)$ is a $K$-type for $\otimes_{v\in \Sigma_p^+}\pi_v$.
\end{enumerate}
Then $r$ is automorphic.
\end{thm}

In Theorem \ref{thm-intro3}, we do not assume that $\rhobar$ is semisimple nor do we make any potential diagonalizability assumptions.  We also allow any tame type not just principal series types. Assumption (1) can be relaxed to the condition that $p$ is unramified in $F^+$.  This requires new representation theoretic techniques which will be discussed in a companion paper (\cite{LLLM2}).  We also believe that the genericity assumptions on the type $\tau$ can be weakened with more work.


Our results shed light on other questions in mod $p$ and $p$-adic Langlands as well.  
For example, Breuil \cite{breuil-buzzati} formulates a conjecture, based on calculations of tamely Barsotti-Tate deformation rings for $\GL_2$, on integral lattices in tame types cut out by completed cohomology. When a tame principal series representation $\pi$ of $G(F^+_p)$ appears in the $\mathfrak{m}_{\rbar}$-part (for a globalization of $\rhobar$) of the completed cohomology of an appropriate Shimura curve, the natural integral structure on the cohomology induces an integral structure on the type associated to $\pi$.  Breuil conjectures that this lattice only depends on the local $p$-adic Galois representation attached to $\pi$ by the hypothetical $p$-adic local Langlands correspondence.  A related conjecture of Demb\'el\'e in \cite{breuil-buzzati}
is a ``multiplicity one'' statement for cohomology at Iwahori level.
Both conjectures for $K/\Qp$ unramified were proven by Emerton-Gee-Savitt \cite{EGS} using the Taylor-Wiles method and geometric Breuil-M\'ezard realized by explicit presentations of tamely Barsotti-Tate deformation rings.  
 
The first author proved a generalization of Breuil's conjecture to $\GL_3$ in the setting of \cite{EGH}.  In a companion work (\cite{LLLM2}), we address Breuil's conjecture for $\GL_3$ and $K/\Qp$ unramified using the methods developed in this paper.  We will also extend the global results of this paper to the case of $K/\Qp$ unramified. For ease of exposition, we restrict ourselves to the case of $\GL_3$ throughout this paper, although \S 2-4, some of \S 5, and \S 6 could be extended to $\GL_n$ without serious difficulty.   

\paragraph{\emph{\textbf{Overview of the paper.}}}
We start in Section \S \ref{KM with dd} with the basic formalism of Kisin modules with tame descent data (notion of eigenbasis, genericity and the basic formulas of semilinear algebra). This is further pursued in \S \ref{subsec:classif}, where we obtain a complete classification of Kisin modules with generic descent data in terms of shapes (cf. Definition \ref{defn:shape} and Theorem \ref{thm:classification}).
Section \ref{Kisintangent} compares the moduli of Kisin modules with generic tame descent data and Galois deformation spaces. The genericity assumption guarantees the triviality of the Kisin variety (Theorem \ref{Kisinvariety}) and injectivity on tangent spaces. We conclude the section with the notion of shape and genericity for a mod $p$ Galois representation (cf. Definitions  \ref{rhobarshape}, \ref{defn n-gen}) together with a Galois cohomology argument which shows that, under a mild assumption on $\rhobar$, restriction to Galois deformations of some deeply ramified extension is fully faithful on tangent spaces.

Section \ref{sec:algorithm} and \ref{sec:Monodromy and PCDR} are the technical heart of the paper. 
In \S \ref{sec:algorithm}, we develop an algorithm to construct a ``universal family'' of Kisin modules of finite height, lifting a residual Kisin module of a given shape (Theorem \ref{thm:classification1}). The strategy is a wide generalization of the methods already appearing in \cite{breuil-buzzati, CDM, LM}.
An algorithm is described to construct a \emph{gauge basis} on the universal family in \S \ref{subsec:algorithm} on which we then impose the $p$-adic Hodge type conditions (cf. Table \ref{table:lifts} and Theorem \ref{univfh})

In \S \ref{sec:Monodromy and PCDR}, we endow the rigid analytification of the universal family of Kisin modules with a canonical monodromy operator, which we determine up to an error term which is divisible by a power of $p$ (Theorem \ref{thm:monodr}); by imposing the monodromy to have logarithmic poles (Proposition \ref{Mcond}), we finally obtain explicit equations for the moduli of Kisin modules with monodromy (Table \ref{table3}), and hence for the Galois deformation ring (\S \ref{subsection:PCDR}, Corollary \ref{dring}, and Table \ref{table:withmon}).

Section \ref{sec:BC} extends the results of \S \ref{sec:Monodromy and PCDR} to  other tame types (cf. Theorem \ref{thm:dringBC}). We generalize the formalism of ``base change'' for deformation rings as developed in \cite{EGS} in dimension 2.

The main global applications are discussed in \S \ref{sec:appl}. Via Kisin-Taylor-Wiles patching and the formalism of patched functors (as introduced in \cite{EGS}), we prove the main theorems discussed above. Section \ref{sec:badcases} is devoted to the analysis of the monodromy condition when the shape has length $\leq1$ where the computations become more involved. In the Appendix, we collect tables summarizing our results.

\emph{Acknowledgments:} We would like to thank Matthew Emerton, Toby Gee, and Florian Herzig for many helpful conversations and for comments on an earlier draft of the paper.   We thank the referee for the very careful reading and many suggestions on an earlier version which improved the exposition and clarity of the article.

\subsection{Notation} \label{subsection:Notation}


If $F$ is any field, we write $G_F\defeq \Gal(\overline{F}/F)$ for the absolute Galois group, where $\overline{F}$ is a separable closure of $F$. If $F$ is moreover a $p$-adic field, we write $I_F$ to denote the inertia subgroup of $G_F$.

We fix once and for all an algebraic closure $\overline{\Q}$ of $\Q$. All number fields are considered as subfields of our fixed $\overline{\Q}$. Similarly, if $\ell\in \Q$ is a prime, we fix algebraic closures $\overline{\Q}_\ell$ as well as embeddings $\overline{\Q}\iarrow\overline{\Q}_{\ell}$. All finite extensions of $\Q_{\ell}$ are considered as subfields of $\overline{\Q}_{\ell}$. Moreover, the residue field of $\overline{\Q}_{\ell}$ is denoted by $\overline{\F}_\ell$.

Let $p>3$ be a prime. 
For $f>0$, we let $K$ be the unramified extension of $\Qp$ of degree $f$. We write $k$ for its residue field and let $W = W(k)$. We set $e\defeq p^f-1$ and consider the Eisenstein polynomials $E(u)\defeq u^e+p\in K[u]$ and $P(v)\defeq v+p\in K[v]$ where $v= u^e$. We fix a root $\pi \defeq (-p)^{\frac{1}{e}}\in \overline{K}$, define the extension $L = K(\pi)$ and set $\Delta\defeq \Gal(L/K)$. 
The choice of the root $\pi$ lets us define a character
\begin{align*}
\teich{\omega}_{\pi}:\Delta&\rightarrow W\s\\
g&\mapsto \frac{g(\pi)}{\pi}
\end{align*}
whose associated residual character is denoted by $\omega_{\pi}$. In particular, for $f=1$, $\omega_{\pi}$ is the mod $p$ cyclotomic character, which will be simply denoted by $\omega$. The $p$-adic cyclotomic character will be denoted by $\varepsilon: G_{\Qp}\ra\Zp\s$. 
If $F_w/\Qp$ is a finite extension and $W_{F_w}\leq G_{F_w}$ denotes the Weil group we normalize Artin's reciprocity map $\mathrm{Art}_{F_w}: F_w\s\ra W_{F_w}^{\mathrm{ab}}$ in such a way that uniformizers are sent to geometric Frobenius elements.

Let $E$ be a finite extension of $\Qp$. We write $\cO$ for its ring of integers, fix an uniformizer $\varpi\in \cO$ and let $\mathfrak{m}_E=(\varpi)$. We write $\F\defeq \cO/\mathfrak{m}_E$ for its residue field. We will always assume that $E$ is sufficiently large, i.e. that any embedding $\sigma: K\iarrow \overline{\Q}_p$ factors through $E\subset \overline{\Q}_p$. In particular, any embedding $\sigma: k\iarrow \overline{\F}_p$ factors through $\F$. 

Let $\rho: G_K\rightarrow \GL_n(E)$ be a $p$-adic, de Rham Galois representation. For $\sigma: K\iarrow \overline{\Q}_p$, we define $\mathrm{HT}_\sigma(\rho)$ to be the multiset of $\sigma$-labeled Hodge-Tate weights of $\rho$, i.e. the set of integers $-i$ such that $\dim_E\big(\rho\otimes_{\sigma,K}\C_p(i)\big)^{G_K}\neq 0$ (with the usual notation for Tate twists).  
In particular, we have $\mathrm{HT}_\sigma(\varepsilon)=\{1\}$ for any $\sigma$. 
We define the \emph{Hodge type} of $\rho$ to be the multiset 
$\big(\mathrm{HT}_{\sigma}(\rho)\big)_{\sigma\in S_K}\in \big(\Z^n\big)^{S_K}$ where $S_K\defeq \big\{\sigma\mid\ \sigma: K\iarrow \Qpbar\big\}$ and the \emph{inertial type} of $\rho$ as the isomorphism class of $\mathrm{WD}(\rho)|_{I_K}$, where $\mathrm{WD}(\rho)$ is the Weil-Deligne representation attached to $\rho$ as in \cite{CDT}, Appendix B.1 (in particular, $\rho\mapsto\mathrm{WD}(\rho)$ is \emph{covariant}).
Recall that an inertial type is a morphism $\tau: I_K\ra \GL_n(\cO)$ with open kernel and which extends to the Weil group $W_K$ of $G_K$.

We fix an embedding $\sigma_0:K \subset E$, which induces maps $W \iarrow \cO$ and $k \iarrow \F$; by an abuse of notation, we denote all of these by $\sigma_0$. We let $\phz$ denote the $p$-th power Frobenius on $k$ and set $\sigma_j \defeq \sigma_0 \circ \phz^{-j}$. The choice of the embedding $\sigma_0$ gives a fundamental character $\omega_f := \sigma_0 \circ \widetilde{\omega}_{\pi}:I_K \ra \cO^{\times}$ of niveau $f$.   


Let $S_3$ denote the symmetric group on $\{1,2,3\}$. We fix an injection $S_3\hookrightarrow \GL_3(\Z)$ sending $s$ to the permutation matrix whose $(k, m)$-entry is $\delta_{k,s(m)}$ and $\delta_{k,s(m)}\in\{0,1\}$ is the Kronecker $\delta$ specialized at $\{k,s(m)\}$. We will abuse notation and simply use $s$ to denote the corresponding permutation matrix. 
Finally for $m\geq 0$ and a collection $(B_j)_{j=0,\dots,m}$ of square matrices of the same size, we write $\prod_{j=0}^m B_j=B_0\cdot B_1\dots B_{m}$.

\section{Kisin modules modulo $p$}

\subsection{Kisin modules with descent datum}
\label{KM with dd}

Let $\tau = \eta_1 \oplus \eta_2 \oplus \eta_3$ be an $\cO\s$-valued inertial type consisting of pairwise distinct characters. 

Let $\mathbf{a}_1 = (a_{1,j})_{j }$, $\mathbf{a}_2 = (a_{2, j})_{j }$, and $\mathbf{a}_3 = (a_{3, j})_{j }$ where $0\leq j\leq f-1$ and $0 \leq a_{k, j} \leq p-1$.    For any $0\leq j\leq f-1$, define
$$
\mathbf{a}^{(j)}_k = \sum_{i =0}^{f-1} a_{k, -j + i} p^i.
$$
For $1 \leq k \leq 3$, we can write 
$$
\eta_k = (\omega_f)^{-\mathbf{a}_k^{(0)}}
$$
for a unique choice of $\mathbf{a}_k$. We say that $(\mathbf{a}_1,\mathbf{a}_2,\mathbf{a}_3)$ is \emph{associated to $\tau$}.
We will need the following genericity assumption throughout the paper. 

\begin{defn} \label{gencond}  
Let $n\in \N$. We say that the triple $(\mathbf{a}_1,\mathbf{a}_2,\mathbf{a}_3)$ is \emph{$n$-generic} if 
\begin{eqnarray}
\label{labelgencond}
n \leq |a_{1, j} - a_{2, j}|,\,|a_{2, j} - a_{3, j}|,\,|a_{1, j} - a_{3, j}| \leq p-1-n
\end{eqnarray}
for all $j$. We say that an inertial type $\tau=\eta_1 \oplus \eta_2 \oplus \eta_3$ is \emph{$n$-generic} if the associated triple $(\mathbf{a}_1,\mathbf{a}_2,\mathbf{a}_3)$ is generic. 
We say that $\tau$ is \emph{generic} (resp. \emph{weakly generic}, resp. \emph{strongly generic}) if $\tau$ is $5$-generic (resp. $3$-generic, resp. $10$-generic).
\end{defn}
 
Let $R$ be an $\cO$-algebra. Any $W \otimes_{\Zp} R$-module $M$ decomposes as direct sum of $R$-modules $M = \bigoplus_{j=0}^{f-1} M^{(j)}$ where $M^{(j)}$ is the submodule such that $(x \otimes 1) m = (1 \otimes \sigma_j(x)) m$ for all $m \in M^{(j)}$ and $x \in W$. 

For any $g \in \Delta$ and any $\cO$-algebra $R$, we let $\widehat{g}$ be the $W \otimes_{\Zp} R$-linear automorphism of $(W \otimes_{\Zp} R)[\![u]\!]$ given by $u \mapsto  (\omega_\pi(g) \otimes 1)u$. 

\begin{defn} Let $\fM_R$ be an $(W \otimes_{\Zp} R)[\![u]\!]$-module. A \emph{semilinear action} of $\Delta$ on $\fM_R$ is collection of $\widehat{g}$-semilinear additive bijections $\widehat{g}:\fM_R \ra \fM_R$  for each $g \in \Delta$ such that
$$
\widehat{g} \circ \widehat{h} = \widehat{gh}
$$
for all $g, h \in \Delta$. 
\end{defn}    

Recall that for any $\cO$-algebra $R$, we define the Frobenius $\phz: (W \otimes_{\Zp} R)[\![u]\!] \ra (W \otimes_{\Zp} R)[\![u]\!]$ to be trivial on $R$, the Frobenius on $W$, and with $\phz(u) = u^p$.  Note that $\widehat{g}\circ \phz=\phz\circ \widehat{g}$.

\begin{defn} Let $R$ be any $\cO$-algebra.  A \emph{Kisin module} with height in $[0, h]$ over $R$ is a finitely generated projective $(W \otimes R)[\![u]\!]$-module $\fM_R$ together with Frobenius $\phi_{\fM_R}:\phz^*(\fM_R) \ra \fM_R$ such that the cokernel is killed by $E(u)^h$. 
Here and throughout the paper the notation $\phz^*(\fM_R)$ stands for the base change of $\fM_R$ along  $\phz: (W \otimes_{\Zp} R)[\![u]\!] \ra (W \otimes_{\Zp} R)[\![u]\!]$.
\end{defn} 

\begin{defn}  A \emph{Kisin module with descent datum} over $R$ is a Kisin module $(\fM_R, \phi_{\fM_R})$ together with a semilinear action of $\Delta$ given by $\{ \widehat{g} \}_{g \in \Delta}$ which commutes with $\phi_{\fM_R}$, i.e., for all $g \in \Delta$,
$$
\widehat{g} \circ \phi_{\fM_R} = \phi_{\fM_R} \circ \phz^*(\widehat{g}).
$$
Let $\fM_R \cong \bigoplus_{j = 0}^{f-1} \fM^{(j)}_R$.  We say that the descent datum is of \emph{type} $\tau$ if the linear representation of $\Delta$ on the $R$-module satisfies $\fM_R^{(j)}/ u \fM_R^{(j)} \cong \tau \otimes_{\cO} R$ for each $0\leq j\leq f-1$.  
\end{defn}

For any $\cO$-algebra $R$, let $Y^{[0,h], \tau} (R)$ be the category of Kisin modules over $R$ with height in $[0,h]$, rank 3, and descent datum of type $\tau$. For a given $N\in\N$, it is shown in \cite{CL} that $Y^{[0,h], \tau}  \mod (\mathfrak{m}_E)^N$ is represented by an Artin stack of finite type over $\cO/(\mathfrak{m}_E)^N$. The aim of this section is to classify the $\overline{\F}$-points of  $Y^{[0,2], \tau}$ which are reductions of Kisin modules with ``Hodge-Tate'' weights $(2,1,0)$.  

\begin{defn} Let $v \defeq u^e$ and $\fM_R \in Y^{[0,2], \tau} (R)$. 
For $k\in\{1,2,3\}$, define $\fM_{R, k}$ to be the $(W \otimes R)[\![v]\!]$-submodule of $\fM_R$ on which $\Delta$ acts by $\eta_k$, i.e., $\fM_{R, k}\defeq \big(\fM_R\big)^{\Delta=\eta_k}$.   
Similarly, we  define $^{\phz} \fM_{R, k}$ to be the $(W \otimes R)[\![v]\!]$-submodule of $\phz^*(\fM_R)$ on which $\Delta$ acts by $\eta_k$, i.e., $^{\phz} \fM_{R, k}\defeq \big(\phz^*(\fM_R)\big)^{\Delta=\eta_k}$.

By considering the decomposition $\fM_R\cong \bigoplus_{j=0}^{f-1}\fM_R^{(j)}$, we write
$\fM^{(j+1)}_{R, k}$ for the $R[\![v]\!]$-submodules of $\fM^{(j+1)}_R$ on which $\Delta$ acts by $\eta_k$ and
we write $^{\phz} \fM^{(j)}_{R, k}$ for the $R[\![v]\!]$-submodules of $(\phz^*(\fM_R))^{(j+1)} = \phz^*(\fM_R^{(j)})$ on which $\Delta$ acts by $\eta_k$ (with the usual convention that $j+1\defeq 0$ if $j=f-1$).
\end{defn}

While we have made a choice of global ordering $\eta_1, \eta_2, \eta_3$, it will be important for uniform statements to order things (possibly) differently at each embedding $\sigma_j:K \ra E$.  We introduce this local ordering now. 

\begin{defn} \label{orientation} 
Let $(\mathbf{a}_1,\mathbf{a}_2,\mathbf{a}_3)$ be a triple as in Definition \ref{gencond}.
An \emph{orientation} of $(\mathbf{a}_1,\mathbf{a}_2,\mathbf{a}_3)$ is an $f$-tuple $(s_{j})_j \in S_3^f$ such that 
$$
\bf{a}_{s_{j}(1)}^{(j)} \geq  \bf{a}_{s_{j}(2)}^{(j)} \geq  \bf{a}_{s_{j}(3)}^{(j)}.
$$
If $\tau$ is an inertial type as above, an \emph{orientation of $\tau$} is defined to be an orientation of the triple $(\mathbf{a}_1,\mathbf{a}_2,\mathbf{a}_3)$ associated to it. In this case, we say that $s_j$ is an \emph{orientation at $j$} of $\tau$.
\end{defn}

Under the weak genericity condition (\ref{labelgencond}), there exists a unique orientation on $\tau$ and the orientation at $j$  is determined by the values of
$a_{1, f-j -1}$, $a_{2, f-j -1}$, $a_{3, f-j -1}$ which are the dominant terms of $\bf{a}_1^{(j)}$, $\bf{a}_2^{(j)}$, $\bf{a}_3^{(j)}$ respectively.  In particular, we have 
\begin{equation} \label{orientlead}
a_{s_j(1), f - j -1} > a_{s_j(2), f - j -1} > a_{s_j(3), f - j -1}.   
\end{equation}

For any $\fM \in Y^{[0,h], \tau}(R)$, we have the following commutative diagram relating the different isotypic components:
\begin{equation} \label{bigdiagram}
\xymatrix@=3pc{
^{\phz} \fM^{(j-1)}_{s_{j}(3)} \ar[rr]^{u^{ e - (\mathbf{a}_{s_{j}(1)}^{(j)} - \mathbf{a}_{s_{j}(3)}^{(j)})}} \ar[d]^{\phi^{(j-1)}_{\fM, s_{j}(3)}} & & ^{\phz} \fM^{(j-1)}_{s_{j}(1)} \ar[rr]^{u^{\mathbf{a}_{s_{j}(1)}^{(j)} - \mathbf{a}_{s_{j}(2)}^{(j)}} } \ar[d]^{\phi^{(j-1)}_{\fM, s_{j}(1)}} && ^{\phz} \fM^{(j-1)}_{s_{j}(2)} \ar[d]^{\phi^{(j-1)}_{\fM, s_{j}(2)}} \ar[rr]^{u^{\mathbf{a}_{s_{j}(2)}^{(j)} - \mathbf{a}_{s_{j}(3)}^{(j)}}} & & ^{\phz} \fM^{(j-1)}_{s_{j}(3)}  \ar[d]^{\phi^{(j-1)}_{\fM, s_{j}(3)}} \\
\fM^{(j)}_{s_{j}(3)} \ar[rr]^{u^{ e - (\mathbf{a}_{s_{j}(1)}^{(j)} - \mathbf{a}_{s_{j}(3)}^{(j)})}} & & \fM^{(j)}_{s_{j}(1)} \ar[rr]^{u^{\mathbf{a}_{s_{j}(1)}^{(j)} - \mathbf{a}_{s_{j}(2)}^{(j)}}} & & \fM^{(j)}_{s_{j}(2)} \ar[rr]^{u^{\mathbf{a}_{s_{j}(2)}^{(j)} - \mathbf{a}_{s_{j}(3)}^{(j)}}} & & \fM^{(j)}_{s_{j}(3)} \\
}
\end{equation}
where the composition along each row is multiplication by $u^{e}$ and the vertical arrows are induced by $\phi_{\fM}$. All the maps in the diagram are injective (again, with the standard convention that $j-1\defeq f-1$ if $j=0$).  In particular, any one of the three maps $\phi^{(j-1)}_{\fM, 1}, \phi^{(j-1)}_{\fM, 2}, \phi^{(j-1)}_{\fM, 3}$ determines the other two.   We choose to focus on $\phi^{(j - 1)}_{\fM, s_{j}(3)}$.   We discuss in more detail at the end of the section how the Frobenii $\phi^{(j - 1)}_{\fM, s_{j}(k)}$, for $1\leq k\leq 3$, are related (Proposition \ref{cyclicsymmetry}).

\begin{rmk} The submodule $^{\phz} \fM_{k}$ of $\phz^*(\fM)$ is NOT the same as the Frobenius pullback of $\fM_{k}$.  In particular, $\phi_{\fM, k}$ does not define a semilinear endomorphism of $\fM_{k}$.   It is merely a linear map from $^{\phz} \fM_{k} \ra \fM_{k}$.   This fact is reflected again in the change of basis formula (Proposition \ref{changeofbasis}). 
\end{rmk} 

We want to consider $\big(\phi^{(j)}_{\fM, s_{j+1}(3)}\big)_j$ as a collection of matrices with respect to a choice of bases. We will refine the basis further in the next section.   

\begin{defn} Let $\fM \in Y^{[0,2], \tau} (R)$.  An \emph{eigenbasis} $\beta :=\big\{\beta^{(j)}\big\}_j$ for $\fM$ is a collection of bases $\beta^{(j)}= \Big( f_{1}^{(j)}, f_{2}^{(j)}, f_{3}^{(j)} \Big)$ of each $\fM^{(j)}$ 
such that $f_{k}^{(j)} \in \fM_{k}^{(j)}$ for each $k\in\{1,2,3\}$.
\end{defn}


\begin{lemma} \label{eigenbases} If $\beta = \Big\{ \Big(f_1^{(j)}, f_{2}^{(j)}, f_{3}^{(j)}\Big) \Big \}_{j}$ is an eigenbasis for $\fM$, then for any $0\leq j\leq f-1$,
$$
\beta^{(j)}_{s_{j}(3)} := \Big( u^{\mathbf{a}_{s_{j}(1)}^{(j)} - \mathbf{a}_{s_{j}(3)}^{(j)}  } f^{(j)}_{s_{j}(1)}, u^{\mathbf{a}_{s_{j}(2)}^{(j)} - \mathbf{a}_{s_{j}(3)}^{(j)}} f_{s_{j}(2)}^{(j )},  f_{s_{j}(3)}^{(j)} \Big),
$$
is a basis for  $\fM_{s_{j}(3)}^{(j)}$.  Similarly, 
$$
^{\phz} \beta^{(j-1)}_{s_{j}(3)} := \Big(u^{\mathbf{a}_{s_{j}(1)}^{(j)} - \mathbf{a}_{s_{j}(3)}^{(j)}  } \otimes f^{(j-1)}_{s_{j}(1)}, u^{\mathbf{a}_{s_{j}(2)}^{(j)} - \mathbf{a}_{s_{j}(3)}^{(j)}} \otimes f_{s_{j}(2)}^{(j-1)}, 1 \otimes f_{s_{j}(3)}^{(j-1)} \Big)
$$
is a basis for $^{\phz} \fM_{s_{j}(3)}^{(j-1)}$.  
\end{lemma}   

\begin{rmk} We always order an eigenbasis $\beta^{(j)}$ with respect to the ordering on the characters $\eta_1, \eta_2, \eta_3$.  On the other hand, when we work with the isotypic pieces $\fM_{s_{j}(3)}^{(j)}$ we order our bases using the orientation $(s_j)$ of $\tau$.  It will be important to keep track of this difference.     
\end{rmk} 

\begin{defn} \label{0matrix} Given an eigenbasis $\beta$ for $\fM$, the \emph{matrix $C^{(j)}$ of $\phi_{\fM}^{(j)}$ with respect to $\beta^{(j)}$} is defined to be the matrix $C^{(j)}$ such that
\[
\phi_{\fM}^{(j)}\left(\phz^*\big(\beta^{(j)}\big)\right)  =\beta^{(j+1)}\, C^{(j)}.
\]
The \emph{matrix $A^{(j)}$ of $\phi_{\fM, s_{j+1}(3)}^{(j)}$ with respect to $\beta^{(j)}$} is defined to be the matrix $A^{(j)}$ such that  
\[
\phi_{\fM, s_{j+1}(3)}^{(j)}\left({}^{\phz} \beta^{(j)}_{s_{j+1}(3)}\right)=\beta^{(j+1)}_{s_{j+1}(3)}\, A^{(j)}.
\]
\end{defn} 

It is customary to write $C^{(j)}=\Mat_{\beta}\big(\phi_{\fM}^{(j)}\big)$ and $A^{(j)}=\Mat_{\beta}\big(\phi_{\fM, s_{j+1}(3)}^{(j)}\big)$ for short.  To understand how $C^{(j)}$ and $A^{(j)}$ relate to each other we define the following conjugation action by diagonal matrices:
\begin{defn} \label{adconj}  For any $b_1, b_2, b_3 \in \Z$ and for any $M \in \Mat_3(R(\!(u)\!))$,  we define
$$
\Ad(u^{b_1},u^{b_2},u^{b_3})(M) := \begin{pmatrix} u^{b_1} & 0 & 0 \\ 0 &  u^{b_2} & 0 \\ 0 & 0 &  u^{b_3} \end{pmatrix} M  
\begin{pmatrix} u^{-b_1} & 0 & 0 \\ 0 &  u^{-b_2} & 0 \\ 
0 & 0 &  u^{-b_3} \end{pmatrix}.
$$
For any $0\leq j\leq f-1$ and any $M \in \Mat_3(R(\!(u)\!))$, we define conjugation with orientation by 
$$
\Ad_{s_{j}}(u^{\bf{a}_1},u^{\bf{a}_2},u^{\bf{a}_3})(M) := s_j \Big(\Ad \Big(u^{\bf{a}^{(j )}_{s_{j}(1)}}, u^{\bf{a}^{(j)}_{s_{j}(2)}}, u^{\bf{a}^{(j )}_{s_{j}(3)}} \Big)(M)\Big) s_j^{-1}
$$
and  
$$
\Ad^{-1}_{s_{j}}(u^{\bf{a}_1},u^{\bf{a}_2},u^{\bf{a}_3})(M) := \Ad \Big(u^{-\bf{a}^{(j )}_{s_{j}(1)}}, u^{-\bf{a}^{(j)}_{s_{j}(2)}}, u^{-\bf{a}^{(j )}_{s_{j}(3)}} \Big)( s_{j}^{-1} M s_j).
$$
\end{defn}

In particular, if $M$ has  the form $C^{(j)}\in \GL_3(R(\!(u)\!))$ as in Definition \ref{0matrix} then $\Ad^{-1}_{s_{j+1}}(u^{\bf{a}_1},u^{\bf{a}_2},u^{\bf{a}_3})(M)$ is a matrix in $\GL_3(R(\!(v)\!))$ (and the above conjugation can be thought as ``removing the descent datum'').
For $\fM \in Y^{[0,h], \tau}(R)$, the following Proposition relates the matrix for $\phi_{\fM}^{(j)}$ to the matrix for $\phi^{(j)}_{\fM, s_{j+1}(3)}$.   
 \begin{prop} \label{addingdd}  Let $\beta$ be an eigenbasis for $\fM$ and $(s_j)$ be the orientation of $\tau$.
Let $A^{(j)}= \Mat_{\beta}\big(\phi_{\fM, s_{j+1}(3)}^{(j)}\big)$ be as in Definition $\ref{0matrix}$.  Then 
\begin{eqnarray}
\label{change isotypical}
C^{(j)}=\Mat_\beta\big(\phi_{\fM}^{(j)}\big)= \Ad_{s_{j+1}}\big(u^{\bf{a}_1}, u^{\bf{a}_2}, u^{\bf{a}_3}\big) \big(A^{(j)}\big). 
\end{eqnarray}
 \end{prop} 
 \begin{proof} This is straightforward from Definition \ref{0matrix} and Lemma \ref{eigenbases} noting that:
\[
\beta^{(j+1)}_{s_{j+1}(3)}= \beta^{(j+1)}{s_{j+1}}\begin{pmatrix} u^{a^{(j+1)}_{s_{j+1}(1)}} & 0 & 0 \\ 0 & u^{a^{(j+1)}_{s_{j+1}(2)}} & 0 \\ 0 & 0 &  u^{a^{(j+1)}_{s_{j+1}(3)}} \end{pmatrix}u^{-a^{(j+1)}_{s_{j+1}(3)}}.
\]
 \end{proof}

For any $\cO$-algebra $R$,  define
\begin{itemize} 
\item $L \Gl_3(R) := \Gl_3(R(\!(v)\!))$
\item $L^+ \Gl_3(R) := \Gl_3(R[\![v]\!])$
\item $\Iw(R) := \{ M \in L^+ \Gl_3(R) \mid M \mod v \text{ is upper triangular } \}$
\item $\Iw_1(R) := \{ M \in L^+ \Gl_3(R) \mid M \mod v \text{ is upper triangular unipotent } \}$ 
\item $\cD_3(R) : = \{ M \in L^+ \Gl_3(R) \mid M \mod v^3 \text{ is diagonal } \}$.
\end{itemize}

 \begin{lemma} \label{integrality} Let $I \in  \Mat_3(R(\!(v)\!))$.  For any integers $b_1, b_2, b_3$ with $e > b_1 - b_3 > b_2 - b_3 > 0$, consider 
 $$
 D = \Ad\big(u^{b_1}, u^{b_2}, u^{b_3}\big)(I).
 $$     
Then $D \in \Mat_3(R[\![u]\!])$ if and only if $I$ is in $\Mat_3(R[\![v]\!])$ and is upper triangular modulo $v$.  
%
\end{lemma} 
\begin{proof}  The proof is a straightforward computation.  
\end{proof}

We will now describe the effect of change of basis for the eigenbasis coordinates.  Recall that $(s_j) \in S_3^f$ is the orientation of $\tau$ (Definition \ref{orientation}), and we associate to $s_j$ the corresponding permutation matrix in $\GL_3(\cO)$ as described in  \S \ref{subsection:Notation}.  

\begin{prop} \label{changeofbasis}
Let $R$ be an $\cO$-algebra.  Let $\fM \in Y^{[0,2], \tau}(R)$ together with two eigenbases $\beta_1^{(j)}:=  \Big( f^{(j)}_{1},  f^{(j)}_{2}, f^{(j)}_{3} \Big) $ and $\beta_2^{(j)}:= \Big ( f'^{(j)}_{1}, f'^{(j)}_{2}, f'^{(j)}_{ 3} \Big)$ related by
$$
\Big (f'^{(j)}_{1}, f'^{(j)}_{2}, f'^{(j)}_{3} \Big) D^{(j)} = \Big (f^{(j)}_{1}, f^{(j)}_{2}, f^{(j)}_{3} \Big)
$$
with $D^{(j)} \in \GL_3(R[\![u]\!])$. 
Let us write $A^{(j)}_1 \defeq \Mat_{\beta_1}\big(\phi^{(j)}_{\fM, s_{j+1}(3)}\big)$ and  $A^{(j)}_2 \defeq \Mat_{\beta_2}\big(\phi^{(j)}_{\fM, s_{j+1}(3)}\big)$ as in Definition \ref{0matrix}.
Then 
\begin{equation}
\label{cob1} 
A^{(j)}_2= I^{(j+1)}A^{(j)}_1 \big(s_{j+1}^{-1} s_j\big(I^{(j), \phz}\big) s_j^{-1} s_{j+1}\big) 
\end{equation}
where, for all $0\leq j\leq f-1$, we have $I^{(j)} \defeq \Ad^{-1}_{s_{j}}(u^{\mathbf{a}_1}, u^{\mathbf{a}_2}, u^{\mathbf{a}_3})(D^{(j)}) \in \Iw(R)$ and $$I^{(j), \phz}\defeq  \Ad(v^{a_{s_{j}(1), f-j - 1}}, v^{a_{s_{j}(2), f-j - 1}}, v^{a_{s_{j}(3), f-j -1}}) (\phz(I^{(j)})^{-1}).$$ 
\end{prop}
\begin{proof}
The proof is a direct computation using Proposition \ref{addingdd}. More precisely, let us write $C_i^{(j)}\defeq \Mat_{\beta_i}(\phi^{(j)}_{\fM})$ for $i\in \{1,2\}$ as in Definition \ref{0matrix}.
We have
\begin{equation}
\label{sigma-conjugation}
C_2^{(j)}=D^{(j+1)}C_1^{(j)}\phz\big(D^{(j)}\big)^{-1}.
\end{equation}

Since $D^{(j)}$ respects the descent datum, $I^{(j)} := \Ad^{-1}_{s_j}(u^{\mathbf{a}_1}, u^{\mathbf{a}_2}, u^{\mathbf{a}_3})(D^{(j)})$ is in $\GL_3(R(\!(v)\!))$, hence in $\Iw(R)$ by Lemma \ref{integrality}.  Using (\ref{sigma-conjugation}) and Proposition \ref{addingdd}, one  obtains:
\begin{align*}
 \Ad_{s_{j+1}}(u^{\mathbf{a}_1}, u^{\mathbf{a}_2}, u^{\mathbf{a}_3})\big(A_2^{(j)}\big) &=  D^{(j+1)} C^{(j)}_1 \phz\big(D^{(j)}\big)^{-1}  \\
 &=  D^{(j+1)}  \Big(\Ad_{s_{j+1}}(u^{\mathbf{a}_1}, u^{\mathbf{a}_2}, u^{\mathbf{a}_3})\big(A^{(j)}_1\big)\Big)  \phz\big(D^{(j)}\big)^{-1}.
\end{align*} 
Conjugating on both sides, we further deduce that
\begin{eqnarray} \label{xx1}
A_2^{(j)} = I^{(j+1)} A^{(j)}_1  \Ad^{-1}_{s_{j+1}}(u^{\mathbf{a}_1}, u^{\mathbf{a}_2}, u^{\mathbf{a}_3})\Big(\phz \big( D^{(j)} \big)^{-1}\Big).
\end{eqnarray} 
We now study the final term of the right-hand side of (\ref{xx1}).  Let $s_{j+1, j} := s_{j+1}^{-1} s_j$. We have 
\begin{align}
 \label{xx2}
&\Ad^{-1}_{s_{j+1}}(u^{\mathbf{a}_1}, u^{\mathbf{a}_2}, u^{\mathbf{a}_3}) \left(\phz \left( D^{(j)} \right)^{-1}\right) \\
=&\Ad \Big(u^{-\mathbf{a}^{(j+1)}_{s_{j+1}(1)}}, u^{-\mathbf{a}^{(j+1)}_{s_{j+1}(2)}}, u^{-\mathbf{a}^{(j+1)}_{s_{j+1}(3)}} \Big)  \Big( s_{j+1, j} \Ad\Big(u^{p \mathbf{a}^{(j)}_{s_{j}(1)}}, u^{p\mathbf{a}^{(j)}_{s_{j}(2)}}, u^{p\mathbf{a}^{(j)}_{s_{j}(3)}}\Big)\big(\phz(I^{(j)})^{-1}\big) s_{j+1, j}^{-1}  \Big) \nonumber \\
=&s_{j+1, j} \Ad\big(v^{a_{s_{j}(1), f-j - 1}}, v^{a_{s_{j}(2), f-j - 1}}, v^{a_{s_{j}(3), f-j -1}}\big) (\phz(I^{(j)})^{-1}) s_{j+1, j}^{-1}. \nonumber
\end{align}
where the last step follows from $p \bf{a}^{(j)}_k - \bf{a}^{(j+1)}_k = a_{k, f- j-1} e$ and 
$$
s_{j+1, j}^{-1} \begin{pmatrix} u^{-\bf{a}^{(j+1 )}_{s_{j+1}(1)}} & 0 & 0 \\ 0 &  u^{-\bf{a}^{(j+1)}_{s_{j+1}(2)}} & 0 \\ 0 & 0 &  u^{-\bf{a}^{(j+1 )}_{s_{j+1}(3)}} \end{pmatrix} s_{j+1, j} = \begin{pmatrix} u^{-\bf{a}^{(j+1 )}_{s_{j}(1)}} & 0 & 0 \\ 0 &  u^{-\bf{a}^{(j+1)}_{s_{j}(2)}} & 0 \\ 0 & 0 &  u^{ -\bf{a}^{(j +1 )}_{s_{j}(3)}} \end{pmatrix}.
$$
The conclusion follows by combining (\ref{xx1}) and (\ref{xx2}).
\end{proof}

\begin{prop} \label{divisibility} Assume that $\tau$ is weakly generic $($Definition \ref{gencond}$)$. Let $I^{(j)} \in \Iw(R)$ be as in Proposition $\ref{changeofbasis}.$ Then
$$
I^{(j), \phz} =  \Ad\big(v^{a_{s_{j}(1), f-j - 1}}, v^{a_{s_{j}(2), f-j - 1}}, v^{a_{s_{j}(3), f-j -1}}\big) (\phz(I^{(j)})^{-1}) \in \cD_3(R).
$$ 
\end{prop} 
\begin{proof} By the weak genericity assumptions and choice of orientation, $a_{s_{j}(1), f-j - 1} - a_{s_{j}(2), f-j - 1} \geq 3$ and 
$$
p-4 \geq a_{s_{j}(1), f-j - 1} - a_{s_{j}(3), f-j - 1} > a_{s_{j}(2), f-j - 1} - a_{s_{j}(3), f-j - 1} \geq 3.
$$
Since $\left(I^{(j)}\right)^{-1} \in \cI(R)$, the entries of $\phz \left(I^{(j)} \right)^{-1}$ below the diagonal are divisible by $v^p$.  A direct computation then shows that 
$$
\Ad \big(v^{a_{s_{j}(1), f-j - 1}}, v^{a_{s_{j}(2), f-j - 1}}, v^{a_{s_{j}(3), f-j -1}} \big) \big(\phz (I^{(j)})^{-1} \big) \in \cD_3(R).
$$
\end{proof}

\subsection{Classification over $\F$}
\label{subsec:classif}

We keep the notations of the previous section except we now work over $\F$ as opposed to $\cO$.  In this subsection, we assume that $\tau$ is \emph{weakly generic}. We are now ready to define the \emph{shape} (or \emph{genre} in French) of a Kisin module.   Let $T$ be the diagonal torus of $\GL_3$ and let $N_{\GL_3}(T)$ denote the normalizer of $T$.   The (extended) affine Weyl group of $\GL_3$ is given by 
$$
\widetilde{W} := N_{\GL_3}(T)(\F(\!(v)\!))/T(\F[\![v]\!]).   
$$
Recall that $\widetilde{W}$ sits in an exact sequence
$$
0 \ra X_*(T) \ra \widetilde{W} \ra S_3 \ra 0
$$
where $S_3$ is the ordinary Weyl group of $\GL_3$ and $X_*(T) \cong \Z^3$ are the cocharacters of $T$. 
If $\lambda\in X_*(T)$ is a cocharacter, it is customary to write $t_\lambda$ to denote the corresponding element (a \emph{translation}) in $\widetilde{W}$.

For any finite extension $\F'$ of $\F$, Bruhat-Tits theory gives the following double coset decomposition
\begin{equation} \label{BTdecomp}
L \GL_3 (\F') = \bigcup_{\widetilde{w} \in \widetilde{W}} \Iw(\F') \widetilde{w} \Iw(\F').
\end{equation}

\begin{defn}
\label{defn:shape}  Let $\bf{w} = (\widetilde{w}_0, \widetilde{w}_1, \ldots, \widetilde{w}_{f-1}) \in \widetilde{W}^f$.   A Kisin module $\fM \in Y^{[0,h], \tau}(\F')$ has \emph{shape} $\bf{w}$ if for any eigenbasis $\beta$, the matrices $ \big(A^{(j)} \big)_j = \big(\Mat_{\beta}\big(\phi_{\fM, s_{j+1}(3)}^{(j)} \big) \big)_j$ have the property that $A^{(j)} \in \Iw(\F') \widetilde{w}_j \Iw(\F')$. 
\end{defn} 

It follows immediately from Propositions \ref{changeofbasis} and \ref{divisibility} that the shape of a Kisin module is well-defined, i.e., it does not depend on the choice of the eigenbasis.  The motivation for this definition comes from the theory of local models and a corresponding stratification there.  We now give an overview of this connection which is described in detail in the joint work \cite{CL} of the third author.   We turn our attention to the study of Kisin modules with ``parallel'' weight $(2,1,0)$.  Precisely, let $\mu = (\mu_j)$ with $\mu_j = (2,1,0)$ for all $j$ considered as a geometric cocharacter of $\Res_{K/\Qp} \Gl_3$. Then, \cite{CL} constructs a closed substack $Y^{\mu, \tau} \subset Y^{[0,2], \tau}$ together with a ``local model diagram''  
$$
\xymatrix{
& \widetilde{Y}^{\mu, \tau} \ar[dl]_{\pi} \ar[dr]^{\Psi} & \\
Y^{\mu, \tau} & & M(\mu), \\
}
$$
where $M(\mu)$ is the Pappas-Zhu local model for $\Res_{K/\Qp} \GL_3$ with Iwahori level structure and cocharacter $\mu$ (cf. \cite[\S5]{CL} in particular Proposition 5.2 and Theorem 5.3).  Both $\pi$ and $\Psi$ are smooth maps.  For any $\bf{w} \in \widetilde{W}^f$,  define $\overline{Y}^{\mu, \tau}_{\bf{w}}(\overline{\F}) \subset \overline{Y}^{\mu, \tau}(\overline{\F})$ to be the set of points with shape $\bf{w}$. 

\begin{prop} Let $\overline{M}(\mu)$ and $\overline{Y}^{\mu, \tau}$ denote the special fibers of $M(\mu)$ and $Y^{\mu, \tau}$.  Then $\overline{M}(\mu)$ has a stratification by locally closed affine Schubert varieties $S_{\bf{w}}^0$ indexed by elements $\bf{w} \in \widetilde{W}^f$.  Furthermore, the set of points $\overline{Y}^{\mu, \tau}_{\bf{w}}(\overline{\F}) \subset \overline{Y}^{\mu, \tau}(\overline{\F})$ of shape $\bf{w}$ is given by $\pi \left(\Psi^{-1} \left(S_{\bf{w}}^0 \right) \right)$.
\end{prop}

The set of $\bf{w} \in \widetilde{W}^f$ such that $S_{\bf{w}}^0$ is a nonempty subscheme of $\overline{M}(\mu)$ is given by the $\mu$-admissible set $\mathrm{Adm}(\mu) = \prod_j \mathrm{Adm}(2,1,0)$ (see \cite[\S 5.2]{CL} for details).  

\begin{cor} The set $\overline{Y}^{\mu, \tau}_{\bf{w}}(\overline{\F})$ is nonempty if and only if $\bf{w} = (\widetilde{w}_0, \widetilde{w}_1, \ldots, \widetilde{w}_{f-1})$ where $\widetilde{w}_j$ is a $(2,1,0)$-admissible element of $\widetilde{W}$.  
\end{cor}

One can describe $\mathrm{Adm}(2,1,0)$ quite concretely.   
Let $\widetilde{W}^0$ be the affine Weyl group of $\SL_3$.  It is a Coxeter group generated by three reflections $\alpha, \beta$ and $\gamma$. We represent  the elements $\alpha$ and $\beta$ which generate the finite Weyl group by
$$
\alpha = \begin{pmatrix} 0 & 1 & 0 \\ 1 &0 & 0 \\ 0 & 0 & 1 \end{pmatrix}, \quad
\beta = \begin{pmatrix} 1 & 0 & 0 \\ 0 & 0 & 1 \\ 0 & 1 & 0 \end{pmatrix}. 
$$    
The element $\gamma$ is given by 
$$
\gamma = \begin{pmatrix} 0 & 0 & v^{-1} \\ 0 & 1 & 0 \\ v & 0 & 0 \end{pmatrix}.
$$ 
The set $\mathrm{Adm}(2,1,0)$ is a subset of $v \widetilde{W}^0 \subset \widetilde{W}$ (i.e., those matrices with determinant $v^3$(unit)). 
In particular, the element $t_{(2,1,0)} \in \widetilde{W}$ corresponding to the $\GL_3$-cocharacter $(2,1,0)$ is in $v \widetilde{W}^0$.  Recall that $\widetilde{W}^0$ is endowed with a Bruhat ordering, and for $\tld{w}_1$, $\tld{w}_2\in  \widetilde{W}^0$, we say $v\tld{w}_1\leq v\tld{w}_2$ if and only if $\tld{w}_1\leq \tld{w}_2$.   $\mathrm{Adm}(2,1,0)$ is then defined to be the subset of elements $\tld{w}\in v\widetilde{W}^0$ such that $\tld{w}\leq t_{s(2,1,0)}$ for some permutation $s\in S_3$. 

There are six extremal elements in $\mathrm{Adm}(2,1,0)$ of length 4, corresponding to the six permutations of $(2,1,0)$. The length three shapes are divided into two different sets which reflect different behavior on the Galois side. The set $\mathrm{Adm}(2,1,0)$ is given in Table \ref{Table admissible elements}.


The following is the key result for classifying Kisin modules with $\F'$-coefficients, where $\F'$ is a finite extension of $\F$.

\begin{lemma} \label{Frowoper}
Let $\fM,\,\,\fM'\in {Y}^{\mu, \tau}(\F')$ and let $\beta$, $\beta'$ be eigenbases of $\fM$, $\fM'$ respectively. We define $A^{(j)}\defeq \Mat_{\beta}\big(\phi^{(j)}_{R,s_{j+1}(3)}\big)$ $($resp. $A'^{(j)}\defeq \Mat_{\beta'}\big(\phi^{(j)}_{R,s_{j+1}(3)}\big))$ as in Definition \ref{0matrix}.
Assume further that there exists $ J^{(j+1)} \in \cI_1(\F')$ such that $A'^{(j)}=J^{(j+1)}A^{(j)}$ for all $0\leq j\leq f-1$. Then there is an isomorphism $\fM\stackrel{\sim}{\lra}\fM'$ in ${Y}^{\mu, \tau}(\F')$.
\end{lemma}
\begin{proof}
We define, by induction, the following sequence $\big(J^{(j)}_n\big)_{n\in\N}$ of elements in $\cI_1(\F')$.  
For all $j \in \Z/f \Z$, set $J_0^{(j)}\defeq \mathrm{Id}_3$.  For $n\geq1$, set
\begin{equation*}
J_{n+1}^{(j+1)}=J^{(j+1)}A^{(j)} \left(A^{(j)} s_{j+1, j} J^{(j), \phz}_n s_{j+1, j}^{-1} \right)^{-1}
\end{equation*}
where  $J^{(j), \phz}_n$ is constructed from $J^{(j)}_n \in \Iw(\F')$ as in Proposition \ref{changeofbasis}.  Note that this defines a sequence in the pro-$v$ Iwahori $\cI_1(\F')$. 

From the definition of $J_n^{(j)}$ and the hypothesis $A'^{(j)}=J^{(j+1)}A^{(j)}$, we obtain
\begin{equation}
A'^{(j)}=J_{n+1}^{(j+1)}A^{(j)}  \big(s_{j+1, j} J^{(j), \phz}_n s_{j+1, j}^{-1}\big).
\end{equation}
Provided that the sequence $\big(J_n^{(j)}\big)_n$ converges,  we deduce the desired isomorphism 
$\fM \stackrel{\sim}{\lra}\fM'$ via Proposition \ref{changeofbasis}. 
 
We now prove the convergence of the sequence $\big(J_n^{(j)}\big)_{n}$. By the definition of the $v$-adic topology on $\cI(\F')$, it is enough to prove that
\begin{equation}
\label{key convergence}
v^{p(n-2)}\big\vert \big(J^{(j+1)}_{n+1}-J_{n}^{(j+1)}\big)
\end{equation}
for all $n\geq 3$. We induct on $n$. 

It follows directly from the definitions that for all $n\geq 1$ the element $J^{(j+1)}_{n+1}-J^{(j+1)}_n$ equals
\begin{equation}
\label{eq:iter:referee}
J^{(j+1)}A^{(j)} s_{j+1, j} \Big(\Ad\big(v^{a_{s_j(1),f-1-j}},v^{a_{s_j(2),f-1-j}},v^{a_{s_j(3),f-1-j}}\big) \phz\big(J^{(j)}_n-J^{(j)}_{n-1}\big)\Big) s_{j+1, j}^{-1} \big(A^{(j)}\big)^{-1}.
\end{equation}

First, let $n=1$. Then $J_1^{(j)} = J^{(j)}\in\cI_1(\F')$ and hence:
$$
\phz\Big(J_1^{(j)}-\mathrm{Id}_3\Big)\in \begin{pmatrix}(v^p)&\F'[\![v]\!]&\F'[\![v]\!]\\
(v^p)&(v^p)&\F'[\![v]\!]\\
(v^p)&(v^p)&(v^p)\end{pmatrix}.
$$
By the weak genericity condition (\ref{labelgencond}) and the height condition, we deduce that 
$$
v^3 \mid \Ad\big(v^{a_{s_j(1),f-1-j}},v^{a_{s_j(2),f-1-j}},v^{a_{s_j(3),f-1-j}}\big)\cdot \phz\big(J^{(j)}-\mathrm{Id}_3\big)
$$ 
and $v^2 \left(A^{(j)}\right)^{-1} \in \Mat_3(\F'[\![v]\!])$, respectively. We conclude from (\ref{eq:iter:referee}) that $v\big\vert \big(J^{(j+1)}_2-J^{(j+1)}_1\big)$. 
 
For $n = 2$, we have $v^p \big\vert\phz\big(J_2^{(j)}- J_1^{(j)} \big)$ by the previous step, and hence $v^2 \big\vert \big(J_3^{(j)} - J_2^{(j)}\big)$ by the weak genericity condition  (\ref{labelgencond}). Finally, we conclude that $v^p \big\vert \big(J_4^{(j)} - J_3^{(j)}\big)$ since $v^{2p} \big\vert  \phz\big(J_3^{(j)}- J_2^{(j)} \big)$.

For  $n \geq 3$, we see by induction, the weak genericity condition, and the height condition on $A^{(j)}$, that
\begin{equation}
\label{key divisibility}
v^{p^2(n-2) - (p-4) - 2} \big\vert \Big(\Ad\big(v^{a_{s_j(1),f-1-j}},v^{a_{s_j(2),f-1-j}},v^{a_{s_j(3),f-1-j}}\big)\Big( \phz\big(J_n^{(j)}-J_{n-1}^{(j)}\big)\Big)\Big)\big(A^{(j)}\big)^{-1}. 
\end{equation}
Hence
$$
p^2(n-2) - p + 2 \geq p(p(n-2) -1) \geq p(n-1)
$$
since $p \geq 3$.  
\end{proof}

With more careful analysis, Lemma \ref{Frowoper} probably holds with even weaker genericity conditions.  However, we do not attempt such an analysis here.

\begin{thm} 
\label{thm:classification} 
Let $\F'$ be a finite extension of $\F$. Let $\fM \in \overline{Y}^{\mu, \tau}_{\bf{w}}(\F')$ with $\bf{w} = (\widetilde{w}_0, \widetilde{w}_1, \ldots, \widetilde{w}_{f-1})$.  
Fix a choice of representatives for $\widetilde{w}_j\cdot \left(P_{\widetilde{w}_j}\backslash \Iw(\F')\right)$, where
\[
P_{\widetilde{w}_j}\defeq \big(\widetilde{w}_j^{-1}\Iw_1(\F')\widetilde{w}_j\big)\cap\Iw(\F').
\]

Then there exists an eigenbasis $\beta$ of $\fM$ such that for each $j\in \Z/f\Z$ the matrix $A^{(j)} = \Mat_{\beta} \big(\phi^{(j)}_{\fM, s_{j+1}(3)} \big)$ lies in our fixed choice of representatives.

\end{thm}

\begin{proof}
Let $A^{(j)}_1:=\Mat_{\beta_1^{(j)}}\big(\phi_{\fM,s_{j+1}(3)}^{(j)}\big)$  for some eigenbasis $\beta_1^{(j)}$ of $\fM^{(j)}$. 
If $A^{(j)}\in \Iw(\F')\widetilde{w}_j\Iw(\F')$ is such that $A^{(j)}$ and $A^{(j)}_1$ lie in the same left coset $\Iw_1(\F')\backslash\Iw(\F')\widetilde{w}_j\Iw(\F')$, then by Lemma \ref{Frowoper} there exists an eigenbasis of $\beta^{(j)}$ of $\fM^{(j)}$ such that 
$A^{(j)}=\Mat_{\beta^{(j)}}\big(\phi_{\fM',s_{j+1}(3)}^{(j)}\big)$.

By considering the obvious isomorphism
$$
 \widetilde{w}_j\cdot \left(P_{\widetilde{w}_j}\backslash \Iw(\F')\right) \stackrel{\sim}{\longrightarrow} \Iw_1(\F')\backslash \Iw_1(\F')\widetilde{w}_j\Iw(\F')=\Iw_1(\F')\backslash \Iw(\F')\widetilde{w}_j\Iw(\F'),
$$
we conclude that if $\fM^{(j)}$ has shape $\widetilde{w}_j$, then there exists an eigenbasis $\beta$ of $\fM$ such that $\Mat_{\beta}\big(\phi_{\fM,s_{j+1}(3)}^{(j)}\big)=A^{(j)}$ where $A^{(j)}$ lies in our choice of representatives in $\widetilde{w}_j\cdot \left(P_{\widetilde{w}_j}\backslash \Iw(\F')\right)$.
\end{proof}
In Table $\ref{table shapes mod p}$, we have listed a choice of representatives for $\widetilde{w}_j\cdot \left(P_{\widetilde{w}_j}\backslash \Iw(\F')\right)$ for $9$ out of the $25$ elements in $\mathrm{Adm}(2,1,0)$. As we will see below, a choice of representatives for the remaining $16$ can be easily obtained from these $9$ elements by cyclic symmetry.

We now introduce the notion of a \emph{gauge basis} of a mod $p$ Kisin module: 
\begin{defn}
\label{definition gauge basis mod p}
Let $\overline{\fM} \in \overline{Y}^{\mu, \tau}_{\bf{w}}(\F')$. A \emph{gauge basis} $\beta=(\beta^{(j)})_j$ of $\fM$ is an eigenbasis such that for each $0\leq j\leq f-1$, the matrix $\Mat_{\beta}\big(\phi^{(j)}_{\fM, s_{j+1}(3)}\big)$ is in the form given by the  $\widetilde{w}_j$-entry in Table \ref{table shapes mod p}.
\end{defn}
\noindent Note that a gauge basis always exists by Theorem \ref{thm:classification}.

In the above discussion,  one could just as well have chosen to use $\phi^{(j)}_{\fM, s_{j+1}(2)}$ or $\phi^{(j)}_{\fM, s_{j+1}(1) }$ instead of $\phi^{(j)}_{\fM, s_{j+1}(3)}$.   There is a simple way of determining the matrices for $\phi^{(j)}_{\fM, s_{j+1}(2)}$ or $\phi^{(j)}_{\fM, s_{j+1}(1) }$ in terms of $\phi^{(j)}_{\fM, s_{j+1}(3)}$.  Furthermore, while the \emph{shape} of the Kisin module depends on the choice of the isotypic piece, there is a simple recipe which relates them.  As a consequence we can restrict ourselves to the study of the shapes listed in Table \ref{table shapes mod p}, and the results for any other shape can be easily deduced by cyclic symmetry via Corollary \ref{symmetry} below.

Let 
$$
 \delta = \begin{pmatrix} 0 & 0 & v^{-1} \\ 1 &0 & 0 \\ 0 & 1 & 0 \end{pmatrix} \in \widetilde{W}.
$$
Conjugation by $\delta$ induces an outer automorphism of $\widetilde{W}^0$ of order 3 satisfying 
$$
\delta \alpha \delta^{-1} = \beta, \quad \delta \beta \delta^{-1} = \gamma, \quad \delta \gamma \delta^{-1} = \alpha.
$$
It is furthermore easy to check that 
$$
\delta \Iw(R) \delta^{-1} = \Iw(R)
$$
for any $\cO$-algebra $R$. 

\begin{prop} \label{cyclicsymmetry} Let  $\fM \in Y^{[0,h], \tau}(R)$.  Let $\beta^{(j)} =\left (f_{1}^{(j)}, f_{2}^{(j)}, f_{3}^{(j)} \right)$ be an eigenbasis for $\fM$ and let $A^{(j)}_{\bf{3}}\defeq \Mat_{\beta}\big(\phi^{(j)}_{\fM, s_{j+1}(3)}\big)$ be as in Definition \ref{0matrix}.
Then 
\begin{align*}
\text{\tiny$\Big \{u^{e - \bf{a}^{(j)}_{s_{j}(2)} + \bf{a}^{(j)}_{s_j(3)}}\otimes f_{s_j(3)}^{(j-1)}, 
u^{\bf{a}^{(j)}_{s_j(1)} - \bf{a}^{(j)}_{s_j(2)}}\otimes f_{s_j(1)}^{(j-1)},  1\otimes f_{s_j(2)}^{(j-1)} \Big \}$},\  
&
\text{\tiny$\Big\{u^{e - \bf{a}^{(j)}_{s_{j}(2)} + \bf{a}^{(j)}_{s_j(3)}} f_{s_j(3)}^{(j)}, 
u^{\bf{a}^{(j)}_{s_j(1)} - \bf{a}^{(j)}_{s_j(2)}} f_{s_j(1)}^{(j)},  f_{s_j(2)}^{(j)} \Big\}$}
\\
\left(\text{\tiny$
\Big\{u^{e - \bf{a}^{(j)}_{s_{j}(1)} + \bf{a}^{(j)}_{s_j(2)}}\otimes f_{s_j(2)}^{(j-1)}, 
u^{e -\bf{a}^{(j)}_{s_j(1)} + \bf{a}^{(j)}_{s_j(3)}}\otimes f_{s_j(3)}^{(j-1)},  1\otimes f_{s_j(1)}^{(j-1)} \Big\}$},\right.&
\left.\text{\tiny$\Big \{u^{e - \bf{a}^{(j)}_{s_{j}(1)} + \bf{a}^{(j)}_{s_j(2)}}  f_{s_j(2)}^{(j)}, 
u^{e -\bf{a}^{(j)}_{s_j(1)} + \bf{a}^{(j)}_{s_j(3)}} f_{s_j(3)}^{(j)},  f_{s_j(1)}^{(j)}\Big\}$}
\right )
\end{align*}
are bases for ${}^\phz\fM_{s_j(2)}^{(j -1)}$ and $\fM_{s_j(2)}^{(j)}$ respectively $($resp. for  ${}^\phz\fM_{s_j(1)}^{(j -1)}$ and $\fM_{s_j(1)}^{(j)}$ respectively$)$.  
If $A^{(j)}_{\bf{2}}\defeq \Mat_{\beta}\big(\phi_{\fM,\mathbf{a}_{s_{j+1}(2)}}^{(j)}\big)$ and $A^{(j)}_{\bf{1}}\defeq \Mat_{\beta}\big(\phi_{\fM,\mathbf{a}_{s_{j+1}(1)}}^{(j)}\big)$ are defined in the evident way following Definition \ref{0matrix}, 
then 
$$
A^{(j)}_{\bf{2}} = \delta A^{(j)}_{\bf{3}} \delta^{-1}, \quad A^{(j)}_{\bf{1}} = \delta^2 A^{(j)}_{\bf{3}} \delta^{-2}.
$$
\end{prop}
\begin{proof} 
We give the proof for $\fM_{s_j(2)}^{(j -1)}$ as the other proof is similar. Let $C^{(j)}\defeq \Mat_{\beta}\big(\phi_{\fM}^{(j)}\big)$ be as in Definition \ref{0matrix}
so that by Proposition \ref{addingdd} we have 
\begin{equation} \label{z1}
C^{(j-1)} = \Ad_{s_{j}}(u^{\bf{a}_1}, u^{\bf{a}_2}, u^{\bf{a}_3}) \big(A^{(j-1)}\big).
\end{equation}
Let $c_1 = \beta \alpha$ be the permutation matrix corresponding to the cycle $(132)$. It is clear that the elements listed in the statement form a basis for ${}^\phz\fM_{s_j(2)}^{(j-1)}$ and $\fM_{s_j(2)}^{(j)}$ respectively and hence the same argument as in Proposition \ref{addingdd} shows that 
\begin{eqnarray} 
\label{z2}
&&A^{(j-1)}_{\bf{2}} = \Ad\big(u^{- \bf{a}^{(j)}_{s_j(3)} -e}, u^{- \bf{a}^{(j)}_{s_j(1)} }, u^{-\bf{a}^{(j)}_{s_{j}(2)}}\big) \big(c_1^{-1} s_{j}^{-1} C^{(j-1)} s_{j} c_1\big), 
\end{eqnarray}
Combining (\ref{z1}) and (\ref{z2}),  we see that 
\begin{eqnarray*}
A^{(j)}_{\bf{2}} = \Ad(\delta_{\bf{2}})\big(A^{(j)}_{\bf{3}}\big),
\end{eqnarray*}
where 
\begin{eqnarray*}
&&\delta_{\bf{2}} = \Diag\big(u^{- \bf{a}^{(j)}_{s_j(3)} -e}, u^{- \bf{a}^{(j)}_{s_j(1)} }, u^{-\bf{a}^{(j)}_{s_{j}(2)}}\big)  c_1^{-1} \Diag\big(u^{\bf{a}^{(j)}_{s_j(1)}}, u^{\bf{a}^{(j)}_{s_j(2)}}, u^{\bf{a}^{(j)}_{s_j(3)}}\big).
\end{eqnarray*}
A direct computation shows that $\delta_{\bf{2}} = \delta$. 
\end{proof} 

\begin{cor} \label{symmetry} Let  $\fM \in Y^{[0,h], \tau}(\F')$. If $\fM$ has shape $\bf{w} = (\widetilde{w}_0, \ldots, \widetilde{w}_{f-1})$ then for any eigenbasis $\beta = (\beta^{(j)})$, 
$$
\Mat_{\beta}(\phi_{\fM,\mathbf{a}_{s_{j+1}(2)}}^{(j)}) \in \Iw(\F') (\delta \widetilde{w}_j \delta^{-1}) \Iw(\F'), \quad \Mat_{\beta}(\phi_{\fM,\mathbf{a}_{s_{j+1}(1)}}^{(j)}) \in \Iw(\F') (\delta^2 \widetilde{w}_j \delta^{-2}) \Iw(\F')
$$
for all $0\leq j\leq f-1$.  
\end{cor} 

\begin{rmk} \label{deltaorbits} As a consequence of Corollary \ref{symmetry}, there is symmetry among the 25 shapes of $\mathrm{Adm}(2,1,0)$.  It is easy to see that $\delta \mathrm{Adm}(2,1,0) \delta^{-1} = \mathrm{Adm}(2,1,0)$ and that there are 9 orbits under conjugation by $\delta$.  In Table \ref{table shapes mod p}, we choose representatives for these 9 orbits and restrict our attention to those 9 shapes.   One can deduce all our results for the remaining 18 shapes simply by conjugating by $\delta$ or $\delta^2$. 
\end{rmk}

\subsection{\'Etale $\phz$-modules}
\label{subsection:etale phi modules}

We recall briefly some properties of \'etale $\phz$-modules which are well-known. We refer to \cite[\S 2.1]{CDM} and \cite[\S 5.3]{CL}  for proofs.

Let $\cO_{\cE, K}$ denote the $p$-adic completion of $\fS[\frac 1v]$, where $\fS\defeq W[\![v]\!]$, endowed with the unique continuous Frobenius morphism such that the natural inclusion $\fS[ \frac 1v]\ia \cO_{\cE, K}$ is Frobenius-equivariant. Let $R$ be a local, complete Noetherian $\cO$-algebra. By base change, the ring $\cO_{\cE, K}\widehat{\otimes}_{\Zp}R$ is naturally endowed with a Frobenius endomorphism $\phz$ and we write $\Phi\text{-}\Mod^{\text{\'et}}(R)$ for the category of \'etale $(\phz,\cO_{\cE, K}\widehat{\otimes}_{\Zp}R)$-modules. We fix once and for all a sequence $\underline{p}\defeq (p_n)_{n\in\N}$ where $p_n\in \Qpbar$ verify $p_{n+1}^{p}= p_n$ and $p_0=-p$. We let $K_{\infty}\defeq\underset{n\in\N}{\bigcup}K(p_n)$ and $G_{K_{\infty}}\defeq\Gal(\Qpbar/K_{\infty})$. 

By classical work of Fontaine (\cite{fontaine-fest}) we have an exact anti-equivalence of $\otimes$-categories:
\begin{eqnarray*}
\Phi\text{-}\Mod^{\text{\'et}}(R)&\stackrel{\sim}{\longrightarrow}& \Rep_{G_{K_{\infty}}}(R)\\
\cM&\longmapsto& \bV^*(\cM)\defeq
\Hom_{\Phi\text{-Mod}}\big(\cM,\cO_{\cE^{un}, K}\big)
\end{eqnarray*}
where $\cO_{\cE^{un}, K}$ is the \'etale extension of $\cO_{\cE,K}$ corresponding to a separable closure of $k(\!(v)\!)$.

The above construction can also be carried out with descent datum. More precisely, choose $(\pi_n)_{n\in\N}$ to be the sequence satisfying $\pi_n^e=p_n$ and $\pi_{n+1}^p= \pi_n$ with $\pi_0=\pi$. Then $L_{\infty}\defeq \underset{n\in\N}{\bigcup}L(\pi_n)$ and $G_{L_{\infty}}\defeq \Gal(\Qpbar/L_{\infty})$. We have $\Gal(L_{\infty}/K_{\infty}) \cong \Gal(L/K) = \Delta$.  Let $\cO_{\cE, L}$ denote the $p$-adic completion of $(W[\![u]\!])[1/u]$ equipped with an action of $\Gal(L_{\infty}/K_{\infty})\cong \Delta$ characterized by $\widehat{g}(u)\defeq \omega_{\pi}(g) u$.
We define, in the evident way, the category $\Phi\text{-}\Mod^{\text{\'et}}_{dd}(R)$ of \'etale $\phz$-module over $\cO_{\cE,L}\widehat{\otimes}_{\Zp}R$ with descent data and note that $\cO_{\cE,L}$ is an $\cO_{\cE,K}$-algebra (by ``ramifying the variable'' $u=v^{1/e}$).


We have an exact anti-equivalence of categories:
\begin{eqnarray*}
\Phi\text{-}\Mod^{\text{\'et}}_{dd}(R)&\stackrel{\sim}{\longrightarrow}& \Rep_{G_{K_{\infty}}}(R)\\
\cM&\longmapsto& \bV_{dd}^*(\cM)\defeq
\Hom_{\phi, \cO_{\cE, L}} \big(\cM,\cO_{\cE^{un},K}\big)
\end{eqnarray*}
where we make $G_{K_{\infty}}$ act on $\Hom_{\phi, \cO_{\cE, L}} \big(\cM,\cO_{\cE^{un},K}\big)$ via $g\cdot f\defeq g\circ f\circ \widehat{\overline{g}}^{-1} $ (here we write $\widehat{\overline{g}}$ to denote the automorphism of $\cM$ associated to $\overline{g}\in\Gal(L_\infty/K_\infty)\cong \Delta$ via the descent data).

Define $T_{dd}^*$ to be the composition $Y^{[0,h], \tau}(R) \ra \Phi\text{-}\Mod^{\text{\'et}}_{dd}(R) \ra \Rep_{G_{K_{\infty}}}(R)$, where the first map is given by tensoring with $\cO_{\cE, L}$ (over $(W[\![u]\!])$).

In order to compute $T_{dd}^*$, it is convenient to both remove the descent datum and pass to a single Frobenius.  This is carried out in \cite[\S 2.1.3]{CDM}. We briefly recall the construction.  

For any $\cO$-algebra $R$, we consider $R$ as a $W$-algebra via $\sigma_0$. We endow the ring $\cO_{\cE, K}\widehat{\otimes}_{W, \sigma_0}R$ with a Frobenius $\phz^f$ (by base change). We can now define in the evident fashion the category $\Phi^f\text{-}\Mod^{\text{\'et}}_{W, \sigma_0}(R)$ of \'etale $(\phz^{f},\cO_{\cE, K}\widehat{\otimes}_{W, \sigma_0}R)$-modules and we have an exact equivalence of categories:
\begin{eqnarray*}
\Phi^f\text{-}\Mod^{\text{\'et}}_{W, \sigma_0}(R)&\stackrel{\sim}{\longrightarrow}& \Rep_{G_{K_{\infty}}}(R)\\
\cM&\longmapsto& \bV^*_{W}(\cM)\defeq
\Hom_{\Phi^f\text{-Mod}}\big(\cM,\cO_{\cE^{un}, K}\big).
\end{eqnarray*}
 
In particular, if $(\cM, \phi_{\cM}) \in \Phi\text{-}\Mod^{\text{\'et}}(R)$, then $(\cM^{(0)}, \phi_{\cM}^f) \in \Phi^f\text{-}\Mod^{\text{\'et}}_{W, \sigma_0}(R)$. This defines a functor $\varepsilon_0:\Phi\text{-}\Mod^{\text{\'et}}(R)\rightarrow \Phi^f\text{-}\Mod^{\text{\'et}}_{W, \sigma_0}(R)$.

We have the following compatibility between the above constructions (cf. Theorem 2.1.6 and equation (12) in \cite{CDM}): 
\begin{equation*}
\hspace{-1in}
\xymatrix@=6pc{
Y^{[0,2],\tau}(R) \ar[r] \ar@/^2pc/^{T_{dd}^*}[rr]
&\Phi\text{-}\Mod^{\text{\'et}}_{dd}(R)\ar^{\bV^*_{dd}}[r]\ar@/^2pc/^{(\bullet)^{\Delta = 1}}[d]&\Rep_{G_{K_{\infty}}}(R)\\
&\Phi\text{-}\Mod^{\text{\'et}}(R)\ar^{\varepsilon_0\cdot(\bullet)}[r]\ar^{\bV^*}[ur]\ar@/^2pc/[u]^{-\otimes_{\cO_{\cE, K}}\cO_{\cE, L}}&\Phi^f\text{-}\Mod^{\text{\'et}}_{W}(R)\ar_{\bV^*_{W}}[u].
}
\end{equation*}





If $R$ is a $\F$-algebra and $\fM\in Y^{[0,2],\tau}(R)$, it will be useful to describe the \'etale $(\phz^f,R(\!(v)\!))$-module $\varepsilon_0\big((\fM \otimes_{\F[\![u]\!]} \F(\!(u)\!))^{\Delta = 1}\big)$ explicitly in terms of the $A^{(j)}$.

\begin{prop}
\label{prop:phif}
Let $\fM\in Y^{[0,2], \tau}(R)$ and $\beta$ be an eigenbasis of $\fM$. Write $(s_j)$ for an orientation of $\tau$, $\big(A^{(j)}\big)=\Mat_{\beta}\big(\phi^{(j)}_{\fM, s_{j+1}(3)}\big)$
and consider $\cM = \fM[1/u] \in \Phi\text{-}\Mod^{\text{\'et}}_{dd}(R)$. Then the \'etale $(\phz^f, R(\!(v)\!))$-module $\varepsilon_0\big(\cM^{\Delta = 1}\big)$ is described with respect to the basis 
$\mathfrak{f}=(u^{\mathbf{a}_{1}^{(0)}}f^{(0)}_{1}, u^{\mathbf{a}_{2}^{(0)}}f^{(0)}_{2},  
u^{\mathbf{a}_{3}^{(0)}}f^{(0)}_{3})$  by
\begin{eqnarray*}
\Mat_{\mathfrak{f}}(\phi_{\cM^{(0)}}^f)=\prod_{j=0}^{f-1}s_{f-j}\cdot\phz^{j}\left(A^{(f-1-j)}\text{\tiny$
\begin{pmatrix}v^{a_{s_{f-j}(1),j}}&0&0\\0&v^{a_{s_{f-j}(2),j}}&0\\
0&0&v^{a_{s_{f-j}(3),j}}\end{pmatrix}
$}\right)\cdot s_{f-j}^{-1}.
\end{eqnarray*}

\end{prop}
\begin{proof}
This is a direct computation (cf. \cite[(24)]{CDM}). 
\end{proof}

\section{Kisin varieties and tangent spaces}
\label{Kisintangent}

In this section, we show that under the weak genericity assumption (Definition \ref{gencond}) the Kisin variety is trivial; in particular, if $\rhobar:G_{K} \ra \GL_3(\F')$ comes from a Kisin module $\overline{\fM} \in Y^{\mu, \tau}(\F')$, then $\overline{\fM}$ is unique. As a consequence, one can attach a shape $\bf{w}(\rhobar, \tau)$ to $\rhobar$.  Later, in Proposition \ref{prop:intersec} (also Table \ref{table:SW}), we show that the shape $\bf{w}(\rhobar, \tau)$ is closely related to the Serre weights of $\rhobar$ which appear in the principal series type $\overline{\sigma}(\tau)$ (as was the case for $\GL_2$ \cite{breuil-buzzati}).  In this section, we also show that the map on tangent spaces from deformations of Kisin modules to deformations of \'etale $\phi$-modules is injective and that under mild assumptions the same is true for the restriction on Galois deformations from $G_K$ to $G_{K_{\infty}}$.   

\subsection{Kisin varieties}
In what follows, let $\F'$ denote a finite extension of $\F$.
If $\cM \in \Phi\text{-}\Mod^{\text{\'et}}_{dd}(\F)$ is an \'etale 
$\cO_{\cE, L} \otimes_{\Zp}\F$-module with descent data (cf. \S \ref{subsection:etale phi modules}) and $R$ is any $\F'$-algebra, we set $\cM_R\defeq \cM \otimes_{\F'(\!(u)\!)}R(\!(u)\!) \in  \Phi\text{-}\Mod^{\text{\'et}}_{dd}(R)$.

\begin{defn} \label{defn:Kisinvariety} The \emph{Kisin variety} $Y_{\cM}^{[0,2], \tau}(R)$  of $\cM$ is a projective scheme over $\Spec(\F)$ which represents the functor: 
\begin{equation*}
Y_{\cM}^{[0,2], \tau}(R) \defeq
\left \{ 
\begin{aligned}
\fM_R \subset \cM_R \mid\, \fM_R[1/u] = \cM_R,\ \phi_{\cM_R}(\fM_R) \subset \fM_R,\ \fM_R \in Y^{[0,2], \tau}(R)  
\end{aligned}
\right \}
\end{equation*}
In other words, $Y_{\cM}^{[0,2], \tau}(R)$ is the set of $(k\otimes_{\Fp}R)[\![u]\!]$-lattices in $\cM_R$ which have type $\tau$ and height $\leq$ 2.  We can also consider 
\begin{equation*}
Y_{\cM}^{\mu, \tau} =\left \{ 
\begin{aligned}\fM_R \subset \cM_R \mid \fM_R[1/u] = \cM_R,\ \phi_{\cM_R}(\fM_R) \subset \fM_R,\ \fM_R \in Y^{\mu, \tau}(R)\end{aligned}  \right \}.
\end{equation*}
\end{defn}

There is an obvious inclusion $Y_{\cM}^{\mu, \tau} \subset Y_{\cM}^{[0,2], \tau}$.  These are projective schemes because they are closed subschemes of finite type in the affine Grassmannian for the group $\Res_{k/\Fp} \GL_3$. (cf. \cite[Proposition 2.1.7]{kisin-annals}: the proof there is in the height 1 case, but works for all heights.) The result for $Y_{\cM}^{\mu, \tau}$ follows from the fact that $Y^{\mu, \tau}$ is a closed substack of $Y^{[0,2], \tau}$ (cf. \cite[Proposition 5.2]{CL}).

 \begin{thm} \label{Kisinvariety} If $\tau$ is weakly generic $($Definition $\ref{labelgencond})$, then $Y_{\cM}^{\mu, \tau}(\F')$ is either empty or a single point. 
 \end{thm} 
 \begin{proof}   
We show that $Y_{\cM}^{[0,2], \tau}(\F')$ is either empty or a single point.  Assume we have two 
$(k\otimes_{\Fp}\F')[\![u]\!]$-lattices $\fM_1$ and $\fM_2$ in $\cM_{\F'}$. 
For $i\in\{1,2\}$ choose an eigenbasis $\beta_i$ for $\fM_i$, such that 
$\Mat_{\beta_i}\big(\phi^{(j)}_{\fM_i,s_{j+1}(3)}\big)=A_i^{(j)}$, where $(s_j)$ denotes the orientation on $\tau$.
Let $\big(D^{(j)}\big) \in \big(\GL_3 (\F'(\!(u)\!))\big)^f$ be the $f$-tuple of matrices which gives the basis for $\fM_2$ in terms of $\fM_1$ as in  Proposition \ref{changeofbasis}. Note that, \emph{a priori}, the matrices $D^{(j)}$ have denominators in $u$.

We want to show that $D^{(j)} \in \GL_3 (\F'[\![u]\!])$ for all $j$.  For all $j \in \Z/f\Z$, let us define
$I^{(j)} = \Ad^{-1}_{s_j}(u^{\bf{a}_1}, u^{\bf{a}_2}, u^{\bf{a}_3})\big(D^{(j)}\big) \in \GL_3(\F'(\!(v)\!))$  and recall the change of basis formula (\ref{cob1}) which remains valid for $\cM$:
\begin{equation} \label{a1}
A^{(j)}_2= I^{(j+1)}A^{(j)}_1 s_{j+1,j}\Big(\Ad\big(v^{a_{s_{j(1),f - j - 1}}},v^{a_{s_{j(2),f - j - 1}}},v^{a_{s_{j(3),f - j - 1}}}\big)\big(\phz\big(I^{(j)}\big)^{-1}\big)\Big)s_{j+1,j}^{-1}.
\end{equation}
For each $j \in \Z/f\Z$, define $k_j \in \Z$ such that  $v^{k_{j}} I^{(j)} =: I^{(j), +} \in \Mat_3(\F'[\![v]\!])$ and such that $I^{(j), +} \not\equiv 0 \mod v$.
Rearranging (\ref{a1}), we get
\begin{equation} \label{a2} 
\text{\small{$v^{-p k_j} \cdot s_{j+1,j}\Big(\Ad\big(v^{a_{s_{j}(1),f - j - 1}},v^{a_{s_{j}(2),f - j - 1}},v^{a_{s_{j}(3),f - j - 1}}\big) 
\cdot \phz\big(I^{(j), +}\big)\Big)s_{j+1,j}^{-1}= v^{-k_{j+1}} \big(A^{(j)}_2 \big)^{-1} I^{(j+1), +}A^{(j)}_1.$}}
\end{equation}
Multiplying through by $v^{2 + k_{j+1}}$ the right side of (\ref{a2}) becomes integral.   Since $\big(I^{(j), +}\big)_{i,k}\in {(\F')}\s+v\F'[\![v]\!]$ for some $1\leq i,k\leq 3$, $\Ad\big(v^{a_{s_{j}(1),f - j - 1}},v^{a_{s_{j}(2),f - j - 1}},v^{a_{s_{j}(3),f - j - 1}}\big) \big( \phz\big(I^{(j), +}\big) \big)$ is at most divisible by $v^{a_{s_j(1),f-1-j}-a_{s_j(3),f-1-j}}$.   
We conclude that
\begin{equation} \label{a3}
k_{j+1} \geq p k_j - (a_{s_j(1),f-1-j}-a_{s_j(3),f-1-j}) - 2.
\end{equation}

Since $\tau$ is weakly generic, $ \max_j\{a_{s_j(1),f-1-j}-a_{s_j(3),f-1-j}\}< p-3$.  Thus, if $k_j \geq 1$ for any $j$, then by iterating (\ref{a3}) we deduce that all $k_j$ become arbitrary large.
Thus, $k_j  \leq 0$ for all $j$ and $I^{(j)} \in \Mat_3(\F'[\![v]\!])$.   
Interchanging the roles of $\fM_1$ and $\fM_2$, we conclude that $I^{(j)} \in \GL_3(\F'[\![v]\!])$.

It remains to show that $I^{(j)} \in \Iw(\F')$: by Lemma \ref{integrality}, this is equivalent to 
$D^{(j)} \in\GL_3(\F'[\![u]\!])$.
Rearranging again the change of basis formula (\ref{cob1}), we have
\begin{equation}
\label{a4}
\big(A^{(j)}_2 \big)^{-1} I^{(j+1)} A^{(j)}_1=s_{j+1,j}
\Big(\Ad\big(v^{a_{s_{j(1),f-j-1}}},v^{a_{s_{j(2),f-j-1}}},v^{a_{s_{j(3),f-j-1}}}\big) \big(\phz\big(I^{(j)}\big)\Big)s_{j+1,j}^{-1}.
\end{equation}
For $1\leq k, h\leq 3$, we have
\begin{eqnarray*}
\Big(\Ad\big(v^{a_{s_j(1),f-j-1}},v^{a_{s_j(2),f-j-1}},v^{a_{s_j(3),f-j-1}}\big) \big(\phz\big(I^{(j)}\big)\big)\Big)_{hk}&\in& 
\big(v^{p\alpha_{hk}-(a_{s_{j}(k),f-1-j}-a_{s_{j}(h),f-1-j})}\big)
\end{eqnarray*}
where the integers $\alpha_{hk}\in\N$ are defined by $\left(I^{(j)}\right)_{h,k}\in \left(v^{\alpha_{h,k}}\right)$.
On the other hand, the height condition on $A_2^{(j)}$ forces the LHS in (\ref{a4}) to be an element in $\frac{1}{v^2}\Mat_3(\F'[\![v]\!])$. 
In particular, we have $p\alpha_{hk}-(a_{s_{j}(k),f-1-j}-a_{s_{j}(h),f-1-j})\geq -2$ for all  $1\leq k, h\leq 3$ and this implies $\alpha_{h,k}\geq 1$ when $h>k$ since $\tau$ is weakly generic. Therefore, $I^{(j)}\in \cI(\F')$, as required.


\end{proof}

Theorem \ref{Kisinvariety} allows us to attach a shape $\mathbf{w}(\rhobar,\tau)$ to $\rhobar: G_K\ra \GL_3(\F')$ when the type $\tau$ is weakly generic: 

\begin{defn} \label{rhobarshape} Let $\rhobar:G_K \ra \GL_3(\F)$ and $\tau$ be as in Theorem \ref{Kisinvariety}. Assume there exists $\overline{\fM}_{\rhobar} \in Y^{\mu, \tau}(\F)$ such that $T_{dd}^*(\overline{\fM}_{\rhobar}) \cong \rhobar|_{G_{K_{\infty}}}$.  We define $\bf{w}(\rhobar, \tau) \in \mathrm{Adm}(2,1,0)^f$ to be the shape of $\overline{\fM}_{\rhobar}$. Whenever we invoke the shape $\bf{w}(\rhobar, \tau)$ implicit in that is the assertion that there exists a Kisin module $\overline{\fM}_{\rhobar}$ as above.
\end{defn}

Next, we study the tangent space at a closed point of $\overline{\fM} \in Y^{\mu, \tau}(\F')$.  Since $\overline{\fM}$ often has automorphisms, we work at the categorical level.  Define 
$$
\mathfrak{t}_{\overline{\fM}} = \Big\{ (\fM, \delta_0) \mid \fM \in Y^{\mu, \tau}(\F'[\eps]/\eps^2),\ \delta_0:\fM/\eps\fM \stackrel{\sim}{\longrightarrow}\overline{\fM} \Big\}
$$   
which we consider as a category where morphisms are maps in $Y^{\mu, \tau}(\F'[\eps]/\eps^2)$ commuting with trivializations. 
 
\begin{prop} \label{tsff} Assume that $\tau$ is weakly generic. The functor $T_{dd}^*$ induces a fully faithful functor
$$
T_{\mathrm{tan}}^*: \mathfrak{t}_{\overline{\fM}} \ra \Rep_{\F'[\eps]/\eps^2}(G_{K_{\infty}}).
$$   
\end{prop}
\begin{proof} 
The functor $\bV_{dd}^*:\Phi\text{-}\Mod^{\text{\'et}}_{dd}(\F'[\eps]/\eps^2) \ra \Rep_{\F'[\eps]/\eps^2}(G_{K_{\infty}})$ is an anti-equivalence of categories so we are reduced to showing that
$$
\fM \mapsto \fM[1/u]
$$
is fully faithful on $\mathfrak{t}_{\overline{\fM}}$.    Let $\fM_1, \fM_2 \in \mathfrak{t}_{\overline{\fM}}$.  Choose an eigenbasis $\overline{\beta}$ of $\overline{\fM}$ and let $A^{(j)}=\Mat_{\overline{\beta}}\big(\phi^{(j)}_{\overline{\fM},s_{j+1(3)}}\big)$. 
For $i\in\{1,2\}$, we fix eigenbases $\beta_i$ of $\fM_i$ lifting $\overline{\beta}$ and write $A^{(j)} + \eps B_i^{(j)}=\Mat_{\beta_i}\big(\phi^{(j)}_{\fM_i,s_{j+1(3)}}\big)$ for some $B_i^{(j)}\in\Mat_3(\F'[\![v]\!])$.
An isomorphism $\iota:\fM_1[1/u] \stackrel{\sim}{\longrightarrow} \fM_2[1/u]$ which is trivial modulo $(\eps)$ satisfies $\Mat_{\beta_1,\beta_2}(\iota)=\Id_3 + \eps D^{(j)}$ for some $D^{(j)} \in \Mat_3(\F'(\!(u)\!))$.  

Define $Y^{(j)} \defeq \Ad^{-1}_{s_j}(u^{\bf{a}_1}, u^{\bf{a}_2}, u^{\bf{a}_3})\big(D^{(j)}\big) \in \Mat_3(\F'(\!(v)\!))$.   A direct computation using  Proposition \ref{changeofbasis} gives
$$
B_2^{(j)} = B_1^{(j)} + Y^{(j+1)} A^{(j)} - A^{(j)} s_{j+1, j} \Big(\Ad\big(v^{a_{s_{j}(1),f-j-1}}, v^{a_{s_{j}(2),f-j-1}}, v^{a_{s_{j}(3),f-j-1}}\big)\big(\phz\big(Y^{(j)}\big)\big)\Big) s_{j+1, j}^{-1}. 
$$
Arguing as in Theorem \ref{Kisinvariety}, we deduce that $Y^{(j)} \in \Mat_3(\F'[\![v]\!])$.   
More precisely, for each $i \in \Z/f\Z$ such that $Y^{(i)}\neq 0$, define $k_i \in \Z$ by $v^{k_i}Y^{(i)}=Y^{(i),+}\in \Mat(\F'[\![v]\!])$, where $Y^{(i),+}\not\equiv 0$ modulo $v$.  We define $k_i\defeq 0$ if $Y^{(i)}= 0$. We deduce as in Theorem \ref{Kisinvariety} that
\begin{eqnarray}
\label{tangent a2}
&&v^{-pk_j}s_{j+1, j} \Big(\Ad\big(v^{a_{s_{j}(1),f-j-1}}, v^{a_{s_{j}(2),f-j-1}}, v^{a_{s_{j}(3),f-j-1}}\big) \big(\phz\big(Y^{(j),+}\big)\big)\Big) s_{j+1, j}^{-1}=\\
&&\qquad\qquad=\big(A^{(j)}\big)^{-1} \Big(-B_2^{(j)}+v^{-k_{j+1}}Y^{(j+1),+}A^{(j)}+B_1^{(j)}\Big)\nonumber
\end{eqnarray}
and therefore, by the height condition on $A^{(j)}$: 
$$
v^{2-pk_j+k_{j+1}}\Big(\Ad\big(v^{a_{s_{j}(1),f-j-1}}, v^{a_{s_{j}(2),f-j-1}}, v^{a_{s_{j}(3),f-j-1}}\big)\big(\phz\big(Y^{(j),+}\big)\big)\Big)\in \Mat_3(\F'[\![v]\!]).
$$
If $Y^{(j)}\neq 0$, we obtain the key inequality (\ref{a3}). The same iterative argument as above then shows that $k_j\leq 0$.

It remains to show that $D^{(j)}\in \Mat_3(\F'[\![u]\!])$, which is equivalent to proving that $Y^{(j)}$ is upper triangular modulo $v$.  This is immediate if $k_j < 0$; otherwise, specializing (\ref{tangent a2}) at $k_j=0$ and $Y^{(j), +}=Y^{(j)}$, we see that $v^2\Ad\big(v^{a_{s_{j}(1),f-j-1}}, v^{a_{s_{j}(2),f-j-1}}, v^{a_{s_{j}(3),f-j-1}}\big)\big( \phz\big(Y^{(j)}\big)\big)$ is integral. By the weak genericity assumption, the same argument in the proof of Theorem \ref{Kisinvariety} shows that $Y^{(j)}$ is upper triangular mod $v$.
\end{proof}  

\begin{cor}
\label{cor:kis:var}
Assume that $\tau$ is weakly generic.  
If $Y^{\mu,\tau}_{\cM}(\F)\neq\emptyset$ then $Y^{\mu,\tau}_{\cM}=\Spec(\F)$.
\end{cor}

\subsection{Kisin resolution}

We now apply the computations from the previous section to obtain preliminary results in our study of potentially crystalline deformation rings.  Fix a representation $\rhobar:G_{K} \ra \GL_3(\F)$.  

Let $R^{\mu, \tau}_{\rhobar}$ be the framed potentially crystalline deformation ring with parallel Hodge-Tate weights $(2,1,0)$ and inertial type $\tau$ as in \cite{KisinPSS}. Consider the projective morphism 
$$
\Theta:Y^{\mu, \tau}_{\rhobar} \ra \Spf R^{\mu, \tau}_{\rhobar}
$$
as constructed in \cite{CL} (see Theorem 5.19 and discussion before).  This is a version with descent datum of the partial resolution introduced in \cite[\S\ (1.4)]{KisinPSS}.  Note that 
\begin{equation}
\label{eq:kis:fibre}
Y^{\mu, \tau}_{\rhobar}\times_{\Spf R^{\mu, \tau}_{\rhobar}}\Spf (R^{\mu, \tau}_{\rhobar}/\fm)=Y^{\mu, \tau}_{\cM}
\end{equation}
with $Y^{\mu, \tau}_{\cM}$ as defined in the previous section. Here $\cM$ is the \'etale $\phz$-module of $\rhobar$.

Since we always work in parallel weight $(2,1,0)$, we  drop $\mu$ from the notation.  Set $D_{\rhobar}^{\square, \tau} \defeq  \Spf R^{\mu, \tau}_{\rhobar}$. For $\tau$ weakly generic, by (\ref{eq:kis:fibre}) and Corollary \ref{cor:kis:var}, $\Theta$ is quasi-finite and hence finite since it is proper. Thus $Y^{\mu, \tau}_{\rhobar} = \Spf R_{\overline{\fM}, \rhobar}^{\tau,\Box}$ with $\Spf R_{\overline{\fM}, \rhobar}^{\tau,\Box}$~local and finite over $\Spf R^{\mu, \tau}_{\rhobar}$.


\begin{cor} \label{RKisin}  Let $\rhobar:G_{K} \ra \GL_3(\F)$.  If $\tau$ is weakly generic, then 
$$
\Theta:Y^{\mu, \tau}_{\rhobar} \ra \Spf R^{\mu, \tau}_{\rhobar}
$$
is an isomorphism.
\end{cor}
\begin{proof} 
We saw above that $\Theta$ is a finite morphism.  By \cite[Theorem 5.19]{CL}, $\Theta[1/p]$ is an isomorphism.  Proposition \ref{tsff} implies that the induced map
$$
\Spf R_{\overline{\fM}, \rhobar}^{\tau,\Box} \ra \Spf R^{\mu, \tau}_{\rhobar}
$$
is an injective on tangent spaces.  Hence $\Theta$ is a closed immersion.  Since $R^{\mu, \tau}_{\rhobar}$ is $\cO$-flat, we conclude that $\Theta$ is an isomorphism.
\end{proof}
Thus $R_{\overline{\fM}, \rhobar}^{\tau,\Box}$ is the complete local $\cO$-algebra representing the deformation problem
\begin{equation} \label{defn:Kisres}
D_{\overline{\fM}, \rhobar}^{\tau,\Box}(A) \defeq \big\{ (\fM_A, \rho_A, \delta_A) \mid \fM_A \in  Y^{\mu, \tau}(A), \ \rho_A \in D_{\rhobar}^{\square, \tau}(A),\  \delta_A:T_{dd}^*(\fM_A) \cong (\rho_A)|_{G_{K_{\infty}}}\big\}.
\end{equation}

\subsection{Galois cohomology}

In this section, we work with Galois cohomology to prove that -under mild hypotheses-  the natural restriction map from $G_{K}$-deformations to $G_{K_{\infty}}$-deformations is a closed immersion. 
This will be important in \S \ref{subsection:PCDR}.  

\begin{defn}
\label{defn n-gen}
Let $\rhobar: G_{K}\rightarrow \GL_3(\F)$ be a continuous semisimple Galois representation and let $m\in\N$ be an integer. Let $\mathbf{a}_k = (a_{k,j})_{j}$ with $0 \leq a_{k, j} \leq p-1$ and $j\in \Z/f\Z$ be $f$-tuples such that $\rhobar|_{I_K} \cong \omega_{f}^{\mathbf{a}_1^{(0)}}\oplus\omega_{f}^{\mathbf{a}_2^{(0)}}\oplus\omega_{f}^{\mathbf{a}_3^{(0)}}$. 
We say that $\rhobar$ is \emph{$m$-generic} if
\begin{eqnarray*}
\label{labelgencond1}
m \leq |a_{1, j} - a_{2, j}|,\,|a_{2, j} - a_{3, j}|,\,|a_{1, j} - a_{3, j}| \leq p-1-m
\end{eqnarray*}
for all $j$. 
We say that a continuous Galois representation $\rhobar: G_{K}\rightarrow \GL_3(\F)$ is $m$-generic if there exists a finite unramified extension $K'/K$ such that $\rhobar^{ss}|_{I_{K'}}$ is the direct sum of characters and is $m$-generic in the previous sense. This does not depend on the extension $K'/K$.
\end{defn}

There is a weaker genericity condition which suffices for our Galois cohomology arguments:
\begin{defn} \label{defn:cyclofree} Let $\rhobar: G_{K}\rightarrow \GL_n(\F)$ be a continuous Galois representation. We say $\rhobar$ is \emph{cyclotomic free} if $\rhobar$ becomes upper triangular over an unramified extension $K'/K$ of degree prime to $p$ such that
$$
H^0\big(G_{K'},\big(\rhobar|_{G_{K'}}^{ss}\big)\otimes\omega^{-1}\big)=0.
$$ 
\end{defn}

\begin{prop} \label{2impliescyc} If $p > 3$ and $\rhobar: G_{K}\rightarrow \GL_3(\F)$ is $2$-generic, then $\ad(\rhobar)$ is cyclotomic free. 
\end{prop} 
\begin{proof}
Let $K'$ denote the unramified extension of $K$ of degree 6. Then $\rhobar|_{G_{K'}}$ is upper triangular and we can write $\left(\rhobar|_{G_{K'}}\right)^{\mathrm{ss}}\vert_{I_{K'}}=\oplus_{i=1}^3\omega_{6f}^{\bf{a}^{(0)}_i}$ where the $6f$-tuple $\bf{a}_i\in \{0,1, \ldots, p-1 \}^{6f}$ is  $2$-generic. 
In particular, it follows that $\ad(\rhobar)|_{G_{K'}}=\ad(\rhobar|_{G_{K'}})$ is upper triangular with diagonal characters of the form $\omega_{6f}^{\bf{a}^{(0)}_i - \bf{a}^{(0)}_{i'}}$, where $i,\ i'\in\{1,2,3\}$. Its semisimplification does not have cyclotomic constituents as long as $\bf{a}^{(0)}_i - \bf{a}^{(0)}_{i'} \not\equiv 1 + p + \cdots + p^{6f-1} \mod p^{6f} - 1$ which follows easily from the $2$-genericity assumption. 
\end{proof}

\begin{lemma}
\label{lemma H1}
Let $\rhobar: G_{K}\rightarrow \GL_n(\F)$ be cyclotomic free.  Then the restriction map $H^1(K,\rhobar)\rightarrow H^1(K_{\infty},\rhobar)$ is injective.
\end{lemma}
\begin{proof}
We first assume that $\rhobar$ is upper triangular. In this case the proof is a standard d\'evissage. More precisely, we have an exact sequence $0\ra \rhobar_1\ra\rhobar\ra\overline{\chi}\ra0$ where $\overline{\chi}:G_{K}\ra \F\s$ is not the cyclotomic character and $\rhobar_1: G_{K}\rightarrow \GL_{n-1}(\F)$ is upper triangular (and $\rhobar_1|_{G_{K'}}^{ss}$ does not contain the cyclotomic character).

Group cohomology provides us with the following commutative diagram, with exact rows:
\begin{equation*}
\xymatrix{
H^0(K,\overline{\chi})\ar^{\delta}[r]\ar^{f_0}[d]&
H^1(K,\rhobar_1)\ar[r]\ar^{f_1}[d]&
H^1(K,\rhobar)\ar[r]\ar^{f_2}[d]&
H^1(K,\overline{\chi})\ar^{f_3}[d]\\
H^0(K_{\infty},\overline{\chi})\ar^{\delta}[r]&
H^1(K_{\infty},\rhobar_1)\ar[r]&
H^1(K_{\infty},\rhobar)\ar[r]&
H^1(K_{\infty},\overline{\chi})
}
\end{equation*}
the vertical maps being induced by restriction to $G_{K_{\infty}}$. 
By \cite[Lemma 5.4.2]{GLS15}, the morphism $f_3$ is injective. By d\'evissage, we can assume that $f_1$ is injective. Finally, as $\overline{\chi}$ is a character and as all characters are tame, $f_0$ is surjective. Hence $f_2$ is injective by the ``four lemma.''

As for the general case, letting $K'_{\infty}=K_{\infty}\cdot K'$ we have an exact sequence of groups
\begin{equation*}
\xymatrix{
1\ar[r]&
G_{K'}\ar^{\vartriangleleft}[r]&
G_{K}\ar[r]&
\Gal(K'/K)\ar[r]&1\\
1\ar[r]&
G_{K'_{\infty}}\ar^{\vartriangleleft}[r]\ar@{^{(}->}[u]&
G_{K_{\infty}}\ar[r]\ar@{^{(}->}[u]&
\Gal(K'_{\infty}/K_{\infty})\ar[r]\ar^{\wr}[u]&1
}
\end{equation*}
and hence restriction to $G_{K_{\infty}}$ induces a morphism between the Hochschild-Serre spectral sequences
\begin{equation}
\label{Hoch-Serre}
\xymatrix{
H^{r}\big(\Gal(K'/K), H^s(K',\rhobar)\big)\ar@{=>}[r]\ar[d]&H^{r+s}(K,\rhobar)\ar[d]\\
H^{r}\big(\Gal(K'_{\infty}/K_{\infty}), H^s(K'_{\infty},\rhobar)\big)\ar@{=>}[r]&H^{r+s}(K_{\infty},\rhobar).
}
\end{equation}
As $p\nmid [K':K]$, the category of $\Gal(K'/K)$-representations over $\F$ is semisimple and the above spectral sequence becomes simply
\begin{equation}
\label{Hoch-Serre ss}
\xymatrix{
H^{0}\big(\Gal(K'/K), H^s(K',\rhobar)\big)\ar^-{\sim}[r]\ar[d]&H^{s}(K,\rhobar)\ar[d]\\
H^{0}\big(\Gal(K'_{\infty}/K_{\infty}), H^s(K'_{\infty},\rhobar)\big)\ar^-{\sim}[r]&
H^{s}(K_{\infty},\rhobar).
}
\end{equation}
The conclusion follows from the result in the upper triangular case.
\end{proof}
A statement similar to Lemma \ref{lemma H1} (via a slightly different argument) has been obtained in \cite[Proposition 6.1]{Gao}.

A similar argument yields the following:
\begin{lemma}
\label{lemma B1}
Let $\rhobar: G_{K}\rightarrow \GL_n(\F)$ be cyclotomic free. Then the natural restriction map $B^1(K,\rhobar)\rightarrow B^1(K_{\infty},\rhobar)$ on Galois cohomology boundaries is an isomorphism.
\end{lemma}
\begin{proof}
The argument follows closely the proof of Lemma \ref{lemma H1} above. Let $V_{\rhobar}$ be the $\F$-linear space underlying $\rhobar$. Then one has
$$
\frac{V_{\rhobar}}{(V_{\rhobar})^{G_{K}}}\stackrel{\sim}{\rightarrow}B^1(K,\rhobar),\qquad
\frac{V_{\rhobar}}{(V_{\rhobar})^{G_{K_{\infty}}}}\stackrel{\sim}{\rightarrow}B^1(K_{\infty},\rhobar);
$$
therefore it is enough to prove that the (obviously injective) restriction map $H^0(K,\rhobar)\rightarrow H^0(K_{\infty},\rhobar)$ is surjective.

We assume first that $\rhobar$ is upper triangular. Let us fix an extension $0\ra \rhobar_1\ra\rhobar\ra\overline{\chi}\ra0$,
where $(\rhobar_1)^{\mathrm{ss}}$ and the character $\overline{\chi}: G_K\ra \F\s$ do not have cyclotomic constituents. The restriction functor to $G_{K_{\infty}}$ and classical group cohomology give us the following commutative diagram, with exact lines: 
\begin{equation*}
\xymatrix{
H^0(K,\rhobar_1)\ar[r]\ar^{f_0}[d]&
H^0(K,\rhobar)\ar[r]\ar^{f_1}[d]&
H^0(K,\overline{\chi})\ar^{f_2}[d]\ar^{\delta}[r]&
H^1(K,\rhobar_1)\ar^{f_3}[d]
\\
H^0(K_{\infty},\rhobar_1)\ar[r]&
H^0(K_{\infty},\rhobar)\ar[r]&
H^0(K_{\infty},\overline{\chi})\ar^{\delta}[r]&
H^1(K_{\infty},\rhobar_1).
}
\end{equation*}
By Lemma \ref{lemma H1}, the morphism $f_3$ is injective; for $i\in\{0,1,2\}$ the morphisms $f_i$ are obviously injective. If $f_0$, $f_2$ are both surjective, the ``four lemma'' again shows that $f_1$ is surjective as well.  Therefore, by d\'evissage, it is enough to show that $H^0(K_{\infty},\overline{\chi}) \neq 0$ if and only if $\overline{\chi}$ is the trivial character of $G_{K}$. This is immediate as all characters are tame.

The deduction for the general case is formal: as in the proof of Lemma \ref{lemma H1}, the hypotheses on $[K':K]$ give us the commutative diagram (\ref{Hoch-Serre ss}). Again, as the natural map $G_{K_{\infty}}/G_{K'_{\infty}}\rightarrow G_{K}/G_{K'}$ is an isomorphism, the isomorphism 
$H^0(K',\rhobar)\stackrel{\sim}{\rightarrow}H^0(K'_{\infty},\rhobar)$ obtained in the upper triangular case respects the residual Galois action on both sides. Therefore 
$H^0\big((\Gal(K'/K),H^0(K',\rhobar)\big)\ia H^0\big((\Gal(K_{\infty}'/K_{\infty}),H^0(K'_{\infty},\rhobar)\big)$ is an isomorphism.
\end{proof}

\begin{prop}
\label{prop:Z1}
Let $\rhobar: G_{K}\rightarrow \GL_n(\F)$ be cyclotomic free.
Then the map on Galois cohomology cycles
$$
Z^1(K,\rhobar)\rightarrow Z^1(K_{\infty},\rhobar)
$$
is injective.
\end{prop}
\begin{proof}
If $\rhobar$ is cyclotomic free, we can apply Lemmas \ref{lemma H1} and \ref{lemma B1} to $\rhobar$ and the claim follows from the snake lemma applied to the defining sequences of $Z^1$:
\begin{equation*}
\xymatrix{
0\ar[r]&\ar[r] B^1(K, \rhobar)\ar[r]\ar[d]&Z^1(K, \rhobar)\ar[r]\ar[d]&H^1(K, \rhobar)\ar[r]\ar[d]&0\\
0\ar[r]&\ar[r] B^1(K_{\infty}, \rhobar)\ar[r]&Z^1(K_{\infty}, \rhobar)\ar[r]&H^1(K_{\infty}, \rhobar)\ar[r]&0.
}
\end{equation*}
\end{proof}

\section{Finite height $K_{\infty}$-deformations}
\label{sec:algorithm}

In this section, $\tau$ will denote a weakly generic tame principal series type (Definition \ref{gencond}).

Let $\overline{\fM} \in Y^{\mu, \tau} (\F)$ with a basis $\overline{\beta}$ as in Theorem \ref{thm:classification}.  We will now compute the deformations of $(\overline{\fM}, \overline{\beta})$ according to the shape of $\overline{\fM}$.  Roughly, we are giving local coordinates for $Y^{\mu, \tau}$ at $\overline{\fM}$. This amounts to giving coordinates for the Pappas-Zhu local model $M(\mu)$ discussed after Definition \ref{defn:shape}, though this won't be used.  The strategy will be to start with an arbitrary lift of $\overline{\fM}$ to a local Noetherian $\cO$-algebra $R$ with finite residue field and then by a convergence process put the Frobenius into a special form where entries are polynomials with coefficients in $R$ with controlled degree. In this special form, it is straightforward to impose the height $[0,2]$ condition as well as a determinant condition. The algorithm  combines the $u$-adic and max adic topologies. For $\GL_2$, a similar strategy was introduced in setting of Breuil modules in \cite{breuil-buzzati} and was implemented for Kisin modules for non-generic types in \cite{CDM}.  
 
\begin{thm} \label{thm:classification1} Let $R$ be a complete local Noetherian $\cO$-algebra with finite residue field $\F$ and let $\tau$ be a weakly generic tame type as in Definition \ref{gencond}. Write $(s_j)_j\in S_3^f$ for the orientation of $\tau$.   Let $\fM \in Y^{\mu, \tau}(R)$ with $\overline{\fM} := \fM\otimes_R\F$ of shape $\bf{w} = (\widetilde{w}_0, \widetilde{w}_1, \ldots, \widetilde{w}_{f-1})$.  Then there exists an eigenbasis $\beta$ for $\fM$ such that for each $0\leq j\leq f-1$ the matrix $\tld{A}^{(j)} =\Mat_{\beta}(\phi_{\fM,s_{j+1}(3)}^{(j)})$  has the form given in row $\widetilde{w}_j$ in Table \ref{table:lifts}.
\end{thm}

\subsection{Algorithm}
\label{subsec:algorithm}

Let $R$ be a complete local Noetherian $\cO$-algebra with maximal ideal $m_R$ and residue field $R/m_R\cong \F$.  Let $P \in R[\![u]\!]$.  For any $r \in R$, let $v_R(r) = \max \{ i \in \N \mid k \geq 0, r \in m_R^k\}$ . This is finite unless $r=0$, by Krull's intersection theorem.

\begin{defn} \label{defndefect} Let $P = \sum_i r_i v^i \in R[\![v]\!]$. Define
$$
d_R(P) = \min_i \{3 v_R(r_i) + i\}. 
$$
\end{defn} 
We define $\Tr_l: R[\![v]\!]\rightarrow R[\![v]\!]$
 to be the order $v^l$-truncation map, defined by $\Tr_l\big(\sum_i r_i v^i\big)\defeq \sum_{i\geq l+1} r_i v^i$.

Given a matrix  $M = (M_{ik}) \in \Mat_n(R[\![v]\!])$, we define
$$
d_R(M) = \min_{i,k}\{d_R(M_{ik})\}
$$
and $\Tr_l(M)\defeq (\Tr_l(M_{ik}))$.


\begin{prop}
\label{semi-valuation}
For any $P, Q \in \Mat_n(R[\![v]\!])$ and any $l\in\N$, we have
\begin{enumerate}
\item $d_R(P + Q)  \geq \min (d_R(P), d_R(Q))$;
\item $d_R(PQ) \geq d_R(P) + d_R(Q)$;
\item $d_R(\Tr_l(P))\geq d_R(P)$.
\end{enumerate}
\end{prop} 
\begin{proof} This follows from $v_R(ab)\geq v_R(a)+v_R(b)$ and $v_R(a+b)\geq \min (v_R(a),v_R(b))$ for $a,b \in R$.
\end{proof} 

\begin{rmk}
In Definition \ref{defndefect}, we could have considered  $d_{R,m}(P) = \min_i \{m v_R(r_i) + i\}$ for $m\geq 3$; \emph{mutatis mutandis}, the statement of Proposition \ref{semi-valuation} still holds  true. 
\end{rmk}

The algorithm proceeds by successive row operations which we introduce now.  For any $x \in R[\![v]\!]$,  we define
$$
U_{12}(x) :=  \begin{pmatrix} 1 & x & 0 \\ 0 & 1 &  0\\ 0& 0  & 1 \end{pmatrix}, U_{13}(x) := \begin{pmatrix} 1 & 0 & x \\ 0 & 1 &  0\\ 0& 0  & 1 \end{pmatrix}, U_{23}(x) := \begin{pmatrix} 1 & 0 & 0 \\ 0 & 1 &  x\\ 0& 0  & 1 \end{pmatrix}.
$$
Similarly, for any $x \in  vR[\![v]\!]$, we define
$$
L_{21}(x) :=  \begin{pmatrix} 1 & 0 & 0 \\ x & 1 &  0\\ 0& 0  & 1 \end{pmatrix}, L_{31}(x) := \begin{pmatrix} 1 & 0 & 0 \\ 0 & 1 &  0\\ x& 0  & 1 \end{pmatrix}, L_{32}(x) := \begin{pmatrix} 1 & 0 & 0 \\ 0 & 1 &  0\\ 0& x  & 1 \end{pmatrix}.
$$
$$
D_{11}(x) :=  \begin{pmatrix} 1+x & 0 & 0 \\ 0 & 1 &  0\\ 0& 0  & 1 \end{pmatrix}, D_{22}(x) := \begin{pmatrix} 1 & 0 & 0 \\ 0 & 1+x &  0\\ 0& 0  & 1 \end{pmatrix}, D_{33}(x) := \begin{pmatrix} 1 & 0 & 0 \\ 0 & 1 &  0\\ 0& 0  & 1+x \end{pmatrix}.
$$

 
The essence of the algorithm is as follows.  
Let $\beta_n= (\beta_n^{(j)})$ be an eigenbasis for $\fM_R$ at the $n$-th step of the algorithm. If $A_n^{(j)}=\Mat_{\beta_n}(\phi_{\fM,s_{j+1}(3)}^{(j)})$, we can always write
$$
A^{(j)}_n = B^{(j)}_{\widetilde{w}_j,n} + E^{(j)}_n
$$
where the entries of $B^{(j)}_{\widetilde{w}_j,n}\in \Mat_3(R[\![v]\!])$ satisfy degree bound conditions according to the shape $\widetilde{w}_j$ of $\overline{\fM}_R^{(j)}$ (the degree bound conditions are listed in Table \ref{table:lifts}; this will be made precise in Definition \ref{definition defect} below.)

We call $E^{(j)}_n$ the \emph{error term} associated to $\beta_n$.  The inductive step is to show that there exists a new eigenbasis $\beta_{n+1}$ with error $E^{(j)}_{n+1}$ such that 
$$
\min_j (d_R(E^{(j)}_{n+1}) )> \min_j (d_R(E^{(j)}_n)).
$$ 

We start with a few definitions.

\begin{defn} 
\label{definition defect}
Let $A^{(j)}\in\Mat_3(R[\![v]\!])$ and 
$\widetilde{w}_j\in\widetilde{W}$ be the shape of $\overline{\fM}$ at $j$. 
Then there exists a unique decomposition in $\Mat_3(R[\![v]\!])$
$$
A^{(j)} = B^{(j)}_{\widetilde{w}_j} + E^{(j)}
$$
such that $B^{(j)}_{\widetilde{w}_j}\in \Mat_3(R[v])$ and for all $i,k$ one has 
$$
\deg_v\big(\big(B^{(j)}_{\widetilde{w}_j}\big)_{ik}\big)=\deg_v\big(\big(\tld{A}^{(j)}_{\widetilde{w}_j}\big)_{ik}\big)<\val_v(E^{(j)}_{ik})
$$ 
where $\deg_v\big(\big(\tld{A}^{(j)}_{\widetilde{w}_j}\big)_{ik}\big)\in\{-\infty,0,1,2\}$ is defined in the third column of Table \ref{table:lifts}.
(In the notations of Table \ref{table:lifts}, an entry of the form $i^*$, $i\in \N$, stands for a polynomial in $R[v]$ of degree $i$ and whose leading coefficient is a unit in $R$; an entry of the form $v(\leq i)$ stands for a polynomial in $R[v]$ of degree at most $i+1$ and which is moreover divisible by $v$; a similar comment applies for entries of the form $\leq i$ and $v(i^*)$.)

The \emph{defect} of $A^{(j)}$ at the entry $(ik)$ is defined as $\delta(A^{(j)}_{ik})\defeq d_R(E^{(j)}_{ik})$.
Similarly, the \emph{total defect} of $A^{(j)}$ is defined as
$$
\delta(A^{(j)})\defeq d_R(E^{(j)}).
$$
\end{defn} 

Typically, the matrix $A^{(j)}$ in definition \ref{definition defect} is either $\Mat_{\beta}(\phi^{(j)}_{\fM,s_{j+1}(3)})$ in some eigenbasis $\beta$ on $\fM_R$, or its modification by row and (adjoint $\phz$-twisted)-column operations by $U_{ik}(x^{(j)}),\,L_{ik}(x^{(j)})$ (cf. Proposition \ref{rightmult}, Proposition \ref{leftoperation}).

Let $(x^{(j)})$ denote an $f$-tuple of elements of $ R[\![v]\!]$.  An \emph{elementary operation} (associated to $(x^{(j)})$) is a change of eigenbasis on $\fM_R$ such that for each embedding $j$, the associated matrix $D^{(j)}$ as in Proposition \ref{changeofbasis} can be written as $D^{(j)}=\Ad_{s_j}(u^{\bf{a}_1},u^{\bf{a}_2},u^{\bf{a}_3}) \left( I^{(j)} \right)$ for some $I^{(j)} \in\{ U_{ik}(x^{(j)}),\,\,L_{ki}(x^{(j)})\ | \,i<k\}\cup\{D_{ii}(x^{(j)})\}$.



The following proposition controls the change in precision after right multiplication by 
$s_{j+1}^{-1}s_j I^{(j),\phz}s_j^{-1}s_{j+1}$ (cf.  (\ref{cob1})) in terms of the precision of an elementary operation $I^{(j)}$. 

\begin{prop}  
\label{rightmult}
Let $\fM \in Y^{[0,2],\tau}(R)$, with eigenbasis $\beta$, and let $A^{(j)} = B^{(j)}_{\widetilde{w}_j} + E^{(j)}$ as in Definition \ref{definition defect}, where we have set
$A^{(j)}\defeq\Mat_{\beta}(\phi^{(j)}_{\fM,s_{j+1(3)}})$.  
Let $(I^{(j)})$ be an elementary operation associated to the $f$-tuple $(x^{(j)})$
and define 
$$
I^{(j),\phz}\defeq 
\Ad(v^{a_{s_j(1),f-1-j}},v^{a_{s_j(2),f-1-j}},v^{a_{s_j(3),f-1-j}})\cdot \phz\big(I^{(j)}\big)^{-1}.
$$ 
Then one has
$$
A^{(j)}  s_{j+1}^{-1}s_jI^{(j),\phz}s_j^{-1}s_{j+1}= B^{(j)}_{\widetilde{w}_j} + E^{\prime (j)} 
$$
where
\begin{enumerate}
\item\label{prop:4.6-1} $d_R(E^{\prime(j)}) \geq \min( \delta(A^{(j)}), 3+d_R(x^{(j)}))$;
\item\label{prop:4.6-2} $d_R(E^{\prime(j)}-E^{(j)})\geq 3+d_R(x^{(j)})$.
\end{enumerate}


\end{prop}
\begin{proof}
We saw in Proposition \ref{divisibility} that 
$$
I^{(j), \phz} = \Id + v^3 X
$$
where $X\in \Mat(R[\![v]\!])$.  The same calculation shows that $d_R(X) \geq d_R(x^{(j)})$.  Therefore
\begin{eqnarray*}
A^{(j)}\cdot  s_{j+1,j}\cdot I^{(j),\phz}\cdot s_{j+1,j}^{-1} = B^{(j)}_{\widetilde{w}_j} + E^{(j)}\left(\Id+v^3s_{j+1,j}Xs_{j+1,j}^{-1}\right)+v^3B^{(j)}_{\widetilde{w}_j}s_{j+1,j}Xs_{j+1,j}^{-1}
\end{eqnarray*}
which immediately implies items (\ref{prop:4.6-1}) and (\ref{prop:4.6-2}) in the statement of the Proposition.
Note that  $ E^{\prime (j)} $ is actually the error term associated to $A^{(j)}\cdot  s_{j+1,j}\cdot I^{(j),\phz}\cdot s_{j+1,j}^{-1}$, since $v^3|E^{\prime (j)}-E^{(j)}$ and the degree bounds appearing in Table \ref{table:lifts} are at most $2$.
\end{proof} 


\begin{cor}
\label{cor:rightmult}
In the setting of Proposition \ref{rightmult}, assume that in the elementary operation $(I^{(j)})$ only the element $I^{(j+1)}$ is not the identity.
Set $(A_1^{(j)})\defeq (A^{(j)})$ and let $(A_2^{(j)})$ be the $f$-tuple obtained by performing the elementary operation $(I^{(j)})$ on $(A_1^{(j)})$.
Then if $f>1$, one has: 
\begin{enumerate}
\item $A_2^{(j)}=I^{(j+1)}A_1^{(j)}$; and
\item $\delta(A_2^{(j+1)})\geq \min(\delta(A_1^{(j+1)}),3+d_R(x^{(j+1)}))$.
\end{enumerate}
If $f=1$, then $d_R(A_2^{(j)}-I^{(j+1)}A_1^{(j)})\geq 3+d_R(x^{(j+1)})$. 
\end{cor}

We introduce the crucial notion of \emph{pivots} associated to a shape: 
\begin{defn}
\label{define pivot}
Let $\widetilde{w}_j\in \widetilde{W}$ be the shape at $j$ of $\overline{\fM} \in Y^{[0,2],\tau}(\F)$.
The \emph{pivots} of $\widetilde{w}_j$ are the pairs $(m, k)$ such that the $(m, k)$-entry of $\widetilde{w}_j\in N_{\GL_3}(T)(\F(\!(v)\!))$ is non-zero.

Let $\fM \in Y^{[0,2],\tau}(R)$ be a Kisin module and let $\widetilde{w}_j$ be the shape of $\overline{\fM} \defeq \fM \otimes_{R}\F$ at $j$. If  $A^{(j)}=\Mat_{\beta}(\phi^{(j)}_{\fM, s_{j+1}(3)})$ with respect to an eigenbasis $\beta$, and such that  $\overline{A}^{(j)}$ is given by the second column of Table \ref{table:lifts}, we say that the pair $(m,k)\in\{1,2,3\}^2$ is a \emph{pivot of $A^{(j)}$} if $(m,k)$ is a pivot of the shape $\widetilde{w}_j$.
We define the \emph{degree} of the pivot $(m,k)$ to be $\deg_v(\overline{A}^{(j)}_{mk})$. 
\end{defn}

\begin{rmk} \label{degreeintable} We explain how to produce the bounded degree conditions in the third column of Table \ref{table:lifts}, starting from the position and the degree of the pivots. (Note that the pivots of the matrices in the third column of Table \ref{table:lifts} are exactly the starred entries.)
Heuristically the effect of change of eigenbases on $(A^{(j)})$ is very close to left-multiplication by elements in $(\Iw_1(R))^f$.
If we literally use left multiplication, we would be able to make $A^{(j)}$ have polynomial entries and moreover if $(m,k)$ is a pivot of $A^{(j)}$ of degree $i$, we see that any entry above $(m,k)$ has degree strictly less than $i$ and any entry below $(m,k)$ has degree at most $i$.  
All the strict lower triangular entries of $A^{(j)}$ must be divisible by $v$ by construction.
\end{rmk} 

The key lemma that enables us to control the convergence is the following:     
\begin{lemma} \label{otherentries}  
Keep the notation as in Definition \ref{define pivot}. Assume that the eigenbasis $\beta$ lifts a gauge basis $\overline{\beta}$ of $\overline{\fM}_R$ $($cf. Definition $\ref{definition gauge basis mod p}$; in particular $\overline{A}^{(j)}=\overline{A}^{(j)}_{\tld{w}_j}$ as in Table $\ref{table shapes mod p})$.
 
Let $(m, k)$ be a pivot for $A^{(j)}$ of degree $i$. For all $k' \geq  k$, we can write 
$$
A^{(j)}_{mk'} = v^i P_{mk'}  + Q_{mk'}
$$
where all the coefficients of $Q_{mk'}$ lie in the maximal ideal of $R$.   For all $k' < k$, we can write 
$$
A^{(j)}_{mk'} = v^{i+1} P_{mk'}  + Q_{mk'}
$$
where all the coefficients of $Q_{mk'}$ lie in the maximal ideal of $R$. 

In particular, if $k' \geq  k$, then $d_R(A^{(j)}_{mk'}) \geq  i$ whereas if $k' < k$ then $d_R(A^{(j)}_{mk'}) \geq  i + 1$.  
\end{lemma}
\begin{proof}
Modulo $m_R$, the matrices $\overline{A}^{(j)}$ are of the form $\tilde{w}_j \cI$, and thus every entry to the left (in the same row) of a pivot of degree $i$ is divisible by $v^{i+1}$, while every entry to the right is divisible by $v^i$.
The last statement follows noting that $i\leq 2$ and $d_R(Q_{mk'})\geq 3$.
\end{proof}

The following Proposition shows that by suitable row operations via the elementary matrices $U_{ik}(x^{(j)}),\,L_{ik}(x^{(j)}), D_{ii}(x^{(j)})$, we can strictly increase the defect of an entry of $A^{(j)}$ \emph{without decreasing the total defect of $A^{(j)}$}.

\begin{prop} 
\label{leftoperation} 
Keep the notations and assumptions of Lemma \ref{otherentries}.
Assume that $(m,k)$ is a pivot of $A^{(j)}$.
There exists $x\in R[\![v]\!]$ such that, by letting
\begin{eqnarray*}
A^{\prime, (j)}\defeq \left\{\begin{array}{cc}
U_{m'm}(x) A^{(j)}\,\, \text{ if }\,\, m' < m,\\
D_{mm}(x) A^{(j)}\,\, \text{ if }\,\, m' = m,\\
L_{m'm}(x) A^{(j)}\,\, \text{ if }\,\, m' > m,
\end{array}\right.
\end{eqnarray*}
one has $\delta(A^{\prime, (j)})\geq\delta(A^{(j)})$ and moreover $\delta(A^{\prime, (j)}_{m'k})>\delta(A^{(j)}_{m'k})$, $\delta(A^{\prime, (j)}_{rs})\geq \min(\delta(A^{(j)}_{rs}),\delta(A^{(j)})+1)$ unless $r=m',s>k$.
\end{prop} 
\begin{proof} 

Let us write $A^{(j)}=B_{\widetilde{w}_j}^{(j)}+E^{(j)}$ as in Definition \ref{definition defect} and let $\delta = \delta(A^{(j)})$ be the total defect of $A^{(j)}$.   
Let $i \in \{0,1,2\}$ be the degree of the pivot of $A^{(j)}$ at $(m, k)$.
As $\overline{A}^{(j)}_{mk}\in \F[\![v]\!]$ is a monomial in $v$ (cf. Definition \ref{define pivot}), we can write $A^{(j)}_{mk}=u_{mk}v^i+Q_{mk}$ for some unit $u_{mk}\in R^{\times}$ and some element $Q_{mk}\in R[\![v]\!]$ verifying $d_R(Q_{mk})\geq 3$. 

Let us consider the case $m' < m$. 
By the definition of the error term $E^{(j)}$, we have $E^{(j)}_{m'k}\in v^iR[\![v]\!]$. In particular, we can write  
$E^{(j)}_{m'k} = v^i P_{m'k}$ for some $P_{m'k}\in R[\![v]\!]$ verifying 
$d_R(P_{m'k})=d_R(E^{(j)}_{m'k}) - i \geq \delta - i$.
We set $x \defeq -u_{mk}^{-1} P_{m'k}$. We have $A^{\prime, (j)}_{m'k}=(B_{\widetilde{w}_j}^{(j)})_{m'k}+xQ_{mk}$. 
Letting $E^{\prime, (j)}$ be the error term of $A^{\prime, (j)}$, by  Proposition \ref{semi-valuation}, we have 
\begin{eqnarray*}
d_R(E^{\prime, (j)}_{m'k})&\geq& d_R(Q_{mk})+d_R(P_{m'k})\\
&\geq&3+\delta - i>\delta.
\end{eqnarray*} 

We now verify that $\delta(A^{\prime, (j)})\geq \delta$.
Indeed, we have $A^{\prime, (j)}_{ik'}=A^{(j)}_{ik'}$ for all $i\neq m'$ and $1\leq k'\leq 3$.

If $i=m'$ and $k'\neq k$, we have
$$
A^{\prime, (j)}_{m'k'}=A^{(j)}_{m'k'} - u_{mk}^{-1} P_{m'k}  A^{(j)}_{mk'}.
$$
By Lemma \ref{otherentries}, we conclude that $d_R(P_{m'k}  A^{(j)}_{mk'}) \geq i + d_R(P_{m'k}) \geq \delta$ and that the inequality is strict unless $k'>k$.
This completes the proof in the case $m'<m$. The other cases are similar.
\end{proof}  

\begin{rmk}
\label{rmk:crucial:algorithm}
 The element $x \in R[\![v]\!]$ used in the proof of Proposition \ref{leftoperation} always has the property that $d_R(x) \geq\delta(A^{(j)}) - 2$, since $i \leq 2$. 
\end{rmk}    

\begin{prop} 
\label{improvedefect}
Let $\fM \in Y^{[0,2],\tau}(R)$ and let $\overline{\beta}$ be a gauge basis of $\overline{\fM}$.
Let ${\beta}$ be an eigenbasis of $\fM$ lifting $\overline{\beta}$ and for all $0\leq j\leq f-1$ set $A^{(j)}=\Mat_{{\beta}}(\phi_{\fM,s_{j+1}(3)}^{(j)})$.  
There exists another eigenbasis $\beta^{\prime}$ lifting $\overline{\beta}$ 
such that 
$$
\min_j \delta(A^{\prime, (j)}) > \min_j \delta(A^{(j)})
$$
for all $0\leq j\leq f-1$, where $A^{\prime, (j)}\defeq\Mat_{{\beta}'}(\phi_{\fM,s_{j+1}(3)}^{(j)})$.
Furthermore, if 
$$
\Ad_{s_j}^{-1}(u^{\bf{a}_1},u^{\bf{a}_2},u^{\bf{a}_3}) \left(D^{(j)} \right) = I^{(j)}
$$
as in Proposition \ref{changeofbasis}, then $d_R(I^{(j)}) \geq \delta(A^{(j)}) - 2$
\end{prop}
\begin{proof}
Let $(m_1,1)$, $(m_2,2)$ and $(m_3,3)$ be the pivot entries for $A^{(j)}$, and put $\delta=\min_j \delta(A^{(j)})$.  
We consider first the case $f>1$.
Using the operations as in 
Corollary \ref{cor:rightmult} with $I^{(j+1)}$ given by the matrices in Proposition \ref{leftoperation} for the pivot $(m_1,1)$, we can find a change of basis such that with respect to the new basis, the matrix $A^{(j)}$ will have entries in its first column of defect $>\delta$. Apply the same argument for the second and the third column, we can make also the second and third column entries have defect $>\delta$
noting that by performing the elementary operations in this order, the last part of Proposition \ref{leftoperation} guarantees that we do not lose the increased defect of an entry of $A^{(j)}$ that was already made to have defect $>\delta$.
During this process, Corollary \ref{cor:rightmult} and Remark \ref{rmk:crucial:algorithm} show that even though $\delta(A^{(j+1)})$ may decrease, whenever it decreases then the decreased value is automatically $\geq \delta+1$.
Thus, by performing this process for each $j$, we arrive at an eigenbasis satisfying the first part of the Proposition.

The claim on $d_R(I^{(j)})$ follows now by Remark \ref{rmk:crucial:algorithm}.
\end{proof}  
 
 \begin{lemma}  Let $(x_{\ell})_{\ell \geq 1}$ be elements of $R[\![v]\!]$.   If $\lim_{\ell \ra \infty} d_R(x_{\ell}) = \infty$, then there exists $x \in R[\![v]\!]$ such that $x = \sum_{\ell = 1}^{\infty} x_{\ell}$.  
 \end{lemma}  
 \begin{proof} This is because $R[\![v]\!]$ is $(m_R,v)$-adically complete.
 \end{proof}
 
 \begin{proof}[Proof of Theorem \ref{thm:classification1}:] 
By a repeated application of Proposition \ref{improvedefect}, we can find a sequence of bases whose change of basis matrix converge by the above Lemma. Taking the limit change of basis matrix produces an eigenbasis with respect to which $A^{(j)}$ has the desired form.
 \end{proof}

\subsection{Gauge basis}
\label{subsec:GaugeBasis}

We introduce the crucial notion of \emph{gauge basis} for a Kisin module $\fM \in Y^{[0,2], \tau}(R)$, and study some of its properties.

\begin{defn} \label{gauge} Let $R$ be a complete local Noetherian $\cO$-algebra and let $\fM \in Y^{[0,2], \tau}(R)$ lifting $\overline{\fM}$. An eigenbasis $\beta$ lifting $\overline{\beta}$ is called a \emph{gauge basis} if the matrix $\tld{A}_{\tld{w}_j}^{(j)} \defeq \Mat_{\beta}(\phi_{\fM,s_{j+1}(3)}^{(j)})$  satisfies the degree conditions in the third column, row $\widetilde{w}_j$ in Table \ref{table:lifts}.
\end{defn}

We now consider the question of the uniqueness of the gauge basis constructed by the algorithm from the previous section. While the basis is not unique, it is unique up to component-wise scaling by a torus.  Let $\fM \in Y^{[0,2], \tau}_{\overline{\fM}}(R)$. Any eigenbasis $\beta$ for $\fM$ induces an eigenbasis on $\fM/u\fM$ (i.e., a basis for $\fM^{(j)}/u\fM^{(j)}$ for each $j$ compatible with the linear action of descent datum).  We denote this by $\beta \mod u$. 

\begin{thm} \label{gaugeunique} Let $\fM, \overline{\fM}, \overline{\beta}$ be as in Definition \ref{gauge}.  The map 
$$
\beta \mapsto \beta \mod u
$$
induces a bijection between gauge bases of $\fM$ and eigenbases of $\fM/u\fM$ lifting $\overline{\beta} \mod u$. 
\end{thm}  

The key consequence of Theorem \ref{gaugeunique} which we will use in the next section is that the addition of a gauge basis is a formally smooth operation.
  
\begin{proof}
Given a gauge basis $\beta = ( \beta^{(j)} )$, scaling any $\beta^{(j)}$ by the diagonal torus $T(R)$ gives a new gauge basis.   Hence, the map is surjective.   

It suffices then to show that if $\beta_1$ and $\beta_2$ are two gauge bases such that 
\begin{equation}
\label{condition mod u}
\beta_1  \mod u = \beta_2 \mod u
\end{equation}
then $\beta_1 = \beta_2$.  

Let us write $\tld{A}^{(j)}_i\defeq \Mat_{\beta_i}(\phi_{\fM,s_{j+1}(3)})$ for $i=\{1,2\}$ (we omit the subscript $\tld{w}_j$ to ease notation). Then the change of basis formula (\ref{cob1}) gives us
\begin{equation*}
\label{cob1*} 
\tld{A}^{(j)}_2 s_{j+1}^{-1} s_j \big(\Ad(v^{a_{s_{j}(1), f-j - 1}}, v^{a_{s_{j}(2), f-j - 1}}, v^{a_{s_{j}(3), f-j -1}})\cdot \phz(\Id_3+I^{(j)})\big) s_j^{-1} s_{j+1}= (\Id_3+I^{(j+1)})\tld{A}^{(j)}_1
\end{equation*}
where all entries of $I^{(j)}$ which are on or below the diagonal are divisible by $v$.
By the weak genericity assumption, we see as in the proof of Proposition \ref{rightmult} that
$$
s_{j+1}^{-1} s_j \big(\Ad(v^{a_{s_{j}(1), f-j - 1}}, v^{a_{s_{j}(2), f-j - 1}}, v^{a_{s_{j}(3), f-j -1}})\cdot \phz(I^{(j)})\big) s_j^{-1} s_{j+1}= v^3 M^{(j)}
$$
where $M^{(j)}\in \Mat_3(R[\![v]\!])$ verifies $d_R(M^{(j)})\geq d_R(I^{(j)})$.
We obtain:
\begin{equation}
\label{key equation for gauge}
\tld{A}_2^{(j)}+v^3\tld{A}_2^{(j)}M^{(j)}=\tld{A}_1^{(j)}+I^{(j+1)}\tld{A}_1^{(j)}.
\end{equation} 

From equation (\ref{key equation for gauge}), we now show that for all $n\in \N$, $d_R(I^{(j)})\geq n$ for all $j=0,\dots,f-1$, i.e., that $I^{(j)}=0$, for all $j=0,\dots,f-1$.
Suppose we have $d_R(I^{(j)})\geq \delta$ for all $j$.

Set $\overline{A}^{(j)}\defeq \tld{A}_1^{(j)}\otimes_{R}{\F}=\tld{A}_2^{(j)}\otimes_{R}{\F}$.
We define a pivot $(k(1),m(1))\in\{1,2,3\}^2$ of degree $i(1)$ (cf. Definition \ref{define pivot})  via the requirement that $\overline{A}_{km(1)}=0$ for all $k\neq k(1)$ and $i(1)$ is minimal among the degrees of the pivots of $\overline{A}^{(j)}$. Similarly, we define a pivot $(k(2),m(2))\in\{1,2,3\}^2$ of degree $i(2)$ via the requirement that $\overline{A}_{km(2)}=0$ for all $k\neq k(1),\ k(2)$ and $i(2)$ is minimal among the degrees of the pivots of $\overline{A}^{(j)}$ which are different from $(k(1),m(1))$.
We write $(k(3),m(3))$ for the remaining pivot, of degree $i(3)$.  Table \ref{table shapes mod p} shows that a choice of pivots like this exists, because each $\overline{A}^{(j)}$ is obtained from an upper triangular matrix by permuting rows and columns. Note that $(i(1),i(2),i(3))=(0,1,2)$ or $(1,1,1)$.

For instance, in shape $\alpha\beta\alpha$, we have $(k(1),m(1))=(3,1)$, $(k(2),m(2))=(2,2)$ and $(k(3),m(3))=(1,3)$ and they all have degree $1$. 

For $l\in\{1,2\}$, we have $d_{R}((\tld{A}_1^{(j)})_{km(l)})\geq i(l)$ for all $k\in\{1,2,3\}$. Furthermore,  $d_{R}((\tld{A}_1^{(j)})_{km(1)})\geq 3$ if $k\neq k(1)$ and  $d_{R}((\tld{A}_1^{(j)})_{km(2)})\geq 3$ for $k\neq k(1), k(2)$.
 
If $i(3)=1$, then one still has $d_{R}((\tld{A}_1^{(j)})_{km(3)})\geq i(3)$ for all $k$ but, when $i(3)=2$ then one loses precision and we just have $d_{R}((\tld{A}_1^{(j)})_{km(3)})+1\geq i(3)$ for $k\neq k(3)$. Moreover, since a pivot reduces to a monomial modulo the maximal ideal of $R$, we have $(\tld{A}_1^{(j)})_{k(l)m(l)}=x_l^*v^{i(l)}+E_l$ where $d_R(E_{l})\geq 3$ and $x^*_l\in R^\times$.

For all $n\in\{1,2,3\}$, we now compare the $nm(1)$-th entry of equation (\ref{key equation for gauge}).  We use $\Tr_{s}$ for truncation (cf. Proposition \ref{semi-valuation})  which deletes the terms of degree $\leq s$. Taking $s=i(1)-\delta_{n<k(1)}$ (the degree at the entry $(nm(1))$), one has
\begin{align*}
\text{\small{$I^{(j+1)}_{nk(1)}x^*_1v^{i(1)}+I^{(j+1)}_{nk(1)}E_1 +\Tr_{s}\big(I^{(j+1)}_{nk(2)}(\tld{A}_1^{(j)})_{k(2)m(1)}\big)+
\Tr_{s}\big(I^{(j+1)}_{nk(3)}(\tld{A}_1^{(j)})_{k(3)m(1)}\big)=v^3\big(\tld{A}_2^{(j)}M^{(j)} \big)_{nm(1)}$}}.
\end{align*}
 Here we use that the truncation kills off the contribution of $\tld{A}_2^{(j)}-\tld{A}_1^{(j)}$, and that $v^{s-i(1)}| I^{(j+1)}_{nk(1)}$.
 Since every term in the equation except the leftmost term has $d_R\geq \delta+3$, we conclude that $d_R(I^{(j+1)}_{nk(1)})\geq \delta+2$.
 
 Similarly, by comparing the $nk(2)$ entries and truncating, using that  $d_R(I^{(j+1)}_{nk(1)})\geq \delta+2$, we also have $d_R(I^{(j+1)}_{nk(2)})\geq \delta+2$.
 Finally, comparing the $nk(3)$ entries and truncating, and using $d_R(I^{(j+1)}_{nk(l)})\geq \delta+2$ for $l=1,2$, we get  $d_R(I^{(j+1)}_{nk(3)})\geq \delta+1$ (note the loss of -1 in the lower bound for $d_R(I^{(j+1)}_{nk(3)})$, which is due to the weaker estimate  $d_{R}((\tld{A}_1^{(j)})_{km(3)})+1\geq i(3)$).
\end{proof}

\subsection{Height conditions} 


Let $\overline{\fM}\in Y^{\mu,\tau}_{\mathbf{w}}(\F)$.  We now compute the universal lift of $\overline{\fM}$ with height conditions. Fix a gauge basis $\overline{\beta}$ mod $p$ of $\overline{\fM}$ (Definition \ref{definition gauge basis mod p}).

We consider the problem of deforming $(\overline{\fM}, \overline{\beta})$. Recall the closed substack  $Y^{\mu, \tau} \subset Y^{[0,2], \tau}$ introduced in \S 3.1 and constructed in \cite[Proposition 5.2]{CL}. For any Artinian $\cO$-algebra $A$ with residue field $\F$, let $D^{\tau, \overline{\beta}}_{\overline{\fM}}(A)$ be the category of pairs $(\fM_A, \beta_A)$ deforming $(\overline{\fM}, \overline{\beta})$ where $\fM_A \in Y^{\mu, \tau}(A)$ and $\beta_A$ is a gauge basis of $\fM_A$. 
By Theorem \ref{gaugeunique}, the morphism $D^{\tau, \overline{\beta}}_{\overline{\fM}}\rightarrow Y_{\overline{\fM}}^{\mu, \tau}$ is a torsor for $\widehat{\mathbb{G}}_m^{3f}$.

The main result of this subsection is the following:

\begin{thm} \label{univfh} The deformation problem $D^{\tau, \overline{\beta}}_{\overline{\fM}}$ is representable by a complete local Noetherian $\cO$-algebra $R^{\tau, \overline{\beta}}_{\overline{\fM}}$.
Let $(\fM^{\univ}, \beta^{\univ})$ be the universal family over $R^{\tau, \overline{\beta}}_{\overline{\fM}}$.   Then $\Mat_{\beta^{\univ}}(\phi_{\fM^{\univ},s_{j+1}(3)}^{(j)})$ is given in column 4 of Table \ref{table:lifts}.  Furthermore, 
\begin{equation}
\label{eq:def:univ}
R^{\tau, \overline{\beta}}_{\overline{\fM}} \cong \widehat{\otimes}_j (R_{\widetilde{w}_j}^{\expl})^{p\text{-flat, red}}
\end{equation}
where $R_{\widetilde{w}_j}^{\expl}$ is given in the second column of Table \ref{table3} and $(R_{\widetilde{w}_j}^{\expl})^{p\text{-flat, red}}$  denotes its $p$-flat and reduced quotient.
\end{thm} 

In order to prove Theorem \ref{univfh}, we need the following preliminary result from \cite{CL}.
Recall that the $p$-adic Hodge type $\leq(2,1,0)$ condition is imposed by flat closure from the generic fiber; Theorem 5.13 and Corollary 5.12 in \cite{CL} give a characterization of points of $Y^{\tau, \mu}$ for $p$-flat and reduced $\cO$-algebras $R$.  In our setting, this translates into the following:
\begin{prop} \label{lemma height-determinant} Let $R$ be a complete local Noetherian flat reduced $\cO$-algebra. Consider $\fM_{R}\in Y^{[0,h], \tau}(R)$ for some $h$ and let $A^{(j)}=\Mat_{{\beta}}\left(\phi_{\fM,s_{j+1(3)}}^{(j)}\right)$ for any eigenbasis $\beta$ of $\fM_R$.  Then $\fM_R \in Y^{\mu, \tau}(R)$ if and only if 
\begin{enumerate}
\item  $\det (A^{(j)}) = x^*_j P(v)^3$ for $x^*_j \in R[\![v]\!]^{\times}$; 
\item $P(v)^2 (A^{(j)})^{-1} \in \Mat_3(R[\![v]\!][1/p])$
\end{enumerate}   
$($recall that $P(v)=v+p$$)$.
\end{prop}

We explain how \ref{lemma height-determinant} was used to generate column 4 of Table \ref{table:lifts} and hence column 2 of Table \ref{table3}.
Letting $\widetilde{A}^{(j)}_{\tld{w}_j}$ be the universal matrix lifting $\overline{A}^{(j)}$ which satisfies the degree conditions in the third column, row $\tld{w}_j$ of Table \ref{table:lifts}, then $R_{\widetilde{w}_j}^{\expl}$ is
obtained by first imposing the conditions
\begin{enumerate}
	\item[$i)$] for all $1\leq i,k\leq 3$ the $(ik)$-minor satisfies $\left(\widetilde{A}^{(j)}_{\tld{w}_j} \right)^{(ik)}\equiv 0$ modulo $P(v)$;
	\item[$ii)$] $\det\left(\widetilde{A}^{(j)}_{\widetilde{w}_j}\right)=x_j^*P(v)^3$.
\end{enumerate}
and then performing a partial $p$-saturation process of the relations.  For example, whenever we have the condition $(v+p)\mid v^k Q(v)$, we actually get $(v+p)\mid Q(v)$. The resulting $R_{\widetilde{w}_j}^{\expl}$ in rows from $\alpha\beta\alpha\gamma$ to $\alpha\beta\gamma$ in Table  \ref{table:lifts} are $p$-flat and reduced. We do not claim that the remaining rings are $p$-flat and reduced, and we will not need that information. 

\begin{proof}[Proof of Theorem \ref{univfh}]
Since $D^{\tau, \overline{\beta}}_{\overline{\fM}}$ is a formal torus torsor on $Y_{\overline{\fM}}^{\mu, \tau}$ and has no non-trivial automorphisms, we deduce that $D^{\tau, \overline{\beta}}_{\overline{\fM}}$ is representable.
The representing ring $R^{\tau, \overline{\beta}}_{\overline{\fM}}$ is $p$-flat and reduced, as it is equisingular to a complete local ring of $M(\mu)$.

Let $R$ denote the right hand side of (\ref{eq:def:univ}). Since $R$ is $p$-flat and reduced, the relations in $R$ implies that the obvious Kisin module $\fM$ with gauge basis over $R$ is actually inside $Y_{\overline{\fM}}^{\mu, \tau}$ by Proposition \ref{lemma height-determinant}. (The Kisin module $\fM$ is defined as the unique Kisin module with descent data of type $\tau$ endowed with an eigenbasis $\beta$ such that $\Mat_{{\beta}}\left(\phi_{\fM,s_{j+1(3)}}^{(j)}\right)=\tld{A}_{\widetilde{w}_j}^{(j)}$. Note that  $\fM$ has finite height since by construction it satisfies the determinant condition.)

Thus there exists a unique map $g:R^{\tau, \overline{\beta}}_{\overline{\fM}}\rightarrow R$ such that $g_*(\fM^{\univ}, \beta^{\univ})=(\fM,\beta)$.
On the other hand the definition of the gauge basis, the elementary divisors condition satisfied by $\fM^{\univ}$ and the fact that $R^{\tau, \overline{\beta}}_{\overline{\fM}}$ is reduced and $p$-flat show that there is a map $h:R\rightarrow  R^{\tau, \overline{\beta}}_{\overline{\fM}}$ such that $(\fM^{\univ}, \beta^{\univ})=h_*(\fM,\beta)$.
One easily checks that the maps $g$ and $h$ are inverse of each other.

\end{proof}
\begin{rmk}
Under the hypotheses of Theorem \ref{univfh}, it can be shown that $Y^{\mu,\tau}_{\overline{\fM}}$ has dimension $4f$ over $\cO$ which implies that $R^{\tau, \overline{\beta}}_{\overline{\fM}}$ has dimension $7f$ over $\cO$. However, we will not need this information in this paper.
\end{rmk}

We end this section by giving some sample computations of the partial $p$-saturation process mentioned above:
\subsubsection{The $\alpha\beta\alpha$ cell.}

Assume that $\overline{\fM}^{(j)}$ has shape $\widetilde{w}_{j}=\alpha\beta\alpha$. From Theorem \ref{thm:classification1}, we deduce that 
\begin{equation*}
\widetilde{A}^{(j)}_{\alpha\beta\alpha}=\begin{pmatrix}
c_{11}&c_{12}& c_{13}+P(v)c_{13}^*\\
0&\widetilde{c}_{22} + P(v)c_{22}^*&\widetilde{c}_{23}+P(v)c_{23}\\
c_{31}^*v&c_{32}v&c_{33}+P(v)c_{33}'
\end{pmatrix}
\end{equation*}
where $c_{13}^*,\,c_{31}^*,\,c_{22}^*$ are units.

Let us consider first condition $i)$. The congruence $\left(\tld{A}^{(j)}_{\alpha\beta\alpha}\right)^{(13)}\equiv 0$ produces (after killing off an extra $v$ factor) $\widetilde{c}_{22}=0$ and, similarly, 
$\left(\tld{A}^{(j)}_{\alpha\beta\alpha}\right)^{(12)}\equiv 0$ implies $\widetilde{c}_{23}=0$. This implies in particular that $\left(\tld{A}^{(j)}_{\alpha\beta\alpha}\right)^{(11)}\equiv 0$ and $\left(\tld{A}^{(j)}_{\alpha\beta\alpha}\right)^{(3k)}\equiv 0$ for all $k=1,2,3$.

Similarly, we deduce from $\left(\tld{A}^{(j)}_{\alpha\beta\alpha}\right)^{(2k)}\equiv 0$ that
\begin{eqnarray*}
c_{12}c_{33}=-pc_{32}c_{13},&
c_{11}c_{33}=-pc_{31}^*c_{13},
&c_{11}c_{32}=c_{31}^*c_{12}
\end{eqnarray*}
(for $k=1,2,3$ respectively).

Now condition $ii)$ becomes equivalent to
\begin{eqnarray*}
&&(c_{11}c_{33}+pc_{31}^*c_{13})+P(v)\left(c_{33}'c_{11}-c_{31}^*c_{13}+pc_{31}^*c_{13}^*\right)-c_{31}^*c_{13}^* P(v)^2 =x^*P(v)^2
\end{eqnarray*}
which implies that
\begin{eqnarray*}
&&P(v)\left(c_{33}'c_{11}-c_{31}^*c_{13}+pc_{31}^*c_{13}^*\right)-c_{31}^*c_{13}^* P(v)^2=x^*P(v)^2.
\end{eqnarray*}

We conclude that conditions $i)$ and $ii)$ and the partial $p$-saturation process above implies the following relations
\begin{eqnarray*}
c_{12}c_{33}=-pc_{32}c_{13},&
c_{11}c_{33}=-pc_{31}^*c_{13},&
c_{11}c_{32}=c_{31}^*c_{12}\\
&c_{33}'c_{11}-c_{31}^*c_{13}+pc_{31}^*c_{13}^*=0&
\end{eqnarray*}
On the other hand, these relations imply that conditions $i)$ and $ii)$ are satisfied. This explains the fourth column, $\alpha\beta\alpha$-row in the table \ref{table:lifts}. 

\subsubsection{The $\beta\alpha$ cell.}
Assume that $\widetilde{w}_{j}=\beta\alpha$. We have
\begin{eqnarray*}
\tld{A}_{\beta\alpha}^{(j)}=\begin{pmatrix}
c_{11}& c_{12}+P(v)c_{12}^* &c_{13}\\
0&c_{22} + P(v)c_{22}'&c_{23}+ P(v) c_{23}^*\\
vc_{31}^*&vc_{32}&c_{33}+P(v)c_{33}'
\end{pmatrix}
\end{eqnarray*}
where $c_{12}^*,\,c_{23}^*,\,c_{31}^*$ are units.

We consider first condition $i)$. 
From $\left(\tld{A}^{(j)}_{\beta\alpha}\right)^{(12)}\equiv 0$ and $\left(\tld{A}^{(j)}_{\beta\alpha}\right)^{(13)}\equiv 0$, we deduce $c_{23}=0$ and $c_{22}=0$ respectively.

These equations imply $\left(\tld{A}^{(j)}_{\beta\alpha}\right)^{(11)}\equiv 0$ and $\left(\tld{A}^{(j)}_{\beta\alpha}\right)^{(3k)}\equiv 0$ are automatically satisfied for all $k=1,2,3$. From 
$\left(\tld{A}^{(j)}_{\beta\alpha}\right)^{(2k)}\equiv 0$, we deduce, for $k=2,3$ respectively,
\begin{eqnarray}
\label{minors beta-alpha}
c_{11}c_{33}=-pc_{31}^*c_{13},\,&c_{11}c_{32}=c_{31}^*c_{12}&
\end{eqnarray}
These relations together with the fact $c_{31}^*$ being a unit implies $\left(\tld{A}^{(j)}_{\beta\alpha}\right)^{(21)}\equiv 0$.

As for the determinant condition, we obtain:
\begin{eqnarray*}
&&P(v)\left(c_{22}'(c_{11}c_{33}+pc_{31}^*c_{13})+pc_{23}^*(c_{32}c_{11}-c_{31}^*c_{12})\right)+\\
&&\quad P(v)^2\left(c_{11}c_{22}'c_{33}'+c_{12}c_{23}^*c_{31}^*-pc_{31}^*c_{12}^*c_{23}^*-c_{31}^*c_{22}'c_{13}-c_{11}c_{32}c_{23}^*\right)+\\
&&\quad\quad c_{12}^*c_{23}^*c_{31}^* P(v)^3 =x^*P(v)^3
\end{eqnarray*}
which gives the equation
\begin{eqnarray}
\label{determinant beta-alpha}
c_{11}c_{22}'c_{33}'+c_{12}c_{23}^*c_{31}^*-pc_{31}^*c_{12}^*c_{23}^*-c_{31}^*c_{22}'c_{13}-c_{11}c_{32}c_{23}^*=0.
\end{eqnarray}

As $c_{11}c_{32}c_{23}^*=c_{31}^*c_{12}c_{23}^*$, the equations (\ref{minors beta-alpha}), (\ref{determinant beta-alpha}) yield precisely the conditions appearing in the fourth column, $\beta\alpha$-row in the table \ref{table:lifts}. Conversely, these relations imply that conditions $i)$ and $ii)$ are satisfied.

The computations for the other cells are analogous and left to the reader.

\section{Monodromy and potentially crystalline deformation rings}
\label{sec:Monodromy and PCDR}

In the previous section, we essentially computed certain finite height $G_{K_{\infty}}$ Galois deformation rings.  We will now describe (framed) potentially crystalline deformation rings $R^{(2,1,0), \tau}_{\rhobar}$ of $p$-adic Hodge type $(2,1,0)$ at each embedding and Galois type $\tau$.   The codimension of  $\Spec R^{(2,1,0), \tau}_{\rhobar}[1/p]$ in the finite height $G_{K_{\infty}}$-deformation space is $f$, the difference being the existence of a monodromy operator (cf. \cite{KisinFcrys}).  We describe this condition explicitly in Theorem \ref{moncond}.  In most cases, it can be described by $f$ equations on the generic fiber (one for each embedding of $K$ into $\Qpbar$).  Although the equations involve power-series, they can be expressed as polynomial conditions plus a transcendental part which is divisible by a high power of $p$ (due to the genericity condition).  

In section \S 5.3, we obtain (in most cases) integral equations for the deformation rings by analyzing the $p$-flatness properties of these equations.  As a result, we obtain descriptions of the special fibers of the deformation spaces.  In \S \ref{sec:appl}, we use these descriptions to prove instances of the Serre weight conjectures and modularity lifting.             

\subsection{Monodromy condition} 

We begin by recalling some notations from \cite{KisinFcrys}.  Let $\cO^{\rig}$ denote the ring of rigid analytic functions on the open unit disc over $K$.  We fix an embedding $\cO^{\rig} \iarrow K[\![u]\!]$, i.e. identify $\cO^{\rig}$ with the ring of power series $\sum_{i = 0}^{\infty} a_n u^n$ where $a_n \in K$ verify $|a_n|_p r^n \ra 0$ for all $r < 1$ (and hence $\fS[1/p]$ is identified with the subring of bounded functions on the open unit disc).  
Set 
$$
\lambda = \prod_{n=0}^{\infty} \phz^n \left(\frac{E(u)}{p} \right) \in \cO^{\rig}.
$$ 
We define a derivation on $\cO^{\rig}$ by $N_{\nabla} \defeq - u \lambda \frac{d}{du}$; the Frobenius on $\fS$ extends to a Frobenius $\phz$ on $\cO^{\rig}$. If $\La$ is a finite flat $\cO$-algebra, we define  $\cO_{\La}^{\rig}\defeq \cO^{\rig} \otimes_{\Zp}\La$. For any Kisin module $\fM_{\La}\in Y^{[0,2],\tau}(\Lambda)$, we define its base change to $\cO^{\rig}$ as $\fM^{\rig}_{\La} \defeq \fM_{\La} \otimes_{\fS} \cO^{\rig}$.  We have a decomposition $\fM^{\rig}_{\La}=\oplus_{j=0}^{f-1} \fM^{\rig, (j)}_{\La}$.

One has the following important result:
\begin{thm} \label{thmKisin} 
The module $\fM^{\rig}_{\La}[1/\lambda]$ is equipped with a canonical derivation $N_{\fM^{\rig}_{\La}}$ over $N_\nabla$ such that 
\begin{equation} \label{commrel}
N_{\fM^{\rig}_{\La}} \phi_{\fM^{\rig}_{\La}} = E(u) \phi_{\fM^{\rig}_{\La}} N_{\fM^{\rig}_{\La}}
\end{equation}
and $N_{\fM^{\rig}_{\La}}\mod u = 0$.   The module $\fM^{\rig}_{\La}$ is stable under $N_{\fM^{\rig}_{\La}}$ if and only if $T^*_{dd}(\fM_{\La})[1/p]$ is the restriction to $G_{K_{\infty}}$ of a potentially crystalline representation of $G_K$ which becomes crystalline when restricted to $G_L$.  
\end{thm} 
\begin{proof} This is essentially \cite[Corollary 1.3.15]{KisinFcrys}. It is stated there without tame descent data, however, using the full faithfulness of the restriction from crystalline $G_{L}$-representations to $G_{L_{\infty}}$-representations (Corollary 2.1.14 loc. cit.) one can extend the result to the potentially crystalline case. 
\end{proof}

We remark that the monodromy operator $N_{\fM^{\rig}_{\La}}$ respects the decomposition $\fM^{\rig}_{\La}=\oplus_{j=0}^{f-1} \fM^{\rig, (j)}_{\La}$. In particular, one has $N_{\fM^{\rig}_{\La}}^{(j+1)} \phi_{\fM^{\rig}_{\La}}^{(j)} = E(u) \phi_{\fM^{\rig}_{\La}}^{(j)} N_{\fM^{\rig}_{\La}}^{(j)}$ where $N_{\fM^{\rig}_{\La}}^{(j)}$ is the monodromy operator induced by $N_{\fM^{\rig}_{\La}}$ on $\fM^{\rig, (j)}_{\La}$.

Let $\fM_{\La}\in Y^{[0,2],\tau}(\La)$ be as above and let $\beta=\{\beta^{(j)}\}$ be an eigenbasis for $\fM_{\La}$.  Given the finite height conditions on $\fM_{\La}$, we always have $N_{\fM^{\rig}_{\La}}(\fM_{\La}) \subset \frac{1}{\lambda} \fM^{\rig}_{\La}$ by same argument from \cite[Proposition 2.2.2]{KisinFcrys}.  In what follows, we set $C^{(j-1)}\defeq \Mat_{\beta}(\phi_{\fM_{\La}}^{(j-1)})$ and define the \emph{matrix of the monodromy at $j$} as $N^{(j)}_{\infty}\defeq\Mat_{\beta}(N_{\fM^{\rig}_{\La}}^{(j)})$.

The following Lemma shows that we can construct $N^{(j)}_{\infty}$ by successive approximation. We state it in a slightly greater generality than our specific situation. 
\begin{lemma} \label{mconverge} 
Let $\tau$ be a tame inertial type and let $\fM_{\La}\in Y^{[0,2],\tau}(\La)$ be a Kisin module over $\La$. Let $N_{0}^{(j)} = 0$ for all $j \in \Z/f\Z$.   
For each $i \geq 1$, set
$$
N_i^{(j)} \defeq E(u)C^{(j-1)} \phz(N_{i-1}^{(j-1)}) (C^{(j-1)})^{-1}  - N_{\nabla}(C^{(j-1)}) (C^{(j-1)})^{-1}.
$$ 
Then $N_i^{(j)}$ converges in $\frac{1}{\lambda} \Mat(\cO^{rig}_{\La})$ to $N^{(j)}_{\infty}$.
Moreover, $\Ad_{s_{j}}(u^{\bf{a}_1}, u^{\bf{a}_2}, u^{\bf{a}_3})\left(N^{(j)}_{\infty}\right)\in\Mat_3\left((\cO^{rig}_{\La})^{\Delta=1}\right)$.
\end{lemma} 
\begin{proof} 
We show by induction that 
\begin{equation} \label{m3}
\lambda(N^{(j)}_{i+1} - N^{(j)}_{i}) \in u^{p^{i-1}} \Mat(\cO^{\rig}_{\La})
\end{equation}
for all $j$ and $i \geq 1$. This proves that $\lambda N^{(j)}_i$ converges to $\lambda \widetilde{N}^{(j)}_{\infty}$ in $\Mat(\La[1/p][\![u]\!])$ and satisfies the commutation relation with Frobenius, and thus we conclude that $\lambda \widetilde{N}^{(j)}_{\infty} = \lambda N^{(j)}_{\infty} \in \Mat_3(\cO^{\rig}_{\La})$ (a priori, the convergence happens in a formal power series ring, however one can estimate the Gauss norms to see that sequence actually converges in $\cO^{\rig}_{\La}$).     

The inductive step for $i\geq 1$ follows easily from the relation
\begin{equation*}
\lambda \left(N_{i+1}^{(j)} - N^{(j)}_i \right) = \frac{E(u)^2}{p} C^{(j-1)} \phz \left(\lambda \left(N_{i}^{(j-1)} - N^{(j-1)}_{i-1} \right) \right) (C^{(j-1)})^{-1} 
\end{equation*} 
since we have $E(u)^2 \big(C^{(j-1)}\big)^{-1}\in \Mat(\La[\![u]\!])$ by the height condition.  
For the base case, we consider 
$$
\lambda N_{1}^{(j)} = -\lambda N_{\nabla}(C^{(j-1)}) (C^{(j-1)})^{-1}. 
$$
By the height condition, $\lambda^2 \left( C^{(j-1)} \right)^{-1} \in \Mat(\cO^{\rig}_{\La})$ so it suffices to show that 
\begin{equation} \label{m4}
\frac{1}{\lambda} N_{\nabla}(C^{(j-1)}) \in u \Mat(\La[\![u]\!])
\end{equation}
which is obvious.

The last assertion is an immediate consequence of the compatibility between the descent data action on $\fM_{\La}^{\rig}$ and the monodromy operator.
\end{proof}

We now state the condition which controls the poles of the monodromy operator.  Recall that we fixed $\pi\defeq (-p)^{\frac{1}{p^{f} - 1}}$ as a uniformizer for $L$.
\begin{prop} 
\label{Mcond}
Let $\fM_{\La}\in Y^{[0,2],\tau}(\La)$ with eigenbasis $\beta$, and write
$\Mat_{\beta}(N_{\fM^{\rig}_{\La}}^{(j)})=N_{\infty}^{(j)}=\underset{i\rightarrow \infty}{\lim} N_i^{(j)}$ as in Lemma \ref{mconverge}.  
Then $\fM^{\rig}_{\La}$ is stable under $N_{\fM^{\rig}_{\La}}$ if and only if 
$$
\lambda N_{\infty}^{(j)} \mid_{u  = \pi} = 0
$$
for all $j$. 
\end{prop} 
\begin{proof}
Since $\lambda N_{\infty}^{(j)} \in \Mat(\cO^{\rig}_{\La})$, $\fM^{\rig}_{\La}$ is stable under $N_{\fM^{\rig}_{\La}}$ if and only if $\lambda N_{\infty}^{(j)}$ is divisible by $\lambda$.  Since $\lambda$ has simple zeroes exactly  at $\left\{ \zeta\pi^{1/p^n}\,|\, n\geq 0,\, \zeta\in \cO,\, \zeta^{p^ne}=1 \right\}$, it suffices to show that 
$$
\lambda N_{\infty}^{(j)} \mid_{u  = \zeta\pi^{1/p^n}} = 0
$$
for all $n\geq 0$, $\zeta$ and all $j$. 
The commutation relation $N_{\fM^{\rig}_{\La}}^{(j)} \phi_{\fM^{\rig}_{\La}}^{(j-1)} = E(u) \phi_{\fM^{\rig}_{\La}}^{(j-1)} N_{\fM^{\rig}_{\La}}^{(j-1)}$
translates into
$$
N_{\infty}^{(j)} C^{(j-1)}+ N_{\nabla}(C^{(j-1)})=E(u)C^{(j-1)}\phz\big(N_{\infty}^{(j-1)}\big).
$$
Since $C^{(j-1)}$ is invertible at $\zeta\pi^{1/p^n}$ when $n > 0$ and $N_{\nabla}(C^{(j-1)})$ is divisible by $\lambda$, we see that 
$$
\lambda N_{\infty}^{(j-1)} \mid_{u  = \zeta^p\pi^{1/p^{n-1}}} = 0 \implies   \lambda N_{\infty}^{(j)} \mid_{u  = \zeta\pi^{1/p^{n}}}=0.
$$
for all $n > 0$.  Thus, we are reduced to checking the pole condition at $u = \zeta\pi$ for $\zeta$ being an $e$-th root of unity.
As   $N_{\infty}^{(j)}$ are all  conjugate to matrices in $\La[1/p][\![v]\!]$ by Lemma \ref{mconverge}, we are reduced to check the condition at $u=\pi$.
\end{proof}

By construction, $N_{\infty}^{(j)}$ only depends on the $(C^{(j)})$.   So, of course, $\lambda N_{\infty}^{(j)}\mid_{u = \pi}$ also only depends on the $C^{(j)}$.  In general, however this could be a complicated condition on the coefficients of $(C^{(j)})$.   We now show in fact this condition can be written as an explicit polynomial equation plus an ``error'' term which is divisible by a power of $p$ depending on the genericity of $\tau$.  If $\tau$ is sufficiently generic, the special fiber of $R^{(2,1,0), \tau}_{\rhobar}$ will only depend on the ``leading term.''

We will want to apply our condition to the universal finite height deformations constructed in \S 4.2.  Let $R$ be any complete local Noetherian flat $\cO$-algebra with finite residue field.  Define $\cO^{\rig}_R$ to be the power series $\sum_{i = 0}^{\infty} a_n u^n$ with $a_n \in R[1/p]$ such that $p^n a_n^k \ra 0$ for all $k>0$.

Let $\fM_R \in Y^{[0,2], \tau}(R)$ equipped with an eigenbasis $\beta$.  As before, we let $C^{(j)}\defeq \Mat_{\beta}(\phi_{\fM_R}^{(j)})$ and write $\Mat_{\beta}(N_{\fM^{\rig}_{R}}^{(j)})=N_{\infty}^{(j)}=\underset{i\rightarrow \infty}{\lim} N_i^{(j)}$ as in Lemma \ref{mconverge}. Note that $N_{\infty}^{(j)} \in \frac{1}{\lambda} \Mat_3(\cO^{\rig}_{R})$.  

 We will now study the convergence in Lemma \ref{mconverge} more carefully. 
 
 \begin{lemma} \label{keylemma} We have 
 $$
 \frac{1}{\lambda} N_{\nabla}(C^{(j-1)}) =  \Ad_{s_j}(u^{\bf{a}_1}, u^{\bf{a}_2}, u^{\bf{a}_3})(A^{(j-1), \dagger})
 $$
 where $A^{(j), \dagger} \in \Mat_3(R[\![v]\!])$.  Furthermore, $A^{(j), \dagger} \mod v$ is upper triangular nilpotent.
 \end{lemma} 
 \begin{proof} Applying Leibniz rule to (\ref{change isotypical}), we see that 
\begin{equation}
\label{eq:Leibniz}
 A^{(j-1), \dagger} = - u \frac{d}{du} A^{(j-1)} - \mathrm{Diag}(\bf{a}^{(j)}_{s_j(1)}, \bf{a}^{(j)}_{s_j(2)}, \bf{a}^{(j)}_{s_j(3)}) A^{(j-1)} + A^{(j-1)}  \mathrm{Diag}(\bf{a}^{(j)}_{s_j(1)}, \bf{a}^{(j)}_{s_j(2)}, \bf{a}^{(j)}_{s_j(3)})
\end{equation}
 which is a matrix in $v$.  Furthermore, $- u \frac{d}{du} A^{(j-1)}$ is divisible by $v$ and the rest is 0  modulo $v$ along the diagonal.  
 \end{proof}

\begin{defn} \label{leadingterm} We define the  \emph{leading term} as
$$
P_{N}(A^{(j-1)}) \defeq   A^{(j-1), \dagger}  P(v)^2 (A^{(j-1)})^{-1}.
$$
Note that
\[
P_{N}(A^{(j-1)})\equiv 
z_j\left(-ev\frac{d}{dv}A^{(j-1)}+A^{(j-1)}  \mathrm{Diag}(\bf{a}^{(j)}_{s_j(1)}, \bf{a}^{(j)}_{s_j(2)}, \bf{a}^{(j)}_{s_j(3)})\right)\left(\frac{1}{P(v)}\mathrm{adj}(A^{(j-1)})\right)
\]
modulo $P(v)$, where $z_j\in R^*$ is a suitable unit.
\end{defn}

The following theorem is the main result of this section:
\begin{thm} \label{moncond}
\label{thm:monodr} Let $\fM_R \in Y^{[0,2], \tau}(R)$ equipped with an eigenbasis $\beta$  and let  $N_{\infty}^{(j)} \in \frac{1}{\lambda} \Mat_3(\cO^{\rig}_{R})$ be defined as in Lemma \ref{mconverge} above.  
Assume that $\tau$ is $n$-generic with $n\geq 2$. 
Then 
$$
\Ad^{-1}_{s_j}(u^{\bf{a}_1}, u^{\bf{a}_2}, u^{\bf{a}_3})  (\lambda N_{\infty}^{(j)})  \mid_{u = \pi} = z \left(P_{N}(A^{(j-1)}) \mid_{u  = \pi} - p^{n-1} M_{\mathrm{err}}^{(j)}\right)
$$
where $z$ is a unit in $R[1/p]$, $P_{N}(A^{(j-1)})$ is as in Definition \ref{leadingterm} and $M^{(j)}_{\mathrm{err}} \in \Mat_3(R)$.   

Moreover there exists a matrix $Z^{(j)}\in \Mat_3(R)$ such that
\[
M_{\mathrm{err}}^{(j)}=\big(A^{(j-1)}|_{v=-p}\big)Z^{(j)}\big(\big(P(v)^2(A^{(j-1)})^{-1}\big)|_{v=-p}\big).
\]
\end{thm}
\begin{proof} 
Let's examine the sequence from Lemma \ref{mconverge} in more detail. Consider that
\begin{align*}
N_{\infty}^{(j)} &= N_1^{(j)} + \sum^{\infty}_{i \geq 1} (N_{i+1}^{(j)} - N_i^{(j)}) \\
&=N_1^{(j)} + \sum^{\infty}_{i \geq 1} C^{(j-1)} \phz(N_{i}^{(j-1)} - N_{i-1}^{(j-1)})  E(u) (C^{(j-1)})^{-1} \\
&=N_1^{(j)} + \sum^{\infty}_{i \geq 1} \left(\prod_{k=0}^{i-1}\phz^{k}(C^{(j-k-1)})\right) \phz^i(N_{1}^{(j-i)})\left(\prod_{k=i-1}^{0}\phz^{k} (C^{(j-k-1), *}) \right)
\end{align*}
where $C^{(j), *} :=  E(u) (C^{(j)})^{-1}$. 

From Lemma \ref{keylemma}, we deduce that 
$$
\lambda N_1^{(j)} = - \frac{\phz(\lambda)^2}{p^2}    \Ad_{s_j}(u^{\bf{a}_1}, u^{\bf{a}_2}, u^{\bf{a}_3})(P_N(A^{(j-1)})).
$$
Let $z =  - \frac{\phz(\lambda)^2}{p^2}|_{u = \pi}$ which is in  $\frac{1}{p^2}\cO_K^{\times}$  since $\phz^n(E(u)/p)$ has constant term 1.  Now consider the trailing term
$$
p^2 \lambda \left(\prod_{k=0}^{i-1}\phz^{k}(C^{(j-k-1)})\right) \phz^i(N_{1}^{(j-i)})\left(\prod_{k=i-1}^{0}\phz^{k} (C^{(j-k-1), *}) \right)
$$
for $i \geq 1$. Substituting $N_1^{(j-i)} = \frac{\phz(\lambda)}{p} \left (u \frac{d}{du} C^{(j-i -1)} \right) C^{(j-i-1), *}$, we can rewrite this as 
\begin{equation} \label{x7}
X_i^{(j)} := \frac{\phz^{i+1}(\lambda)^2}{p^i} \left(\prod_{k=0}^{i-1}\phz^{k}(C^{(j-k-1)})\right) \phz^i\left(u \frac{d}{du} C^{(j-i -1)} \right) \left(\prod_{k=i}^{0}\phz^{k} \left(E(u) C^{(j-k-1), *} \right) \right).
\end{equation}
We would now ``remove the descent datum'' and write this as an expression in the $A^{(j)}$'s. Define
\begin{equation} \label{x8}
Z^{(j)}_i = \Ad^{-1}_{s_j}(u^{\bf{a}_1}, u^{\bf{a}_2}, u^{\bf{a}_3}) \left(\frac{1}{\phz^{i+1}(\lambda)^2} X_i^{(j)} \right);
\end{equation}
and note that so far we can write
\[
p^2\Ad^{-1}_{s_j}(u^{\bf{a}_1}, u^{\bf{a}_2}, u^{\bf{a}_3})  (\lambda N_{\infty}^{(j)})=
-\phz(\lambda)^2P_N(A^{(j-1)})+\sum_{i\geq 1}^{\infty}\phz^{i+1}(\lambda)^2Z_i^{(j)}.
\]
We inductively show that $Z^{(j)}_i \in \frac{v^{(n-1)p^{i-1}}}{p^i} \Mat(R[\![v]\!])$ when $i >1$ and $Z^{(j)}_1 \in \frac{v^{n}}{p} \Mat(R[\![v]\!])$ (recall that $n\leq \frac{p-1}{2}$). 
This suffices to prove the first part of the Theorem by evaluating at $u = \pi$ (i.e. $v = -p$), and shows in particular that
\begin{equation}
\label{eq:err:term:expl}
M_{\mathrm{err}}^{(j)}=\frac{1}{\phz(\lambda)^2}\sum_{i\geq 1}^{\infty}\phz^{i+1}(\lambda)^2\frac{Z_i^{(j)}|_{v=-p}}{p^{n-1}}
\end{equation}
with $\frac{Z_i^{(j)}|_{v=-p}}{p^{n-1}}\in \Mat_3(R)$ for all $i\geq 1$.

First observe that by compatibility with descent datum and by the height condition
$$
\Ad^{-1}_{s_j}(u^{\bf{a}_1}, u^{\bf{a}_2}, u^{\bf{a}_3})(\phz^k(C^{(j-k-1)})), \Ad^{-1}_{s_j}(u^{\bf{a}_1}, u^{\bf{a}_2}, u^{\bf{a}_3})(\phz^k(E(u) C^{(j-k-1), *}))  \in \Mat(R[\![v]\!]). 
$$
The key divisibility comes from the middle term.   Take $\ell = j- i$, then 
\begin{align*}
Y^{(j), \dagger}_{i} :=& -\Ad^{-1}_{s_j}(u^{\bf{a}_1}, u^{\bf{a}_2}, u^{\bf{a}_3}) \left( \phz^i\left(u \frac{d}{du} C^{(\ell -1)} \right)\right) \\
=& s \Ad \left(u^{p^{i} a_{s_{\ell}(1)}^{(\ell)} - a_{s_{\ell}(1)}^{(j)}}, u^{p^{i} a_{s_{\ell}(2)}^{(\ell)} - a_{s_{\ell}(2)}^{(j)}}, u^{p^{i} a_{s_{\ell}(3)}^{(\ell)} - a_{s_{\ell}(3)}^{(j)}} \right)\left( \phz^i(A^{(\ell -1), \dagger})\right) s^{-1} \\
=& s \Ad \left( v^{r_{i, 1}}, v^{r_{i, 2}}, v^{r_{i, 3}} \right)\left( \phz^i(A^{(\ell -1), \dagger})
\right) s^{-1}
\end{align*}
for some $s \in S_3$, by the same calculation as in (\ref{xx2}).  We have $u^{p^{i} a_{s_{\ell}(k)}^{(\ell)} - a_{s_{\ell}(k)}^{(j)}} = u^{er_{i, k}} = v^{r_{i, k}}$ for some $r_{i, k}$ where $r_{i, 1} > r_{i, 2} > r_{i, 3}$.  When $i =1$, we have $r_{1, k} = a_{s_{\ell}(k), f-\ell -1}$ and by the $n$-genericity condition, 
$$
p-1-n\geq |r_{1,1} - r_{1, 2}|, |r_{1,2} - r_{1,3}| \geq n
$$
which implies that $v^n \mid  Y^{(j), \dagger}_{1}$.  When $i > 1$, an elementary calculation shows that 
$$
(p-1-n) p^{i-1} \geq |r_{1,1} - r_{1, 2}|, |r_{1,2} - r_{1,3}| \geq (n-1) p^{i-1} 
 $$
and so $v^{(n-1) p^{i-1}} \mid  Y^{(j), \dagger}_{i}$.
We now prove the second statement in the Theorem.
An easy computation shows that, by letting $s_{j,j-1}\defeq s_j^{-1}s_{j-1}$ (cf. the proof of Proposition \ref{changeofbasis}), one has
\begin{align*}
Z_i^{(j)}&=A^{(j-1)}\left(\Ad_{s_j}^{-1}(u^{\bf{a}_1},u^{\bf{a}_2},u^{\bf{a}_3})\left(
\Ad_{s_{j-1}}(u^{p\bf{a}_1},u^{p\bf{a}_2},u^{p\bf{a}_3})\left(\frac{\phz(Z^{(j-1)}_{i-1})}{p}\right)\right)\right)P(v)^2 (A^{(j-1)})^{-1}\\
&=
A^{(j-1)}s_{j,j-1}\left(\Ad(v^{a_{s_{j-1}(1),f-j}},v^{a_{s_{j-1}(2),f-j}},v^{a_{s_{j-1}(3),f-j}})\left(\frac{\phz(Z^{(j-1)}_{i-1})}{p}\right)\right)s_{j,j-1}^{-1}P(v)^2 (A^{(j-1)})^{-1}
\end{align*}
for all $j=0,\dots,f-1$ and $i\geq 1$ (and where we define $Z^{(j-1)}_0\defeq \Ad_{s_{j-1}}(u^{\bf{a}_1},u^{\bf{a}_2},u^{\bf{a}_3})\left(u\frac{d}{du} C^{(j-2)}\right)$ for all $j$).
We now prove that 
\[
\left(\Ad(v^{a_{s_{j-1}(1),f-j}},v^{a_{s_{j-1}(2),f-j}},v^{a_{s_{j-1}(3),f-j}})\left(\frac{\phz(Z^{(j-1)}_{i-1})}{p}\right)\right)|_{v=-p}\in p^{n-1}\Mat_3(R)
\]
for all $i\geq 1$ from which the conclusion follows easily  from (\ref{eq:err:term:expl}).

For $i>1$ we have
\[
\left(\frac{\phz(Z^{(j-1)}_{i-1})}{p}\right)|_{v=-p}\in p^{p-2}\Mat_3(R)
\]
using the fact that $\frac{\phz(Z^{(j-1)}_{i-1})}{p}\in \frac{v^{(n-1)p^{i-1}}}{p^i}\Mat_3(R[\![v]\!])$ for $i\geq 3$ and $\frac{\phz(Z^{(j-1)}_{1})}{p}\in \frac{v^{pn}}{p^2}\Mat_3(R[\![v]\!])$.
This together with the $n$-genericity of $\tau$ gives the desired inclusion in this case.
For $i=1$ the inclusion follows from Lemma \ref{keylemma}.
\end{proof} 

Let $\overline{\fM} \in Y^{\mu, \tau}(\F)$ with shape $\bf{w} = (\widetilde{w}_0, \widetilde{w}_1, \ldots, \widetilde{w}_{f-1})$.  Fix a gauge basis $\overline{\beta}$ on $\overline{\fM}$. Let $R^{\tau, \overline{\beta}}_{\overline{\fM}} = \widehat{\otimes}_{j \in \Z/f\Z} (R^{\expl}_{\widetilde{w}_j})^{p\text{-flat, red}}$ where $R^{\expl}_{\widetilde{w}_j}$ is given in the first column and $\widetilde{w}_j$ row of Table \ref{table3}.  This represents the universal family $(\fM^{\mathrm{univ}}, \beta^{\mathrm{univ}})$ of deformations of $(\overline{\fM}, \overline{\beta})$ (Theorem \ref{univfh}).  

\begin{defn} \label{monideal} 
Assume that $\tau$ is $n$-generic.
Let $A^{(j-1)} = \Mat_{\beta}(\phi_{\fM^{\mathrm{univ}}, s_j(3)}^{(j-1)})$. The \emph{monodromy condition} at $j$ on $(\fM^{\mathrm{univ}}, \beta^{\mathrm{univ}})$ is  
$$
P_{N}(A^{(j-1)})|_{v = -p} = p^{n-1} M^{(j)}_{\mathrm{err}}.
$$
Let $I_{\mathrm{mon}}^{(j)} \subset R^{\tau, \overline{\beta}}_{\overline{\fM}}$ be the ideal generated by the nine equations from the monodromy condition at $j$. 
\end{defn} 

The following proposition allows us to reduce, in most cases, the monodromy condition to just one equation. 
\begin{prop} \label{onequation} 
Keep the hypotheses of Theorem \ref{moncond}.
If $\widetilde{w}_{j-1} \neq \mathrm{id}$, then $I_{\mathrm{mon}}^{(j)}[1/p]$ is principal. 
\end{prop}
\begin{proof} 
Recall from Theorem \ref{moncond} that we can write 
\[
p^{n-1}M_{\mathrm{err}}^{(j)}=\big(A^{(j-1)}|_{v=-p}\big)\widetilde{Z}^{(j)}=\widetilde{Z}^{\prime,(j)}\big(\big(P(v)^2(A^{(j-1)})^{-1}\big)|_{v=-p}\big)
\]
 for some $\widetilde{Z}^{(j)},\, \widetilde{Z}^{\prime,(j)}\in \Mat_3(R^{\tau, \overline{\beta}}_{\overline{\fM}}[1/p])$.

Furthermore, we have $P_{N}(A^{(j-1)}) = A^{(j-1), \dagger} P(v)^2 (A^{(j-1)})^{-1}$.  We also claim that $P_{N}(A^{(j-1)}) \equiv A^{(j-1)} Y \mod P(v)$.  This follows from (\ref{eq:Leibniz}) and the fact that
$$
- u \frac{d}{du} \left(A^{(j-1)} P(v)^2 (A^{(j-1)})^{-1}\right)\equiv 0 \mod P(v),
$$
which implies
$$
- u \frac{d}{du} (A^{(j-1)}) P(v)^2 (A^{(j-1)})^{-1} \equiv A^{(j-1)} u \frac{d}{du} (P(v)^2 (A^{(j-1)})^{-1}) \mod P(v).
$$
We conclude then the monodromy condition satisfies
\begin{equation} \label{z7}
P_N(A^{(j-1)})|_{v = -p} - p^{n-1} M_{\mathrm{err}}^{(j)} = A^{(j-1)}|_{v = -p} X = X' (P(v)^2 (A^{(j-1)})^{-1})|_{v = -p} 
\end{equation}
for some $X, X' \in \Mat_3(R^{\tau, \overline{\beta}}_{\overline{\fM}}[1/p]).$  

Each $2\times 2$ minor of the matrices $A^{(j-1)}|_{v = -p}$ and $(P(v)^2 (A^{(j-1)})^{-1})|_{v = -p}$ is zero by the height conditions, in particular these matrices have rank at most one.
It follows from \ref{z7} that 
the ratios between the rows and the columns of $P_N(A^{(j-1)})|_{v = -p} - p^{n-1} M_{\mathrm{err}}$ are the same as the ratios of the rows of  $A^{(j-1)}|_{v = -p}$ and the ratios of the columns of
$ (P(v)^2 (A^{(j-1)})^{-1})|_{v = -p} $ respectively. 
A survey of Table \ref{table:lifts} shows that as long as $\widetilde{w}_{j-1} \neq \mathrm{id}$ the matrix $A^{(j-1)}|_{v = -p}\in\Mat_3(R^{\tau, \overline{\beta}}_{\overline{\fM}}[1/p])$  has at least one unit entry (say in row $m$) and the same is true for $(P(v)^2 (A^{(j-1)})^{-1})_{v = -p}$ (say in column $k$). 

It follows that $I_{\mathrm{mon}}^{(j)}[1/p]$ is generated by the $(m,k)$-entry  of $P_N(A^{(j-1)})|_{v = -p} - p^{n-1} M_{\mathrm{err}}$, hence is principal.
\end{proof}

In Table \ref{table3}, we list the one equation which generates $I_{\mathrm{mon}}^{(j)}[1/p]$.

\begin{rmk} \label{rmk111} Proposition \ref{onequation} is \emph{false} for the case $\widetilde{w}_{j-1}=\mathrm{id}$. The reason is that the monodromy conditions only cut out potentially crystalline representations whose Hodge-Tate weights are $\leq (2,1,0)$, and so the Hodge-Tate weights could be either $(2,1,0)$ or $(1,1,1)$. In case of $\mathrm{id}$ shape, the representations with Hodge-Tate weights $(1,1,1)$ do show up, and one must further refine the monodromy condition to get rid of them. This will be addressed separately in \S \ref{sec:badcases}. 
\end{rmk}

\subsection{Potentially crystalline deformation rings}
\label{subsection:PCDR}

In the previous section, we gave a condition for Kisin module with descent datum and $p$-adic Hodge type $(2,1,0)$ to come from a potentially crystalline representation (Proposition \ref{Mcond}).  We will now construct a candidate for the (framed) potentially crystalline Galois deformation ring.

We begin by introducing some deformation problems.  Let $\rhobar: G_{K} \ra \GL_3(\F)$.   Recall that $R^{(2,1,0), \tau}_{\rhobar}$ is the universal \emph{framed} potentially crystalline deformation ring with $p$-adic Hodge type $(2,1,0)$. Let $D_{\rhobar}^{\tau, \Box} := \Spf R^{(2,1,0), \tau}_{\rhobar}$ denote the deformation functor. 
\emph{Since we will always be working in parallel weight $(2,1,0)$, we omit the $p$-adic Hodge type in the notation.} 
   
Assume there exists $\overline{\fM} \in Y^{\mu, \tau}(\F)$ such that $T^*_{dd}(\overline{\fM}) \cong \rhobar|_{G_{K_{\infty}}}$. Note that this is a necessary condition for $R^{(2,1,0), \tau}_{\rhobar}$  to be non-zero. By Theorem \ref{Kisinvariety}, if such a Kisin module exists, then it is unique.  Furthermore, we fix an isomorphism $\overline{\gamma}: T^*_{dd}(\overline{\fM}) \cong \rhobar|_{G_{K_{\infty}}}$. 

We fix a gauge basis $\overline{\beta}$ of $\overline{\fM}$ (in particular, $\Mat_{\beta} (\phi^{(j)}_{\overline{\fM}, s_{j+1}(3)})$ has the form given in Table \ref{table shapes mod p}).  

\begin{defn} \label{defprob} In the following definitions, all data is taken to be compatible with the corresponding data on $\overline{\fM}$ when reduced modulo the maximal ideal.   
\begin{enumerate}
\item Let $R_{\overline{\fM}, \rhobar}^{\tau, \Box}$ denote the complete local Noetherian $\cO$-algebra which represents the deformation problem
$$
D_{\overline{\fM}, \rhobar}^{\tau, \Box}(A) := \left \{ (\fM_A, \rho_A, \delta_A) \mid \fM_A \in  Y^{\mu, \tau}(A), \rho_A \in D_{\rhobar}^{\tau, \Box}(A),  \delta_A:T_{dd}^*(\fM_A) \cong (\rho_A)|_{G_{K_{\infty}}} \right \}
$$
\item Let $R^{\tau, \overline{\beta}, \Box}_{\overline{\fM}, \rhobar}$ denote the complete local Noetherian $\cO$-algebra which represents the deformation problem 
$$
D^{\tau, \overline{\beta}, \Box}_{\overline{\fM}, \rhobar}(A) = \left\{ (\fM_A, \rho_A, \delta_A,\beta_A) \mid (\fM_A, \rho_A,\delta_A) \in D_{\overline{\fM}, \rhobar}^{\tau,\Box}(A), \beta_A \text{ a gauge basis for } \fM_A \right \}
$$
\item Let $R^{\tau, \overline{\beta}}_{\overline{\fM}}$ be as in Theorem \ref{univfh} which represents the deformation problem $D^{\tau, \overline{\beta}}_{\overline{\fM}}(A)$.
\item Let $R^{\tau, \overline{\beta}, \Box}_{\overline{\fM}}$ denote the complete local Noetherian $\cO$-algebra which represents the deformation problem of triples $(\fM_A, \beta_A, \un{e}_A)$ where $(\fM_A, \beta_A) \in D^{\tau, \overline{\beta}}_{\overline{\fM}}(A)$ and $\un{e}_A$ is a basis of $T_{dd}^*(\fM_A)$ lifting the basis on $\rhobar|_{G_{K_{\infty}}}$ so that $(T_{dd}^*(\fM_A), \un{e}_A)$ is a framed deformation of $\rhobar|_{G_{K_{\infty}}}$.
\item Let $R^{\tau, \overline{\beta}, \nabla}_{\overline{\fM}}$ denote the $\cO$-flat and reduced quotient of $R^{\tau, \overline{\beta}}_{\overline{\fM}}$ such that $\Spec R^{\tau, \overline{\beta}, \nabla}_{\overline{\fM}}[1/p]$ is the vanishing locus of the monodromy equations on $\Spec R^{\tau, \overline{\beta}}_{\overline{\fM}}[1/p]$.  
 We define  $R^{\tau, \overline{\beta},\Box, \nabla}_{\overline{\fM}}$ from $R^{\tau, \overline{\beta},\Box}_{\overline{\fM}}$ in a similar way.
\end{enumerate}
\end{defn}

The relationships between the various deformation problems are summarized in the following diagram.  The square is Cartesian and f.s. stands for formally smooth. 
\begin{equation} \label{defdiagram}
\xymatrix{
& & \Spf R^{\tau, \overline{\beta}, \square, \nabla}_{\overline{\fM}} \ar[r]^{f.s.} \ar@{^{(}->}[d]  & \Spf R^{\tau, \overline{\beta}, \nabla}_{\overline{\fM}} \ar@{^{(}->}[d] \\
 &\Spf R^{\tau, \overline{\beta}, \Box}_{\overline{\fM}, \rhobar} \ar[d]^{f. s.}  \ar@{^{(}.>}[ru]^{\xi} \ar@{^{(}.>}[r]  & \Spf R_{\overline{\fM}}^{\tau, \overline{\beta}, \Box} \ar[r]^{f.s.} & \Spf R_{\overline{\fM}}^{\tau, \overline{\beta}}  \\
 \Spf R^{\mu, \tau}_{\rhobar} & \Spf R_{\overline{\fM},\rhobar}^{\tau, \Box}  \ar[l]_{\sim}  
}
\end{equation}

The maps which are formally smooth correspond to forgetting either a framing on the Galois representation or a gauge basis on the Kisin module.  The former is clearly formally smooth while the latter is formally smooth by Theorem \ref{gaugeunique}. The dotted arrows will be proved in Proposition \ref{tanspace2} and Theorem \ref{thm:factors} below to exist and be closed immersions when $\ad(\rhobar) $ is cyclotomic free.

The isomorphism between  $\Spf R^{\mu, \tau}_{\rhobar}$ and $ \Spf R_{\overline{\fM},\rhobar}^{\tau, \Box}$ is Corollary \ref{RKisin}.  If $\overline{\fM}$ has shape $(\widetilde{w}_0, \ldots, \widetilde{w}_{f-1})$, then $R_{\overline{\fM}}^{\tau, \overline{\beta}} = \widehat{\otimes}( R^{\expl}_{\widetilde{w}_j})^{p\text{-flat, red}}$ where $R^{\expl}_{\widetilde{w}_j}$ is given in Table \ref{table3} (Theorem \ref{univfh}).  As we will see, as long as $\widetilde{w}_j \neq \mathrm{id}$ for all $j$, the map $\xi$ will be an isomorphism (see Remark \ref{rmk111}).

The following Proposition follows from Proposition \ref{prop:Z1}:
\begin{prop} \label{tanspace2} Assume that $\ad(\rhobar)$ is cyclotomic free $($Definition $\ref{defn:cyclofree})$.   The obvious map
$$
\Spf R_{\overline{\fM},\rhobar}^{\tau, \overline{\beta}, \Box}  \ra  \Spf R_{\overline{\fM}}^{\tau, \overline{\beta}, \Box}
$$
between the deformation spaces defined in Definition \ref{defprob} is a closed immersion.   
\end{prop}


\begin{thm} \label{thm:factors} Assume that $\ad(\rhobar)$ is cyclotomic free. The map 
$$
\Spf R^{\tau, \overline{\beta}, \Box}_{\overline{\fM}, \rhobar} \iarrow \Spf R_{\overline{\fM}}^{\tau, \overline{\beta}, \Box}
$$
factors through $ \Spf R^{\tau, \overline{\beta}, \square, \nabla}_{\overline{\fM}}$ inducing a surjective map $R^{\tau, \overline{\beta}, \Box, \nabla}_{\overline{\fM}} \xrightarrow{\xi} R^{\tau, \overline{\beta}, \Box}_{\overline{\fM}, \rhobar}$.  Furthermore, if $\widetilde{w}_j \neq \mathrm{id}$ for all $j \in \Z/f\Z$, then $\xi$ is an isomorphism.  
\end{thm}
\begin{proof}
 Both $R_{\overline{\fM}}^{\tau, \overline{\beta}, \Box}$ and $R^{\tau, \overline{\beta}, \Box}_{\overline{\fM}, \rhobar}$ are flat $\cO$-algebras.  Furthermore, $R^{\tau, \overline{\beta}, \Box}_{\overline{\fM}, \rhobar}[1/p]$ is reduced since the same is true for the potentially crystalline deformation ring $R^{\mu, \tau}_{\rhobar}$. Thus, it suffices to show factorization at the level of $\overline{\Q}_p$-points. 

For any $\cO'$ finite over $\cO$, an $\cO'$-point of  $R^{\tau, \overline{\beta}, \Box}_{\overline{\fM}, \rhobar}$ corresponds to a Kisin module $\fM_{\cO'} \in Y^{\mu, \tau}(\cO')$ such that $T_{dd}^*(\fM_{\cO'})$ is a lattice in a potentially crystalline representation with Hodge-Tate weights $(2,1,0)$.  By Theorem \ref{thmKisin}, the monodromy condition holds at the corresponding $\cO'[1/p]$ point.  

Any homomorphism $R^{\tau, \overline{\beta}, \Box, \nabla}_{\overline{\fM}} \ra \cO'$ gives rise to a Kisin module $\fM_{\cO'}$ together with a gauge basis on which $A^{(j)}=\Mat_{\beta}(\phi_{\fM_{\cO'}, s_{j+1}(3)}^{(j)})$ has the form given in Table \ref{table:lifts}.  Furthermore, $\fM_{\cO'} \otimes_{\fS} \cO^{\rig}$ is stable under the monodromy operator and hence $T^*_{dd}(\fM_{\cO'})[1/p] =: V_{E'}$ extends to a potentially crystalline representation of $G_{K}$.  The claim is as long as $\widetilde{w}_j \neq  \mathrm{id}$, then $V_{E'}$ has $p$-adic Hodge type of parallel weight $(2,1,0)$. The $p$-adic Hodge type at the embedding $\sigma_j$ is $(1,1,1)$ if and only if the Frobenius $C^{(j)}$ is divisible by $E(u)$, equivalently $A^{(j)}$ is divisible by $P(v)$.  A survey of last column of Table \ref{table:lifts} shows that this can only happen when $\widetilde{w}_j = \mathrm{id}$.    
\end{proof} 

\begin{cor} \label{dring} Assume that $\ad(\rhobar)$ is cyclotomic free. If $\widetilde{w}_j \neq \mathrm{id}$ for all $j \in \Z/f\Z$, then 
$$
R^{\mu, \tau}_{\rhobar}[\![S_1, \ldots, S_{3f}]\!] \cong R^{\tau, \overline{\beta}, \nabla}_{\overline{\fM}}[\![T_1,\ldots, T_8]\!].
$$
\end{cor} 

\begin{rmk} The assumption that $\ad(\rhobar)$ is cyclotomic free is automatic (by Proposition \ref{2impliescyc}) if one assumes a slightly stronger genericity condition on $\tau$ which forces $\rhobar$ to be 2-generic.  In any case, it is likely that this assumption could be removed by using more about the tangent space of the potentially crystalline deformation ring.  
\end{rmk}

\subsection{Explicit deformation rings}

\label{Explicit deformation rings}
We now proceed to compute $R^{\tau, \overline{\beta}, \nabla}_{\overline{\fM}}$ for generic $\tau$ (which means $5$-generic cf. Definition \ref{gencond}) in many cases thus obtaining by Corollary \ref{dring} a description of $R^{\mu, \tau}_{\rhobar}$.  Assume that $\widetilde{w}_j \not \in \{ \alpha,\,\beta,\,\gamma,\, \mathrm{id} \}$ for all $j$. 

Recall that $R^{\tau, \overline{\beta}}_{\overline{\fM}} := \widehat{\otimes} (R^{\expl}_{\widetilde{w}_j})^{p\text{-flat, red}}$ where $R^{\expl}_{\widetilde{w}_j}$ is the $\cO$-algebra corresponding to shape $\widetilde{w}_j$ in Table \ref{table3}. 
Let $\widetilde{I}_{\mathrm{mon}} \subset R^{\tau, \overline{\beta}}_{\overline{\fM}}$ denote the $p$-saturation of the sum of the ideals $I^{(j)}_{\mathrm{mon}}$ generated by the monodromy conditions (Definition \ref{monideal}).  By definition,
$$
R^{\tau, \overline{\beta}, \nabla}_{\overline{\fM}} = (R^{\tau, \overline{\beta}}_{\overline{\fM}}/\widetilde{I}_{\mathrm{mon}})^{\text{red}}. 
$$ 

Now, we can consider the explicit quotient $R^{\expl, \nabla}_{\overline{\fM}}$ of $R^{\tau, \overline{\beta}}_{\overline{\fM}}$ given by imposing the  single monodromy equation for each $j$ in Table \ref{table3}.  By Proposition \ref{onequation},
$$
R^{\expl, \nabla}_{\overline{\fM}}[1/p] = R^{\tau, \overline{\beta}}_{\overline{\fM}}/\widetilde{I}_{\mathrm{mon}}[1/p].
$$
The aim is to determine the $p$-torsion free and reduced quotient $R^{\tau, \overline{\beta}, \nabla}_{\overline{\fM}}$ of $R^{\expl, \nabla}_{\overline{\fM}}$.  As long as $\tld{w}_j$ has length at least $2$ for all $j$, this is non-canonically isomorphic to a completed tensor product of the explicit rings  $R^{\expl, \nabla}_{\overline{\fM}, \widetilde{w}_j}$ given in  Table \ref{table:withmon}. Note that Table \ref{table:withmon} gives $R^{\expl, \nabla}_{\overline{\fM}, \widetilde{w}_j}$ for each shape of length at least $2$, with further conditions from $\overline{\fM}$.

\begin{rmk} Table \ref{table:withmon} only includes the $R^{\expl, \nabla}_{\overline{\fM}, \widetilde{w}_j}$ for our chosen representatives for the $\delta$-orbits on $\mathrm{Adm}(2,1,0)$ (see Corollary \ref{symmetry} and Remark \ref{deltaorbits}).  If $\widetilde{w}'$ is in the same $\delta$-orbit as $\widetilde{w}$, then the explicit ring for that shape is isomorphic to $R^{\expl, \nabla}_{\overline{\fM}, \widetilde{w}_j}$, only the labelling of the variables by entry changes. 
\end{rmk}
 
We give two sample calculations of $R^{\tau, \overline{\beta}, \nabla}_{\overline{\fM}}$ with the rest being similar.
The computations below show that this is a quotient of (a completed tensor product of) the ring appearing in Table \ref{table:withmon}. Since the rings in Table \ref{table:withmon} are reduced and $p$-flat and satisfy the defining relations of $R^{\expl}_{\tld{w}_j}$ and the monodromy conditions, this gives us $R^{\tau, \overline{\beta}, \nabla}_{\overline{\fM}}$.

\subsubsection{The $\alpha\beta\alpha$ cell.}

From the monodromy equation in Table \ref{table3}, we have
\begin{equation}
\label{eq:mon:cnd:ABA}
(e-(a-c))c_{33}c_{22}^*-p(a-b)c_{23}c_{32}+pec_{22}^*c'_{33}+O(p^{n-2}).
\end{equation}
By the finite height and determinant equations, we see that
\[
c_{11}(c_{33}+pc'_{33})\equiv -p^2 c^*_{13}c^*_{31}
\]
hence $c_{11}$ is a unit in $R^{\tau, \overline{\beta}}_{\overline{\fM}}[1/p]$.
We multiply  (\ref{eq:mon:cnd:ABA}) by $c_{11}$: using the finite height and the determinant equations, we obtain after easy manipulations:
\[
c_{11}\left((a-b) c_{23}c_{32} -(a-c) c_{22}^* c'_{33}\right) + p(e-a+c)c_{31}^*c_{22}^*c_{13}^* + O(p^{n-3}).
\]
In particular if the descent data is generic, we can write
\[
c_{11}\left((a-b) c_{23}c_{32} -(a-c) c_{22}^* c'_{33}\right)=pz^*
\]
where $z^*$ is a unit in $R^{\tau, \overline{\beta}}_{\overline{\fM}}$.  Let $y'_{33}  = ((a-b) c_{23}c_{32} -(a-c) c_{22}^* c'_{33}) (z^*)^{-1}$ which replaces $c'_{33}$. (Note that the former change of variables in $R^{\tau, \overline{\beta}}_{\overline{\fM}}$ makes sense even if the unit $z^*$ involves $c'_{33}$.) 
We have 
\begin{equation}
\label{eq:chg:var}
c_{11} y'_{33} = p.
\end{equation}
By $p$-flatness, $y'_{33}$ is not a zero-divisor, and we can multiply the first height equation by $y'_{33}$ to get
$$
y'_{33} c_{11} c_{33} = -p y'_{33} c_{13} c^*_{31} \stackrel{\text{\tiny{(\ref{eq:chg:var})}}}{\implies} c_{33} = - c_{13} y'_{33}c^*_{31}
$$
thereby eliminating $c_{33}$. The second finite height equation can be solved to eliminate $c_{13}$. We are left then with the one equation
$$
c_{11} y'_{33} = p.
$$
There are two cases. When $(a-b) \overline{c}_{23} \overline{c}_{32} -(a-c) \overline{c}_{22}^* \overline{c}'_{33} \neq 0$, then $y'_{33}$ is unit in which case we can solve for $c_{11}$. Otherwise, $y'_{33}$ is in the maximal ideal, and we are left with this one equation as in Table \ref{table:withmon}.

\subsubsection{The $\beta\alpha$ cell.}

From Table \ref{table3}, the monodromy condition gives 
\begin{equation}
\label{eq:mon:cnd:BA}
(-e+a-c)c_{33}c'_{22}+p(a-b)c_{32}c^*_{23}-pec'_{22}c'_{33}+O(p^{n-2}).
\end{equation}
By the finite height and determinant equations, we see that
\[
c_{11}(c'_{22}c'_{33}+\frac{1}{p}c_{33}c'_{22})\equiv p c^*_{12}c^*_{23}c^*_{31}
\]
hence $c_{11}$ is a unit in $R^{\tau, \overline{\beta}}_{\overline{\fM}}[1/p]$.
We multiply  (\ref{eq:mon:cnd:BA}) by $c_{11}$: using the finite height and the determinant equations, we obtain
\[
p\left(c_{11}((a-b)c_{32}c^*_{23}-(a-c)c'_{22}c'_{33})-p(e-a+c)c^*_{12}c^*_{23}c^*_{31}\right)+O(p^{n-2}).
\]
In particular if the descent data is generic, we can write
\[
c_{11}\big((a-b)c_{32}c^*_{23}-(a-c)c'_{22}c'_{33}\big) = p z^*
\]
where $z^*$ is a unit.  Let $y_{32}  = ((a-b)c_{32}c^*_{23}-(a-c)c'_{22}c'_{33}) (z^*)^{-1}$ which replaces $c_{32}$ so that we have
$$
c_{11} y_{32} = p.
$$
Multiplying the first height equation by $y_{32}$, we get
$$
y_{32} c_{11} c_{33} = - p c_{31}^* c_{13} y_{32} \implies c_{33} = - c_{31}^* c_{13} y_{32} 
$$
thereby eliminating $c_{33}$.  Let $y_{13} = -(c_{11} c'_{33} - c_{13} c^*_{31}) (c^*_{23} c^*_{12} c^*_{31})^{-1}$ which replaces $c_{13}$.  We have reduced the equations to 
$$
c_{11} y_{32} = p,\quad c'_{22} y_{13} = p.
$$ 
There are again two cases. When $\overline{c}_{32} \neq 0$, $y_{32}$ is a unit in which case we can solve the first equation.  Otherwise, $y_{32}$ is in the maximal ideal in which case this is a minimal set of equations. Note that $c_{13}\equiv 0$ modulo $\varpi$ so $y_{13}$ is never a unit.

\subsubsection{The $\alpha\beta$ cell.}  

The computations are similar to the $\beta\alpha$ case and we only outline how to obtain the relevant monodromy equations.
From Table \ref{table3}, the monodromy condition gives 
\begin{equation}
\label{eq:mon:cnd:AB}
(e-a+c)c_{31}c_{23}+p(e-a+b)c^*_{21}c'_{33}+p(a-b)c_{31}c'_{23}+O(p^{n-2}).
\end{equation}
Note that $c_{22}$ is a unit in $R^{\tau, \overline{\beta}}_{\overline{\fM}}[1/p]$.
Multiplying the determinant equation by $c_{22}$ and using the finite height relation $c_{22}c_{13}=c_{12}c_{23}$, we obtain

\[
c_{12}(c_{23}c^*_{32}-c'_{33}c_{22})= p c^*_{32}c^*_{13}c_{22}
\]
hence $c_{12}$ is a unit in $R^{\tau, \overline{\beta}}_{\overline{\fM}}[1/p]$.
We multiply  (\ref{eq:mon:cnd:AB}) by $c_{12}$: using the finite height equations, we obtain
\[
-(e-a+c)pc_{13}^*c^*_{21}c^*_{32}+p(e-a+b)c_{12}c^*_{21}c'_{33}+p(a-b)c_{12}c_{31}c'_{23}+O(p^{n-2}).
\]
Using now the determinant condition, we finally get
\[
p\left(c_{12}((a-b)c_{31}c'_{23}+(b-c)c^*_{21}c'_{33})-p(e-a+c)c^*_{21}c^*_{32}c^*_{13}\right)+O(p^{n-2}).
\]
In particular if the descent data is generic, we can write
\[
c_{12}((a-b)c_{31}c'_{23}+(b-c)c^*_{21}c'_{33}) = p z^*
\]
where $z^*$ is a unit. The computations are now similar to those of the $\beta\alpha$ case.

\section{Base change}
\label{sec:BC}

In this section, we extend the results of \S 4-5 to non-principal series tame types.  These tame inertial types provide more flexibility in isolating certain combinations of Serre weights in global applications.  The setup is similar to \cite{EGS} where deformation rings for tame cuspidal types are computed for $\GL_2$. The end result is that the deformation rings have essentially the same form and shapes as for the principal series types. 

Let $r\in\{2,3\}$ and define $f'\defeq fr$,  $K'\defeq\Qpf{f'}$. We write $e'\defeq p^{f'}-1$, let 
$\pi'\defeq {\pi}^{\frac{e}{e'}}$ and set $L'\defeq K'(\pi')$. We fix a sequence of $p$-power roots $\pi'_n$ of $\pi'$, such that ${\pi'_n}^{\frac{e}{e'}}=\pi_n$, and set $L'_{\infty}\defeq\underset{n\in\N}{\bigcup}L'(\pi'_n)$, $K'_{\infty}\defeq\underset{n\in\N}{\bigcup}K'(p_n)$, $L_{\infty}\defeq\underset{n\in\N}{\bigcup}L(\pi_n)$ and $K_{\infty}\defeq\underset{n\in\N}{\bigcup}K(p_n)$.

We have $\Gal(L'_\infty/K_\infty)\cong \Gal(L'/K)$  is generated by $\Delta'\defeq\Gal(L'/K')$ and $\tld{\sigma}$ subject to the relations $\tld{\sigma}g\tld{\sigma}^{-1}=g^{p^f}$ and $\tld{\sigma}^r=1$ (here $\tld{\sigma}$ is characterized by $\tld{\sigma}(\pi')=\pi'$ and $\tld{\sigma}(\zeta)=\zeta^{p^f}$ for $\zeta$ any $e'$-th root of unity). We will fix once and for all a lift of $\tld{\sigma}$ to $G_{K_\infty}$, and abusively also call it $\tld{\sigma}$. Note that the image of $\tld{\sigma}$ in $\Gal(K'_\infty/K_\infty)$ is a generator.

As in Section \ref{subsection:etale phi modules}, we have the rings $\cO_{\cE^{un}, K}= \cO_{\cE^{un}, K'}$ with an action of $G_{K_\infty}$, and subrings $ \cO_{\cE, K}$, $ \cO_{\cE, L}$, $\cO_{\cE, K'}$ and  $\cO_{\cE, L'}$ which are the ring of invariants under $G_{K_\infty}$, $G_{L_\infty}$, $G_{K'_\infty}$ and $G_{L'_\infty}$, respectively.

Recall the character $\omega_{\pi'}:I_{K'}\rightarrow W(k')^{\times}$. By fixing an embedding $\sigma_0':k'\ia \F$ extending $\sigma_0: k\ia\F$, we obtain a fundamental character $\omega_{f'}$ satisfying $\omega_{f'}^{\frac{e'}{e}}=\omega_f$. We also fix a compatible embedding $\sigma_0': W(k') \hookrightarrow \cO$, which allows us to regard $\omega_{f'}$ as an $\cO^\times$ -valued character.
\subsection{Tame descent datum}
\label{sec:baseChangeTameType}

Recall the notations and the general setting of \S\ref{KM with dd}. Consider $f$-tuples $\bf{a}_k\in\{0,\dots,p-1\}^f$ for $1 \leq k \leq 3$ 
and the following tame (non-principal series) inertial types $\tau$:
\begin{equation*}
\left\{\begin{matrix}
\omega_{f'}^{-\bf{a}_1^{(0)}-p^f\bf{a}_1^{(0)}} \oplus \omega_{f'}^{-\bf{a}_2^{(0)}-p^f\bf{a}_3^{(0)}}\oplus \omega_{f'}^{-\bf{a}_3^{(0)} - p^f\bf{a}_2^{(0)}} \text{ when  $r = 2$}
\\
\\
\omega_{f'}^{-\bf{a}_1^{(0)}-p^f\bf{a}_2^{(0)}-p^{2f}\bf{a}_3^{(0)}}\oplus \omega_{f'}^{-\bf{a}_2^{(0)}-p^f\bf{a}_3^{(0)}-p^{2f}\bf{a}_1^{(0)}}\oplus \omega_{f'}^{-\bf{a}_3^{(0)}-p^f\bf{a}_1^{(0)}-p^{2f}\bf{a}_2^{(0)}} \text{ when  $r =3$.} \\
\end{matrix}\right.
\end{equation*}
We write $\tau'$ for the base change of $\tau$ to $K'/K$ (which is just $\tau$ considered as a representation of $I_{K'}$). There is a triple $(\mathbf{a}'_1,\mathbf{a}'_2, \mathbf{a}'_3)$ with $\mathbf{a}'_k \in \{0, \ldots, p-1\}^{f'}$ associated to $\tau'$ such that $\tau' = \eta_1 \oplus \eta_2 \oplus \eta_3$ with $\eta_k = \omega_{f'}^{-\mathbf{a}_k'^{(0)}}$ (with the characters ordered as above).  We say that the type $\tau$ is \emph{$n$-generic} if $\tau'$ is \emph{$n$-generic} (equivalently, the triple $(\mathbf{a}_1,\mathbf{a}_2,\mathbf{a}_3)$ is $n$-generic). 
Similar conventions apply for the notion of genericity, weak genericity and strong genericity. 
Let $(s_j) \in S_3^f$ be the orientation of $(\mathbf{a}_1,\mathbf{a}_2,\mathbf{a}_3)$ (cf. Definition \ref{orientation})

The following lemma records the effect of base change on the orientation:
\begin{prop}
\label{basechange2}
Let $\tau$ be a weakly generic tame type. For $0 \leq j \leq f-1$ and $0 \leq i \leq r - 1$, define
$$
s'_{j + if} = s_{\tau}^{i+1} \circ s_{j}\in S_3
$$
where $s_{\tau}=(23)$ $($resp. $s_{\tau}=(123))$ if $r =2$ $($resp. $r =3)$. Then, the $f'$-tuple $(s'_{j'}) \in S_3^{f'}$ is an orientation of $\tau'$.
\end{prop}
\begin{proof}
This is a casewise computation, remarking that the orientation at $j'$ on $(\bf{a}'_1,\bf{a}'_2,\bf{a}'_3)$ is determined, under the weak genericity assumption, by 
$(a'_{1,f'-1-j'},a'_{2,f'-1-j'},a'_{3,f'-1-j'})$.
\end{proof}

We now study Kisin modules with descent datum of type $\tau$ in relation to Kisin modules with descent datum of type $\tau'$. In the next subsection, we will apply this to potentially crystalline deformation rings with tame Galois type $\tau$.   We write $\sigma\in \Gal(K'/\Qp)$ for the absolute Frobenius on $K'$ and recall $\Delta'= \Gal(L'/K')$.
We also denote by $\sigma$ the automorphism of $W(k')[\![u]\!]$ which fixes $u$ and acts as $\sigma$ on $W(k)$. This extends to an automorphism of $\cO_{\cE, L'}$, also denoted by $\sigma$, and $\sigma^f$ agrees with the Galois automorphism $\tld{\sigma}\in G_{K_\infty}$ on $\cO_{\cE^{un}, K}$ restricted to $\cO_{\cE, L'}$.
We define a Frobenius-twist morphism
\begin{eqnarray*}
(\sigma^f)^{*}: Y^{[0,h], \tau'}\rightarrow Y^{[0,h], {(\tau')}^{p^{f}}}.
\end{eqnarray*} 
Let $R$ be any $\cO$-algebra and let $\fM \in Y^{[0,h], \tau'}(R)$.  Define 
$(\sigma^f)^{*}(\fM)$ to be the $(W(k')\otimes_{\Zp}R)[\![u]\!]$-module obtained from $\fM$ via the base change $\sigma^f: W(k')\ra W(k')$. We define the  Frobenius by $\phi_{(\sigma^f)^{*}(\fM)}\defeq (\sigma^{f})^*(\phi_{\fM})$ and an action of $\Delta'$ via the canonical isomorphism
$$
(\widehat{g^{p^f}})^*\left((\sigma^f)^*(\fM)\right)\cong (\sigma^f)^*\left(\widehat{g}^*(\fM)\right)
$$
using that $g \mapsto g^{p^f}$ is an automorphism of $\Delta'$.  In the Lemma below, we see that if $\fM$ has type $\tau'$, then $(\sigma^f)^* (\fM)$ has type $(\tau')^{ p^{-f}} := (\eta_1')^{ p^{-f}} \oplus (\eta_2')^{p^{-f}} \oplus (\eta_3')^{p^{-f}}$.  

We have a canonical $\sigma^{-f}$ semilinear bijection $(\sigma^f)^{*}(\fM)\rightarrow \fM$ given by $a\otimes m \mapsto \sigma^{-f}(a)m$.
\begin{lemma}
\label{lem basechange1}
Let $\fM\in Y^{[0,h], \tau'}(R)$. For all $j\in\{0,\dots,f'-1\}$, one has the following commutative diagram of $R[\![u]\!]$-modules
\begin{eqnarray*}
\xymatrix@=2pc{
\varphi^*(((\sigma^f)^*(\fM))^{(j)})\ar^{\phi^{(j)}_{(\sigma^f)^*(\fM)}}[rr]\ar^{\wr}[d]&&
((\sigma^f)^*(\fM))^{(j+1)}\ar^{\wr}[d]\\
\varphi^*(\fM^{(j-f)})\ar^{\phi^{(j-f)}_{\fM}}[rr]&&
\fM^{(j-f+1)}.
}
\end{eqnarray*}
Furthermore, for any character $\eta:\Delta' \ra \cO^{\times}$, one has the following commutative diagram of $R[\![v]\!]$-modules
\begin{eqnarray*}
\xymatrix@=2pc{
{}^\phz((\sigma^f)^*(\fM))_{\eta}^{(j)}\ar[r] \ar^{\wr}[d]&
((\sigma^f)^*(\fM))^{(j+1)}_{\eta} \ar^{\wr}[d]\\
{}^\phz\fM^{(j-f)}_{\eta^{p^f}}\ar[r]&
\fM^{(j-f+1)}_{\eta^{p^f}}.
}
\end{eqnarray*}
with horizontal maps induced by the Frobenius.
\end{lemma}

Since $\tau'$ is the base change of the tame inertial type for $I_K$, $(\tau')^{p^{-f}}=\tau'$ and the Frobenius-twist induces an automorphism of $Y^{\mu, \tau'}$. We define the `fixed points' of this automorphism:
 
\begin{defn} 
\label{defn:fixed}
For any $\cO$-algebra $R$, define
$$
Y^{\mu, \tau}(R) = \{ (\fM, \iota) \mid \fM \in Y^{\mu, \tau'}(R), \iota: (\sigma^f)^{*}(\fM) \stackrel{\sim}{\ra} \fM \}
$$
such that the following cocycle condition holds: $\iota \circ (\sigma^f)^*\iota = \mathrm{id}_{\fM}$ (resp. $\iota \circ (\sigma^f)^*\iota \circ (\sigma^{2f})^*\iota  = \mathrm{id}_{\fM}$) when $r = 2$ (resp. $r = 3$).

We define $Y^{[0,2], \tau}$ in a similar fashion.
\end{defn} 

A morphism $(\fM_1,\iota_1)\rightarrow (\fM_2,\iota_2)$ in $Y^{\mu, \tau}(R)$ is a morphism $\fM_1\rightarrow \fM_2$ in $Y^{\mu, \tau'}$ which commutes with the Frobenius twist.

Let $R$ be a complete local Noetherian $\cO$-algebra. Recall from Section \ref{subsection:etale phi modules} the functor $T^*_{dd}$ from $Y^{\mu, \tau'}(R)$ to $G_{K'_\infty}$-representations given by
$$
T^*_{dd}(\fM) = \bV^*_{dd}(\cM) = \Hom_{\phz,\cO_{\cE, L'}}(\cM, \cO_{\cE^{un}, K'})
$$ 
where $\cM=\fM\otimes_{W(k')[\![u]\!]}\cO_{\cE, L'}$ is the \'etale $\phz$-module with descent data corresponding to $\fM$.

If $(\rho,V)$ is a linear representation of a group $G$ and $\psi$ is an automorphism of $G$, the $\psi$-twist of $V$ is the $G$-representation obtained by the composition $\rho\circ \psi:G\to \GL(V)$. 

The following computes the effect of Frobenius twisting under $T^*_{dd}$:
\begin{prop}
\label{proposition:frob twist effect}
There is a canonical bijection 
$$
\mathrm{can}: T^*_{dd}(\fM) \stackrel{\sim}{\ra} T^*_{dd}((\sigma^f)^*(\fM)) 
$$
which identifies $T^*_{dd}((\sigma^f)^*(\fM))$ as the $\Ad(\tld{\sigma})$-twist of $T^*_{dd}(\fM)$
\end{prop}
\begin{proof} We construct the map by sending $h\in \Hom_{\phz,\cO_{\cE, L'}}(\cM, \cO_{\cE^{un}, K'})$ to $\tld{\sigma}\circ h$, which is a $\sigma^{f}$-semilinear map from $\cM$ to $  \cO_{\cE^{un}, K'}$ commuting with $\phz$, hence gives an element of $T^*_{dd}((\sigma^f)^*(\fM))$ (namely $\Id\otimes_{\sigma^f}h$). One easily checks it has the correct equivariance properties.
\end{proof}

\begin{lemma} 
\label{abstractextension}
Let $G$ be a group and $G'$ a normal subgroup of $G$ such that $G/G'$ is a cyclic group of order $d$. Let $g\in G$ be an element lifting a generator of $G/G'$. Let $V$ be a linear $G'$-representation. Then the data of an extension of the $G'$-action to a $G$-action is the same as the data of a linear isomorphism $h:V \stackrel{\sim}{\ra} V$ such that
\begin{enumerate}
\item
$h(g'v)=\left(\Ad(g)(g')\right)h(v)$,
\item
$h^d(v)=g^d(v)$,
\end{enumerate}
for all $v\in V$ and $g'\in G'$.
\end{lemma}
\begin{proof} The data of $h$ is exactly equivalent to the action of the element $g$, and the conditions are exactly what is needed to make it a group action of $G$.
\end{proof}
\begin{prop}
\label{proposition:extension}
Let $R$ be a complete local Noetherian $\cO$-algebra and let $(\fM, \iota) \in Y^{\mu,\tau}(R)$. Let $\cM$ be the \'etale $\phz$-module associated to $\fM$. Then the data of an extension of the $G_{K'_\infty}$-representation $T^*_{dd}(\fM)=\bV^*_{dd}(\cM)$ to a $G_{K_\infty}$-representation is equivalent to the data of an isomorphism 
$$ 
 \iota: (\sigma^f)^{*}(\cM) \stackrel{\sim}{\ra} \cM
$$
satisfying the cocycle condition (cf Definition \ref{defn:fixed}).
\end{prop}
\begin{proof}
This follows from Proposition \ref{proposition:frob twist effect}, Lemma \ref{abstractextension} and the fact that $\bV^*_{dd}$ is an anti-equivalence.
Note that the action of $\tld{\sigma}$ on $\bV^*_{dd}(\cM)$ is given by $\bV^*_{dd}(\iota^{-1})\circ \mathrm{can}$.
\end{proof}
\begin{cor}
Let $R$ be a complete local Noetherian $\cO$-algebra and let $(\fM, \iota) \in Y^{\mu,\tau}(R)$. The $G_{K'_\infty}$-re\-pre\-sen\-ta\-tion $T_{dd}^*(\fM)$ admits a canonical extension to a $G_{K_{\infty}}$-re\-pre\-sen\-ta\-tion which we denote by $T^*_{dd'}(\fM)$.
\end{cor}
\begin{rmk}

The \'etale $\phi$-module $\cM$ over $\cO_{\cE, L'}$ associated to $\fM$ has an action of the group $\Delta'$ from the descent datum.  The isomorphism $\iota$ extends this to an action of $\Gal(L'/K)$. 

One can describe $T^*_{dd'}$ as 
$$
T^*_{dd'}(\fM) = \bV^*_{dd}(\cM) = \Hom_{\phz,\cO_{\cE, L'}}(\cM, \cO_{\cE^{un}, K'})
$$
with the  $G_{K_{\infty}}$-action given by $g\cdot f\defeq g\circ f\circ \overline{g}^{-1}$, where $G_{K_{\infty}}$-acts on $\cM$ through $\Gal(L'/K)$.
\end{rmk}


\begin{prop} \label{parallel} Let $(\overline{\fM},\overline{\iota})\in Y^{\mu, \tau}(\F)$.  If $\bf{w}=(\widetilde{w}_{0},\dots,\widetilde{w}_{f'-1})\in \widetilde{W}^{f'}$ is the shape of $\overline{\fM}$ considered as an element of $Y^{\mu, \tau'}(\F)$, then 
$$
\widetilde{w}_{j'_1}=\widetilde{w}_{j'_2}\qquad\text{whenever}\quad j'_1\equiv j'_2\ \mod{f}.
$$
In other words, $\overline{\fM}$ has \emph{parallel shape}. 
\end{prop}
\begin{proof}
The isomorphism $\overline{\iota}$ induces an isomorphism $\fM^{(j)}_{\eta_{s'_j(3)}} \cong\fM^{(j-f)}_{(\eta_{s'_j(3)})^{p^{f}}}$ (Lemma \ref{lem basechange1}). The Proposition follows from the fact that  
$(\eta_{s'_j(3)})^{p^{f}} = \eta_{s'_{j-f}(3)}$. 
\end{proof} 

\begin{defn} \label{rhobarshape2} Let $\rhobar:G_{K} \ra \GL_3(\F)$ such that $T_{dd'}^*(\overline{\fM}) \cong \rhobar|_{G_{K_{\infty}}}$ for some $(\overline{\fM},\overline{\iota})\in Y^{\mu, \tau}(\F)$.  Define $\bf{w}(\rhobar, \tau) = (\widetilde{w}_0, \ldots, \widetilde{w}_{f-1}) \in \mathrm{Adm}(2,1,0)^f$ where $\widetilde{w}_j$ is the shape of $\overline{\fM} \in Y^{\mu,\tau'}(\F)$ at $j'$ for any $j' \equiv j \mod f$.   This is well defined by Proposition \ref{parallel} and Theorem \ref{Kisinvariety} applied to $\rhobar|_{G_{K'}}$ and $\tau'$.
\end{defn} 

We now discuss gauge bases in the current setting.
\begin{defn} \label{gauge BC} Assume $\tau$ is weakly generic. Let $R$ be a complete local Noetherian $\cO$-algebra and let $(\fM,\iota) \in Y^{[0,2], \tau}(R)$ . A gauge basis of $(\fM,\iota)$ is a gauge basis $\beta$ of $\fM\in Y^{[0,2], \tau'}(R)$ which is compatible with $\iota$, that is, $\iota((\sigma^f)^*(\beta))=\beta$.
\end{defn}
\begin{prop}
\label{gaugeuniqueBC} Assume $\tau$ is weakly generic. Let $R$ be a complete local Noetherian $\cO$-algebra and let $(\fM,\iota) \in Y^{[0,2], \tau}(R)$. Then the set of gauge bases of $(\fM,\iota)$ is a torsor for $(\Res_{W(k')/\Zp} T(R))^{\sigma^f=\Id}=T(W(k')\otimes_{\Zp}R)^{\sigma^f=\Id}$
\end{prop}
\begin{proof} 
Let $\beta_1$ be a gauge basis of $\fM\in Y^{[0,2], \tau'}(R)$. Then $\beta_2=\iota((\sigma^f)^*(\beta_1))$ is also a gauge basis of $\fM$. By Theorem \ref{gaugeunique}, the set of gauge bases of $\fM$ is exactly $T(W(k')\otimes_{\Zp}R)\beta_1$ (note that the proof of Theorem \ref{gaugeunique} also implies that gauge bases of $\overline{\fM}$ are uniquely determined up scalings). 
Thus $\beta_2=c\beta_1$ for a unique $c\in T(W(k')\otimes_{\Zp}R)$. The cocycle condition satisfied by $\iota$ is equivalent to $c\cdot \sigma^f(c)\cdots \sigma^{(r-1)f}(c)=1$.
Observe that $\iota((\sigma^f)^*(t\beta_1))=\sigma^f(t)\iota((\sigma^f)^*(\beta_1))=\sigma^f(t)c\beta_1$. Thus the set of gauge bases of $(\fM,\iota)$ is exactly the set of solutions $t \in T(W(k')\otimes_{\Zp}R)$ to the equation
$t=\sigma^f(t)c$.
As $\Res_{W(k')/\Zp}$ splits over $\cO$, the equation has a solution, and the solution set is a torsor over $(\Res_{W(k')/\Zp} T(R))^{\sigma^f=\Id}$.
\end{proof}
Finally, we observe that by Lemma \ref{lem basechange1}, if $\beta$ is a gauge basis of $(\fM,\iota)$, and $A^{(j)} = \Mat_{\beta}\big(\phi^{(j)}_{\fM, s'_{j+1}(3)}\big)$, then $A^{(j)}=A^{(j+f)}$. The analogue of Theorem \ref{univfh} also holds in our setting, namely the problem of deforming a pair $(\overline{\fM},\overline{\iota})$ is representable by $\widehat{\otimes}_{j=1}^f (R_{\widetilde{w}_j}^{\expl})^{p\text{-flat, red}}$. This is obtained as the quotient of the universal deformation of $(\overline{\fM},\overline{\beta})$ over $Y^{\mu,\tau'}$ by imposing the condition $A^{(j)}=A^{(j+f)}$.

\subsection{Tame deformation rings}
Throughout this section, we assume that $\tau$ is weakly generic.
Let $(\overline{\fM},\overline{\iota})\in Y^{\mu, \tau}(\F)$.  Fix a gauge basis $\overline{\beta}$ of $(\overline{\fM},\iota)$, which exists by Proposition \ref{gaugeuniqueBC}. We define the same deformation problems from Definition \ref{defprob} but with $Y^{\mu, \tau}$ as in Definition \ref{defn:fixed}, using $T_{dd'}^*$ (Proposition \ref{proposition:extension}) in place of $T_{dd}^*$, and with the notion of gauge basis as in Definition \ref{gauge BC}.
For instance, we now have
\begin{eqnarray*}
{D}^{\tau, \Box}_{\overline{\fM},\rhobar}(R)&\defeq& \left\{
\begin{matrix}(\fM_R,\iota, \rho_R,\delta_R)\ \mid\ (\fM_R,\iota)\in Y^{\mu, \tau}_{\overline{\fM}}(R), \rho_R \in {D}_{\rhobar}^{\tau,\Box}(R)\\
\hspace{3cm} \text{and}\ \delta_R:\ T_{dd'}^*(\fM_R)\stackrel{\sim}{\ra}\rho_R|_{G_{K_{\infty}}}
\end{matrix}\right\}
\end{eqnarray*}
We obtain a diagram analogous to (\ref{defdiagram}). We stress that the universal Kisin module living over $R^{\tau, \overline{\beta}}_{\overline{\fM}}$ is a Kisin module of type $(\mu,\tau')$, and $R^{\tau, \overline{\beta}, \nabla}_{\overline{\fM}}$ is obtained by imposing the monodromy condition on a Kisin module of type $(\mu,\tau')$. The arguments from \S \ref{subsection:PCDR} will largely go through, so we will only discuss the modifications that need to be made.  

The fact that the map $\Spf R_{\overline{\fM},\rhobar}^{\tau, \Box} \ra \Spf R^{\mu, \tau}_{\rhobar}$ is an isomorphism is due to the following: By Corollary \ref{RKisin}, there is a unique Kisin module $\fM$ of type $(\mu,\tau')$ living over $\Spf R^{\mu, \tau}_{\rhobar}$. What needs to be checked is that there is a unique isomorphism $ \iota: (\sigma^f)^{*}(\fM) \stackrel{\sim}{\ra} \fM $ verifying the cocycle condition. Proposition \ref{proposition:extension} shows that the extension to $G_{K_\infty}$ of $T^*_{dd}(\fM)$ given by the universal Galois deformation corresponds exactly to an isomorphism  $\iota: (\sigma^f)^{*}(\cM) \stackrel{\sim}{\ra} \cM $ verifying the cocycle condition, where $\cM$ is the \'etale $\phz$-module of $\fM$. But the uniqueness of $\fM$ shows that this $\iota$ respects $\fM\subset \cM$.

The fact that adding gauge bases on $(\fM,\iota)$ is a formally smooth operation follows from Proposition \ref{gaugeuniqueBC}.

The analogue of Theorem \ref{thm:factors} holds by the following Lemma:
\begin{lemma}
Let $E'/E$ be a finite extension.  Let $V_{E'}$ be a continuous representation of $G_{K_{\infty}}$. Then $V_{E'}$ extends to a potentially $($for $L'/K)$ crystalline representation of $G_{K}$ if only if $V_{E'}$ extends to a potentially $($for $L'/K')$ crystalline representation of $G_{K'}$.  
\end{lemma} 
\begin{proof}
This is a consequence of the fact that the restriction from crystalline $G_{L'}$-representations to $G_{L'_{\infty}}$ is fully faithful which is Corollary 2.1.14 in \cite{KisinFcrys}. 
\end{proof} 

We deduce, with the same hypotheses as in Corollary \ref{dring}, that 
$$
R^{\mu, \tau}_{\rhobar}[\![S_1, \ldots, S_{3f}]\!] \cong R^{\tau, \overline{\beta}, \nabla}_{\overline{\fM}}[\![T_1,\ldots, T_8]\!].
$$
Finally, we deduce an explicit description of $R^{\tau, \overline{\beta}, \nabla}_{\overline{\fM}}$ as in Section \ref{Explicit deformation rings}.
\begin{thm} \label{thm:dringBC}  Let $\tau$ be a generic type.
If $\widetilde{w}_{j} \notin \{ \alpha,\, \beta,\, \gamma,\,\mathrm{id} \}$ for all $j$, then
$$
R^{\tau, \overline{\beta}, \nabla}_{\overline{\fM}}\cong \widehat{\otimes}_{j \in \{0, \ldots, f - 1\}} R_{\overline{\fM},\widetilde{w}_j}^{\expl, \nabla}
$$
where $R_{\overline{\fM},\widetilde{w}_j}^{\expl, \nabla}$ is as in Table \ref{table:withmon}, using $e'$ in place of $e$ and~$\bf{a}_k'^{(j)}$ in place of $\bf{a}_k^{(j)}$. 
\end{thm} 
\begin{proof} Let $R=R^{\tau, \overline{\beta}}_{\overline{\fM}}$. The monodromy condition on $(\fM_R,\iota)\in Y^{\mu, \tau}_{\overline{\fM}}(R)$ is by definition the monodromy condition on $\fM_R\in Y^{\mu, \tau'}_{\overline{\fM}}(R)$ (with $\overline{\fM}$ being considered as an element of $Y^{\mu, \tau'}(\F)$ as well). Let $A^{(j)} = \Mat_{\beta}\big(\phi^{(j)}_{\fM, s'_{j+1}(3)}\big)$. We already saw that $A^{(j+f)}=A^{(j)}$. Furthermore, the isomorphism $\iota$ shows that the monodromy conditions at the $j$-th embedding and the $(j+f)$-th embedding are exactly the same.
\end{proof}

\section{Applications}
\label{sec:appl}

In this section, we apply the descriptions of the deformation rings to modularity lifting and the Serre weight conjectures.  Before stating the main theorems, we describe a global setup, the particulars of which are not so important.  The proofs of the main theorems only rely on the existence of patched modules satisfying the axioms spelled out in Definition \ref{minimalpatching}.

\subsection{Global setup}
\label{subsec:glsetup}

Let $F/\Q$ be a CM field with maximal totally real subfield $F^+\neq \Q$ and write $\Sigma_p^+$ (resp. $\Sigma_p$) for the places of $F^+$ (resp. of $F$) lying above $p$.
Let $c$ denote the generator of $\Gal(F/F^+)$ and assume that for all places $v\in \Sigma_p^+$, $v$ decomposes as $ww^c$ in $F$.

Let $G_{/F^+}$ be a reductive group which is an outer form for $\GL_3$ which is quasi-split at all finite places of $F^+$  and which splits over $F$.
Suppose that $G(F^+_v) \cong U_3(\R)$ for all $v|\infty$.
Recall from \cite[\S 7.1]{EGH} that $G$ admits a reductive model $\cG$ defined over $\cO_{F^+}[1/N]$, for some $N\in \N$ which is prime to $p$, together with an isomorphism
\begin{equation}
\label{iso integral}
\iota:\,\cG_{/\cO_{F}[1/N]} \stackrel{\iota}{\rightarrow}{\GL_3}_{/\cO_{F}[1/N]}
\end{equation}
which specializes to
$
\iota_w:\,\cG(\cO_{F^+_v})\stackrel{\sim}{\rightarrow}\cG(\cO_{F_w})\stackrel{\iota}{\rightarrow}\GL_3(\cO_{F_w})
$
for all places  $v\in \Sigma_p^+$.

Define $F_p^+:= F^+\otimes_{\Q}\Q_p$ and $\cO_{F^+,p}:=\cO_{F^+}\otimes_\Z\Z_p$. If $W$ is a finite $\cO$-module endowed with a continuous action of $\cG(\cO_{F^+,p})$ and 
$U\leq G(\A_{F^+}^{\infty,p})\times\cG(\cO_{F^+,p})$ is a compact open subgroup, the space of algebraic automorphic forms on $G$ of level $U$ and coefficients in $W$ is the $\cO$-module defined as:
\begin{equation}
S(U,W) \defeq \left\{f:\,G(F^{+})\backslash G(\A^{\infty}_{F^{+}})\rightarrow W\,|\, f(gu)=u_p^{-1}f(g)\,\,\forall\,\,g\in G(\A^{\infty}_{F^{+}}), u\in U\right\}.
\end{equation}

We recall that the level $U$ is said to be \emph{sufficiently small} if for all $t \in G(\bA^{\infty}_{F^+})$, the order of the finite group $t^{-1} G(F^+) t \cap U$
is prime to $p$.
For a finite place $v$ of $F^+$ that splits in $F$, we say that $U$ is \emph{unramified} at $v$ if one has a decomposition $U=\cG(\cO_{F_v^+})U^{v}$ for some compact open subgroup $U^v\leq G(\A^{\infty,v}_{F^+})$. If $w$ is a finite place of $F$ we say, with an abuse, that $w$ is an unramified place for $U$ if its restriction $w|_{F^+}$ is unramified for~$U$.

Let $\cP_U$ be the set of finite places $w$ of $F$ such that $v\defeq w|_{F+}$ is split in $F$, $v\nmid p$ and $U$ is unramified at $v$. For any subset $\cP\subseteq \cP_U$ of finite complement that is closed under complex conjugation, we write $\bT_{\cP}=\cO[T^{(i)}_w,\,\,w\in\cP,\,i\in\{0,1,2,3\}]$ for the universal Hecke algebra on $\cP$.
The space of algebraic automorphic forms $S(U,W)$ is endowed with an action of $\bT_{\cP}$, where $T_w^{(i)}$ acts by the usual double coset operator
$$
\iota_w^{-1}\left[ \GL_3(\cO_{F_w}) \left(\begin{matrix}
      \varpi_{w}\mathrm{Id}_i & 0 \cr 0 & \mathrm{Id}_{3-i} \end{matrix} \right)
\GL_3(\cO_{F_w}) \right].
$$

A \emph{Serre weight} (\emph{for} $\cG$) is an isomorphism class of a smooth, absolutely irreducible representation $V$ of $\cG(\cO_{F^+,p})$.
If $v|p$ is a place of $F^+$, a \emph{Serre weight at $v$} is an isomorphism class of a smooth, absolutely irreducible representation $V_v$ of $\cG(\cO_{F^+_v})$. 
Finally, if $w|p$ is a place of $F$, a \emph{Serre weight at $w$} is an isomorphism class of a smooth, absolutely irreducible representation $V_w$ of $\GL_3(\cO_{F_w})$.
Note that if $V_v$ is a Serre weight at a place $v$ such that $v=ww^c$ in $F$, then the Serre weights at $w^c$ defined by $V_v\circ\iota_w^{-1}\circ c$, $V_{v}\circ \iota_{w^c}^{-1}$ are dual to each other. Any Serre weight $V$ for $\cG(\cO_{F^+,p})$ can be written as $V\cong \underset{v|p}{\bigotimes}V_v$ where $V_v$ are Serre weights at $v$.

\begin{defn}
\label{definition modularity}
Let $\overline{r}:G_F\rightarrow \GL_3(\F)$ be a continuous Galois representation and let $V$ be a Serre weight for $\cG$. We say that $\overline{r}$ is \emph{automorphic of weight $V$} (or that $V$ is a Serre weight of $\overline{r}$) if there exists a compact open subgroup $U$ of $G(\bA^{\infty,p}_F)\times \cG(\cO_{F^+,p})$ which is unramified at places $v|p$, and a cofinite subset $\cP \subset \cP_U$ such that
$$
S(U,V)_{\overline{\mathfrak{m}}}\neq0
$$
where $\overline{\mathfrak{m}}$ is the kernel of the system of Hecke eigenvalues $\overline{\alpha}:\bT_{\cP}\rightarrow \F$ associated to $\overline{r}$, and $\overline{\alpha}$ satisfies the equality
$$
\det\left(1-\overline{r}^{\vee}(\mathrm{Frob}_w)X\right)=\sum_{j=0}^3 (-1)^j(\mathbf{N}_{F/\Q}(w))^{\binom{j}{2}}\overline{\alpha}(T_w^{(j)})X^j
$$
for all $w\in \cP$.
We write $W(\rbar)$ for the set of all Serre weights of $\rbar$.
We say that $\rbar$ is \emph{automorphic} if $W(\rbar)\neq \emptyset$.
\end{defn}

From now until the end of this subsection, we assume that $p$ splits completely in $F$. If $w|p$ is a place of $F$ and $w|_{F^+} = v$, following \cite{GHS} we write $(X_1^{(3)})_v$ for the set of $p$-restricted pairs $\{\un{a}_w,\un{a}_{w^c}\} \subset \Z^3$ such that $a_{i,w}+a_{2-i,w^c}=0$  for all $0\leq i\leq 2$ (recall that $p$-restricted means that $p-1\geq a_{i,w}-a_{i+1,w}\geq 0$ for $i\in\{0,1\}$). To a $p$-restricted element $\un{a}_w \in \Z^3$, we associate an irreducible representation $F_{\un{a}_w}$ of $\GL_3(k_w)$ and, by inflation, $\GL_3(\cO_{F_w})$ as in \cite[\S 3.1]{GHS} (cf. also \cite[(4.1.3)]{EGH}).
To an element $\un{a}_v = \{\un{a}_w,\un{a}_{w^c}\}\in (X_1^{(3)})_v$, we associate an irreducible representation $F_{\un{a}_v}\defeq F_{\un{a}_w}\circ \iota_w$ of $\cG(\cO_{F^+_v})$ that is independent of the choice of place $w$ dividing $v$.
Let $(X_1^{(3)})_0^{\Sigma_p}$ be the set of $\un{a} = (\un{a}_w)_{w|p}$ where $\{\un{a}_w,\un{a}_{w^c}\} \in (X_1^{(3)})_v$.
Given an element $\un{a} \in (X_1^{(3)})^{\Sigma_p}_0$, we associate an irreducible representation $F_{\un{a}}\defeq \underset{v\vert p}{\bigotimes} F_{\underline{a}_v}$ of $\cG(\cO_{F^+,p})$, or in other words a Serre weight for $\cG$.
All Serre weights are of the form $F_{\un{a}}$ for some $a \in (X_1^{(3)})_0^{\Sigma_p}$ and $F_{\un{a}} \cong F_{\un{a}'}$ if and only if $\un{a} \sim \un{a}'$ as in Section 3.1 of \cite{GHS}.  

\begin{defn}\label{definition genericity}  If $w\vert p$ and $\un{a}_w = (a,b,c) \in \Z^3$ is $p$-restricted, let $F(a,b,c) := F_{\un{a}_w}$ be the corresponding weight at $w$.   Then $F(a,b,c)$ is \emph{lower alcove} if $a - c < p-2$ and it is in the \emph{upper alcove} if $a- c > p - 2$.  
We say that $F(a,b,c)$ is \emph{regular} if $0\leq a-b,\, b-c<p$.
Following \cite[Theorem 5.2.5]{EGH}, we say that $F(a,b,c)$ is \emph{reachable} if either
\[
a-c\leq p-4 
\]
or 
\[
a-b , b- c \leq p-6 \text{ and } a - c \geq p+2.
\]
An inspection on Table \ref{table:JH} shows that $\JH(\sigmabar(\tau))$ consists of reachable weights as soon as $\tau$ is $6$-generic. 

Let $v|p$ be a place of $F^+$.
From the definition of $(X_1^{(3)})_v$, if $\un{a}_v = \{\un{a}_w,\un{a}_{w^c}\}\in (X_1^{(3)})_v$, then $F_{\un{a}_w}$ is regular (resp. ~reachable) if and only if $F_{\un{a}_{w^c}}$ is regular (resp. ~reachable.
For $\un{a}_v = \{\un{a}_w,\un{a}_{w^c}\}\in (X_1^{(3)})_v$, we say that $F_{\un{a}_v}$ is regular (resp. ~reachable) if $F_{\un{a}_w}$ is regular (resp. ~reachable).
Finally, if $\un{a} \in (X_1^{(3)})^{\Sigma_p}_0$, we say that $F_{\un{a}}=\underset{v\vert p}{\bigotimes} F_{\underline{a}_v}$ is regular (resp. ~reachable) if $F_{\underline{a}_v}$  is regular (resp. ~reachable)  for all $v\vert p$. 
If $\overline{r}:G_F\rightarrow \GL_3(\F)$ is as in Definition \ref{definition modularity}, we write $W_{reg}(\rbar)$ (resp. $W_{elim}(\rbar)$) to denote the set of regular (resp. reachable) elements of $W(\rhobar)$.
\end{defn} 


We now recall the tame types for $\GL_3(\Qp)$. 
Let $\tau:I_{\Qp}\rightarrow \cO^{\times}$ be a tame inertial type.  We define a $\GL_3(\F_p)$-representation $\sigma(\tau)$, valued in $E$, via  the ``inertial local Langlands correspondence'' (cf. \cite[Theorem 3.7]{CEGGPS}).
For each tame type $\tau$, $\sigma(\tau)$ is given in Table \ref{ILL}. 


If $\sigma(\tau)^{\circ}$ is  a $\GL_3(\F_p)$-stable $\cO$-lattice inside $\sigma(\tau)$, we write $\JH(\sigmabar(\tau))$ to denote the set of Jordan--H\"older constituents of $\overline{\sigma}(\tau)^{\circ}\defeq \sigma(\tau)^{\circ}\otimes_{\cO}\F$. The set $\JH(\sigmabar(\tau))$ does not depend on the choice of the lattice $\sigmabar(\tau)^{\circ}$.
When $\tau$ is weakly generic, the set $\JH(\sigmabar(\tau))$ consists of nine Serre weights which we list in Table \ref{table:JH}.

\subsection{Modularity lifting and Serre weight conjectures} 
\label{sbsec:ModLift}
We are now ready to state our main theorems.  
Fix once and for all an isomorphism $\imath:\Qpbar\stackrel{\sim}{\ra}\bC$.

If $r: G_F\rightarrow \GL_3(E)$ is a continuous Galois representation, we say (following \cite{BLGG}) that $r$ is \emph{automorphic} if there exists a RACSDC representation $\pi$ of $\GL_3(\bA_{F})$ such that $r\otimes_E\Qpbar\cong r_{\imath}(\pi)$ where $r_{\imath}(\pi):G_F\rightarrow \GL_3(\Qpbar)$ is the continuous representation attached to $\pi$ by \cite[Theorem 2.1.2]{BLGG}.
\begin{defn}
\label{TWconditions}
Let $\rbar: G_F\rightarrow \GL_3(\F)$ be a continuous Galois representation. We say that $\rbar$ satisfies the Taylor-Wiles conditions if
\begin{itemize}
	\item \label{adequate} $\rbar$ has image containing $\GL_3(\F_0)$ for some $\F_0 \subset \F$ with $\#\F_0>9$.
	\item $\overline{F}^{\ker \mathrm{ad} \rbar}$ does not contain $F(\zeta_p)$.
\end{itemize}
\end{defn}

From now on, we further assume that
\begin{itemize}
\item the extension $F/F^+$ is unramified at all finite places; and
\item If $\rbar:G_F \rightarrow \GL_3(\F)$ is ramified at a place $w$ of $F$, then $v = w|_{F^+}$ splits as $ww^c$ (split ramification).
\end{itemize}
We make these two assumptions in order to construct a minimal patching functor in Section \ref{sec:WMPM}. These assumptions can be removed by using not necessarily minimal patched modules, but we avoid this for ease of exposition.

\begin{thm} \label{thm:modularity}
Let $r: G_F\rightarrow \GL_3(E)$ be an absolutely irreducible Galois representation and write $\rbar$ for the reduction of a $G_F$-stable $\cO$-lattice in $r$.

Assume that:
\begin{enumerate}
	\item $p$ splits completely in $F^+$;
	\item $r$ is unramified almost everywhere and satisfies $r^c\cong r^{\vee} \epsilon^{-2}$;
	\item for all places $w\in \Sigma_p$, the representation $r|_{G_{F_w}}$ is potentially crystalline, with parallel Hodge type $(2,1,0)$ and with strongly generic tame inertial type $\tau_{\Sigma_p^+}=\otimes_{v\in \Sigma_p^+}\tau_v$ $($cf. Definition \ref{gencond}$)$;
	\item $\rbar$ verifies the Taylor-Wiles conditions $($cf. Definition \ref{TWconditions}$)$ and $\rbar$ has split ramification;
	\item $\rbar\cong \rbar_{\imath}(\pi)$ for a RACSDC representation $\pi$ of $\GL_3(\bA_{F})$ with trivial infinitesimal character such that $\otimes_{v\in \Sigma_p^+}\sigma(\tau_v)$ is a $K$-type for $\otimes_{v\in \Sigma_p^+}\pi_v$.
\end{enumerate}
Then $r$ is automorphic.
\end{thm}

\begin{rmk}  Note that we do not make any potentially diagonalizability assumption. In fact, we do not know whether or not $r|_{G_{F_w}}$ in the theorem is
potentially diagonalizable. We also do not assume that $\rbar|_{G_{F_w}}$ has any particular form.   
\end{rmk}

\begin{rmk} The first assumption and the strong genericity condition can both be relaxed if one assumes that at each place $w$ the shape is not one of $\{ \alpha, \beta, \gamma, \mathrm{id} \}$.  The difficulty comes from the absence of a general, explicit description of the deformation ring (rather than its special fiber) in those cases, where we need the Serre weight conjectures as input to show Theorem \ref{thm:conngenfib}.
\end{rmk}

Theorem \ref{thm:modularity} is a consequence of the following Theorem using standard Kisin-Taylor-Wiles patching methods. (Note that, in the setup of Theorem \ref{thm:modularity}, the representation $\rbar|_{G_{F_w}}$ satisfies the hypotheses of Theorem \ref{thm:conngenfib} for all $w\in\Sigma_p$.)
\begin{thm} \label{thm:conngenfib} Let $\rhobar:G_{\Qp} \ra \GL_3(\F)$ be a continuous Galois representation and let $\tau$ be a generic tame inertial type such that $R^{(2,1,0), \tau}_{\rhobar}\neq0$.   If $\bf{w}(\rhobar, \tau) \in \{ \alpha, \beta, \gamma, \mathrm{id} \}$, assume furthermore that $\tau$ is strongly generic.  Then the framed potentially crystalline deformation ring $R^{(2,1,0), \tau}_{\rhobar}$ with Hodge-Tate weights $(2,1,0)$ has connected generic fiber. 
\end{thm} 
\begin{proof}
If $\bf{w}(\rhobar, \tau) \notin \{ \alpha, \beta, \gamma, \mathrm{id} \}$, then this is immediate upon inspection of Table \ref{table:withmon}.
The remaining cases will be proved in \S \ref{sec:badcases}.
\end{proof}

If $\rhobar:G_{\Qp}\rightarrow \GL_3(\F)$ is a continuous semisimple Galois representation an explicit set of weights $W^{?}(\rhobar|_{I_{\Qp}})$ is defined in \cite[Conjecture 6.9]{herzig-duke}. The main conjecture in \emph{loc.~cit.}~ is that $W^{?}(\rhobar|_{I_{\Qp}})$ should give the set of regular modular weights.  More precisely, fix a place $\tld{v}$ above each $v \in \Sigma_p^{+}$, we prove the following generalization of the weight part of Serre's conjecture as conjectured in \cite[Conjecture 6.9]{herzig-duke} (cf. \S \ref{patching}):
\begin{thm} \label{thm:serreweightconj}
Assume that $p$ splits completely in $F$.
Let $\rbar:G_F\rightarrow \GL_3(\F)$ be a continuous  Galois representation, verifying  the Taylor-Wiles conditions. Assume that $\rbar|_{G_{F_{\tld{v}}}}$ is semisimple and $8$-generic $($Definition $\ref{defn n-gen})$ for all $v\in \Sigma_p^+$, that $\rbar$ is automorphic of some reachable Serre weight, and that $\rbar$ has split ramification outside $p$. Then 
$$
\underset{v\in \Sigma_p^+}{\bigotimes}F_{\un{a}_v}\in W_{elim}(\rbar)\Longleftrightarrow F_{\un{a}_v}\circ\iota_{\tld{v}}^{-1}\in W^{?}(\rbar|_{I_{F_{\tld{v}}}}) \quad\text{for all}\quad v\in \Sigma_p^+.
$$
\end{thm}

\begin{rmk} Theorem \ref{thm:serreweightconj} is stated only for $\rbar$ which are semisimple above $p$ because those are the only representations for which there is an explicit conjecture.  Our computations together with work of \cite{HLM}, \cite{MP} suggest a set $W^{?}(\rbar|_{G_{F_{\tld{v}}}})$ for non-semisimple $\rbar|_{G_{F_{\tld{v}}}}$ for which the analogue of Theorem \ref{thm:serreweightconj} should hold. One example is worked out in Proposition \ref{Prop-non-ss}.  
A complete analysis for a set $W^{?}(\rbar|_{G_{F_{\tld{v}}}})$ when $\rbar|_{G_{F_{\tld{v}}}}$ is not semisimple is carried out in \cite{LLLM3}.
\end{rmk}

\begin{rmk}  \label{rem:WE}
The restriction to reachable weights in the statement of Theorem \ref{thm:serreweightconj} is due to the current weight elimination results. 
For the niveau 1 and 2 case the works \cite{HLM}, \cite{MP}, \cite{LMP} provide weight elimination for all weights (not just reachable ones). Specifically, \cite[\S 2.5]{HLM} and \cite{MP} deals with the niveau 1 case (when $(\rbar|_{G_{F_{\tld{v}}}})^{ss}$ is $3$-generic) and \cite[\S 3]{LMP} with the niveau 2 case (when $(\rbar|_{G_{F_{\tld{v}}}})^{ss}$ is 4-generic).
For the niveau 3 case the weight elimination for \emph{reachable} weights has been established in \cite[Theorem 5.2.5]{EGH} (this is the reason for our definition and restriction to reachable weights). This has been improved by unpublished work from John Enns (private communication) to eliminate also non-reachable weights when $\rbar|_{G_{F_{\tld{v}}}}$ is irreducible at places above $p$ and $9$-generic. In particular, granting the work of John Enns, Theorem \ref{thm:serreweightconj} would hold for $W(\rbar)$ instead of $W_{elim}(\rbar)$ and assuming only that $\rbar$ is automorphic.

On the other hand, our arguments and the elimination results of \cite{LLL} suffice to show Theorem \ref{thm:serreweightconj} when $W_{elim}(\rbar)$ is replaced by $W_{reg}(\rbar)$ (which is the analogue of Conjecture 6.9 in \cite{herzig-duke} in our setting).
If we make the stronger assumption that $\rhobar$ is $9$-generic, we can even replace $W_{reg}(\rbar)$ by $W(\rbar)$.
\end{rmk}

The rest of \S 7 is devoted to the proof of Theorem \ref{thm:serreweightconj}, which uses the Breuil-M\'ezard philosophy introduced in \cite{gee-kisin}.  Namely, we use the descriptions of the special fibers of deformation rings to determine the Hilbert-Samuel multiplicities of minimal patched modules. The argument only requires Theorem \ref{thm:conngenfib} in the case where $\bf{w}(\rhobar, \tau)$ has length at least 2, and thus makes no use of the results in \S 8. 

Assuming first that $\rbar$ is modular of a lower alcove weight, we use an inductive argument involving carefully chosen tame types to prove modularity of the shadow weights (Proposition \ref{prop:intersec}).
A slightly more intricate argument shows that if $\rbar$ is modular, then it is modular of a lower alcove weight.

\subsection{Weak minimal patched modules}
\label{sec:WMPM}

As before, let $F/F^+$ be a CM extension.
With an eye towards future applications, in this subsection, we assume that every place $v|p$ of $F^+$ splits in $F$, but nothing more about the splitting behavior at $p$.
Let $\rbar: G_F \rightarrow \GL_3(\F)$ be a Galois representation.
For each place $v|p$ of $F^+$, fix a place $\widetilde{v}$ of $F$ such that $\widetilde{v}|_{F^+} = v$.
Let $R_{\widetilde{v}}^\square$ denote the unrestricted universal $\cO$-framed deformation ring of $\rbar|_{G_{F_{\widetilde{v}}}}$.
Fix a natural number $h$ and let 
\[R_\infty = \Big( \widehat{\underset{v\in \Sigma_p^+}{\bigotimes}} R_{\widetilde{v}}^\square \Big)[\![x_1,x_2,\ldots, x_h]\!] \textrm{ and } X_\infty = \Spf R_\infty.\]
If $\tau_{\widetilde{v}}$ is an inertial type for $G_{F_{\widetilde{v}}}$, then let $R_{\widetilde{v}}^{\square,\tau_{\widetilde{v}}}$ be the universal $\cO$-framed potentially crystalline deformation ring of $\rbar|_{G_{F_{\widetilde{v}}}}$ of inertial type $\tau_{\widetilde{v}}$ (and $p$-adic Hodge type $(2,1,0)$).
If $\tau = \underset{v\in \Sigma_p^+}{\bigotimes} \tau_{\widetilde{v}}$, then let
\[R_\infty(\tau) = \Big( \widehat{\underset{v\in \Sigma_p}{\bigotimes}} R_{\widetilde{v}}^{\square,\tau_{\widetilde{v}}} \Big) [\![x_1,x_2,\ldots, x_h]\!] \textrm{ and } X_\infty(\tau) = \Spf R_\infty(\tau).\]
Let $d+1$ be the dimension of $R_\infty(\tau)$ (the dimension is independent of $\tau$ by Theorem 3.3.4 of \cite{KisinPSS}). Note that $X_\infty(\tau)[1/p]$ is regular by \cite[Theorem 3.3.8]{KisinPSS}.
We denote by $\overline{R}_{\widetilde{v}}^\square$, $\overline{R}_\infty$, etc. the reduction of these objects modulo $\varpi$.
The following definition is adapted from Definition 4.1.1 of \cite{GHS}.

\begin{defn}\label{minimalpatching}
A \emph{weak minimal patching functor for $\rbar$} is defined to be a covariant exact functor 
$M_{\infty}:\Rep_{K}(\cO)\ra \Coh(X_{\infty})$ satisfying the following axioms:
\begin{enumerate}
	\item Let $\tau\defeq \underset{v\in \Sigma^+_p}{\bigotimes}\tau_{\tld{v}}$, where for all $v\in \Sigma^+_p$, $\tau_{\tld{v}}$ is an inertial type,
and let $\sigma(\tau)\defeq \underset{v\in \Sigma^+_p}{\bigotimes}\sigma(\tau_{\tld{v}})\circ\iota_{\tld{v}}$ be the associated $K$-type as in \cite[Theorem 3.7]{CEGGPS}.
If $\sigma(\tau)^{\circ}$ an $\cO$-lattice in it, then $M_{\infty}(\sigma(\tau)^{\circ})$ is $p$-torsion free and is maximally Cohen-Macaulay over $R_\infty(\tau)$; \label{support}
	\item if $V=\underset{v\in \Sigma^+_p}{\bigotimes}V_{v}$, where for all $v\in \Sigma^+_p$ the $V_v$ are irreducible $\cG(k_v)$-representations over $\F$ (i.e.~$V$ is a Serre weight for $\cG$), the module $M_{\infty}(V)$ has nonempty support if and only if $\rbar$ is automorphic of weight $V$; furthermore if $M_{\infty}(V)\neq 0$ then its support is equidimensional of dimension $d$; and 
	\label{dimd}
	\item the sheaf $M_\infty(\sigma(\tau)^\circ)[1/p]$ over $X_\infty(\tau)[1/p]$ (which is locally free, being maximal Cohen-Macaulay over a regular scheme) has rank at most one on each connected component. \label{minimal}
\end{enumerate}
\end{defn}

\begin{rmk}
The adjective ``weak" corresponds to the fact that $M_\infty(\sigma(\tau)^\circ)$ is not assumed to have full support on $X_{\infty}(\tau)$ for all inertial types $\tau$ in contrast to Definition 4.1.1 of \cite{GHS}.
\end{rmk}

\begin{rmk} \label{minimalcondition}
The adjective ``minimal" corresponds to the multiplicity one property in condition (\ref{minimal}).
Our results on automorphy of global Serre weights could be proved without requiring minimality using the geometric perspective of \cite{EG}, but we have avoided this for ease of exposition.
\end{rmk}

Given a Noetherian ring $R$ and an $R$-module $M$, we denote the Hilbert-Samuel multiplicity of $M$ by $e(M,R)$.
If $R = R_\infty$, let $e(M) = e(M,R_\infty)$.
The following proposition is the key to relating automorphy of global Serre weights to multiplicities of deformation rings.

\begin{prop} \label{upperbound}
If $M_\infty$ is a weak minimal patching functor, then $e(M_\infty(\sigma(\taubar))) \leq e(\overline{R}_\infty(\tau))$, and we have equality if and only if $M_\infty(\sigma(\tau)^\circ)$ has full support on $X_{\infty}(\tau)$ $($for any choice of lattice $\sigma(\tau)^\circ)$.
\end{prop}
\begin{proof}
Let $\T_\infty(\tau)$ be the quotient of $R_\infty(\tau)$ which acts faithfully on $M_\infty(\sigma(\tau)^\circ)$.
Then 
\[e(M_\infty(\sigma(\taubar)^\circ)) = e(\overline{\T}_\infty(\tau)) \leq e(\overline{R}_\infty(\tau))\]
where the equality follows from Definition \ref{minimalpatching}(3) and Corollary 1.3.5 of \cite{kisin-fontaine-mazur} and the inequality follows from the fact that $\dim \T_\infty(\tau) = \dim R_\infty(\tau)$ by Definition \ref{minimalpatching}(1).
The inequality is an equality if and only if $\T_\infty(\tau) = R_\infty(\tau)$ since $R_\infty(\tau)$ is reduced and equidimensional.
\end{proof}


We now construct a weak minimal patching functor for $\rbar$ under some hypotheses using the Taylor-Wiles method.
We write $\Sigma_0$ to denote the finite primes of $F$ where $\rbar$ ramifies and define $\Sigma_0^+\defeq\{w|_{F^+},\ w\in \Sigma_0\}$ and $\Sigma^+\defeq \Sigma_p^+\cup\Sigma_0^+$. 
Assume for the rest of this section that $\rbar$ satisfies the Taylor-Wiles conditions of Definition \ref{TWconditions}.
Note that the first condition, which is stronger than the usual condition of adequacy, allows us (see Section 2.3 of \cite{CEGGPS}) to choose a place $v_1\notin \Sigma^+$ of $F^+$ such that
\begin{itemize}
 \item $v_1$ splits in $F$ as $v_1 = \tld{v}_1 \tld{v}_1^c$;
 \item $v_1$ does not split completely in $F(\zeta_p)$; and
 \item $\overline{\rho}(\Frob_{\tld{v}_1})$ has distinct $\F$-rational eigenvalues, no two of which have ratio $(\mathbf{N}_{F^+/\Q} v_1)^{\pm 1}$.
\end{itemize}
In order to satisfy the minimality condition, recall (cf. \S \ref{sbsec:ModLift}) that we have made the following two further assumptions.
\begin{itemize}
\item The extension $F/F^+$ is unramified at all finite places.
\item If $\rbar:G_F \rightarrow \GL_3(\F)$ is ramified at a place $w$ of $F$, then $v = w|_{F^+}$ splits as $ww^c$.
\end{itemize}

As mentioned in Remark \ref{minimalcondition}, these two assumptions can be removed by working with weak patched modules which are not necessarily minimal.
For $v \in \Sigma_0^+$, let $\tau_{\tld{v}}$ be the type which is minimally ramified with respect to $\rbar|_{G_{F_w}}$ ($\tau_{\tld{v}}$ is the restriction to inertia of the Weil-Deligne representation attached to a Galois representation which is minimal in the sense of Definition 2.4.14 of \cite{CHT}).
Let $R_{\tld{v}}^{\square,\tau_{\tld{v}}}$ be the corresponding universal $\cO$-framed deformation ring of $\rbar|_{G_{F_{\tld{v}}}}$.
Let $R_{\tld{v}_1}^\square$ be the unrestricted universal $\cO$-framed deformation ring of $\rbar|_{G_{F_{\tld{v}_1}}}$.
Let \[R^{\textrm{loc}} = \widehat{\underset{v\in \Sigma_p^+}{\bigotimes}} R^\square_{\tld{v}} \widehat{\otimes}
\widehat{\underset{v\in \Sigma_0^+}{\bigotimes}} R^{\square,\tau_{\tld{v}}}_{\tld{v}} \widehat{\otimes} R^\square_{\tld{v}_1}.\]
Choose an integer $q \geq 3[F^+:\QQ]$ as in Section 2.5 of \cite{CEGGPS}, and let \[R_\infty = R^{\textrm{loc}}[[x_1,\ldots ,x_{q-3[F^+:\QQ]}]].\]
By \cite[Corollary 2.4.21]{CHT}, $R_{\tld{v}}^{\square,\tau}$ is formally smooth over $\cO$ for $v\in \Sigma_0^+$.
By the choice of $v_1$, $R^\square_{\widetilde{v}_1}$ is formally smooth over $\cO$ by \cite[Proposition 2.5]{CEGGPS}.
We conclude that $R_\infty$ is formally smooth over $\widehat{\underset{v\in \Sigma_p^+}{\bigotimes}} R^\square_{\tld{v}}$ and hence there is an isomorphism $R_\infty \cong \widehat{\underset{v\in \Sigma_p^+}{\bigotimes}} R^\square_{\tld{v}}[\![x_1,x_2,\ldots,x_h]\!]$ for some natural number $h$.

One can construct an $R_\infty[\![G(\cO_{F^+,p})]\!]$-module $M_\infty$ as in \cite[Section 4.2]{le} ($p$ is assumed to split in $F$ in \cite{le}, however the construction, results, and proofs of Section 4 extend verbatim).
Then define a covariant functor $M_{\infty}:\Rep_{\cG(\cO_{F^+,p})}(\cO)\ra \Coh(X_{\infty})$ by $M_\infty(W) = \Hom_{\cG(\cO_{F^+,p})}(W,M_\infty^\vee)^\vee$ where $\cdot^\vee$ denotes the Pontriagin dual.

\begin{prop}
\label{prop:WMPF}
If $E$ is sufficiently large, $M_\infty$ is a weak minimal patching functor.
\begin{proof}
This proof is adapted from various proofs in \cite{CEGGPS} and \cite{le}.
While the contexts differ slightly, the proofs apply verbatim.
Note that the definition of $M_\infty(\sigma(\tau)^\circ)$ agrees with the definition given after 4.13 of \cite{CEGGPS} by \cite[Remark 4.15]{CEGGPS}.
This definition guarantees that $M_\infty(\sigma(\tau)^\circ)$ is $p$-torsion free.
Exactness of $M_\infty$ follows from \cite[Proposition 2.10]{CEGGPS} (the choice of place $v_1$ guarantees projectivity).
(\ref{support}) and (\ref{minimal}) are proved similarly to \cite[Lemma 4.18(1)]{CEGGPS}.

If $V$ is a Serre weight for $\cG$, by Theorem 5.2.1(iii) of \cite{HLM} and Nakayama's lemma (using Theorem 5.2.1(i) of \cite{HLM}), $M_\infty(V)$ is nonzero if and only if $\rbar$ is automorphic of weight $V$.
By Theorem 4.1.4(2) of \cite{le}, $M_\infty(V)$ is maximal Cohen-Macaulay of depth $d$, which shows (\ref{dimd}).
\end{proof}
\end{prop}

\subsection{Shapes and Serre weights}
\label{patching}

We now prove Theorem \ref{thm:serreweightconj}.  The key ingredient in the proof of Theorem \ref{thm:serreweightconj} is the description of the deformation rings in Table \ref{table:withmon} and combinatorics of the sets $W^{?}(\rhobar)$ and $\JH(\sigmabar(\tau))$ for generic tame types $\tau$.  We first recall the notion of shadow weight. For $\rhobar:G_{\Qp} \ra \GL_3(\overline{\F})$ semisimple and weakly generic, $W^{?}(\rhobar)$ contains a set of six obvious weights denoted $W_{\mathrm{obv}}(\rhobar)$ (three upper alcove weights and three lower alcove weights, cf. \cite[Definition 7.1.4 ]{GHS}) and an additional three weights called shadow weights, as summarized in Table \ref{table:SW}.  Each lower alcove weight $F(a,b,c) \in W_{\mathrm{obv}}(\rhobar)$ has a corresponding \emph{shadow weight} $F(p-2 + c, b, a- p+2) \in W^{?}(\rhobar)$ (cf. \cite[Definition 7.2.3]{GHS}). 

The following combinatorial result matches shapes with predicted Serre weights.  
The terms shadow and non-shadow shapes are defined in Table \ref{Table admissible elements}.

\begin{prop}
\label{prop:intersec} Let $\rhobar:G_{\Qp} \ra \GL_3(\F)$ be semisimple and $n$-generic with $n\geq 5$. 
\begin{enumerate}
\item If $F(a, b, c) \in W_{\mathrm{obv}}(\rhobar|_{I_{\Qp}})$, then there is a $(n-2)$-generic tame type $\tau$ such that $\JH(\sigmabar(\tau))\cap W^{?}(\rhobar|_{I_{\Qp}})= \{F(a, b, c)\}$ and  $\mathbf{w}(\rhobar, \tau)$ has length 4.

\item If $F(p-2 +c, b, a - p +2)\in W^{?}(\rhobar|_{I_{\Qp}})$ is a shadow weight, then there is a $(n-2)$-generic tame type $\tau$ such that $\JH(\sigmabar(\tau))\cap W^{?}(\rhobar|_{I_{\Qp}})= \{F(a, b, c), F(p-2 +c, b, a - p +2)\} $ and $\mathbf{w}(\rhobar, \tau)$ is a length 3 \emph{shadow} shape. 

\item If $F(a,b,c) \in W^{?}(\rhobar|_{I_{\Qp}})$ is an obvious upper $($resp. lower$)$ alcove weight, then there is a $(n-2)$-generic tame type $\tau$ such that $\JH(\sigmabar(\tau))\cap W^{?}(\rhobar|_{I_{\Qp}})$ contains $F(a,b,c)$ plus one other obvious lower $($resp. upper$)$ alcove weight and $\mathbf{w}(\rhobar, \tau)$ is a length 3 non-shadow shape.    

\item If $F(p-2 +c, b, a - p +2)\in W^{?}(\rhobar|_{I_{\Qp}})$ is a shadow weight, then there is a $(n-2)$-generic tame type $\tau$ such that $\JH(\sigmabar(\tau))\cap W^{?}(\rhobar|_{I_{\Qp}})$ contains $F(p-2 +c, b, a - p +2)$ plus three obvious weights of $ W^{?}(\rhobar|_{I_{\Qp}})$ and $\mathbf{w}(\rhobar, \tau)$ has length 2. 
\end{enumerate}
\end{prop}
\begin{proof}

The strategy is the same in all four cases so we focus on the proof of (2). The type $\tau$ for (2) is given in Table \ref{table: types}.  The table is constructed starting with a tame type $\tau$ over $\Qp$.   For each shape $\widetilde{w} \in \{ \alpha \beta \alpha, \gamma \beta \gamma, \alpha \gamma \alpha \}$, one can consider the mod $p$ Kisin modules of shape $\widetilde{w}$ as in Table \ref{table shapes mod p}. We  consider the special Kisin modules of this shape where $A_{\widetilde{w}}$ is a monomial matrix---these Kisin modules give rise to semisimple $\rhobar$. For example, when $\tau = \omega^{-a} \oplus \omega^{-b} \oplus \omega^{-c}$, the Kisin module with 
$$
A_{\alpha \beta \alpha} = \begin{pmatrix} 0&0&v \overline{c}_{13}^* \\0&v\overline{c}_{22}^*&0 \\v\overline{c}_{31}^*& 0&0 \end{pmatrix}
$$
corresponds (under $T^*_{dd}$) to a $G_{{\Q_{p,\infty}}}$-representation which extends to a $G_{\Qp}$-representation $\rhobar$ with $\rhobar|_{I_{\Qp}} = \omega^{b+1} \oplus \Ind(\omega_2^{(a +1) + p(c+1)})$.  
This confirms the shape $\bf{w}(\rhobar,\tau)$ in the 10th row of Table \ref{table: types}.

As another example, consider $\sigma(\tau)=\Ind_{P_1(\Fp)}^{\GL_3(\Fp)}\big(\teich{\omega}^{b+1}\otimes \Theta(\teich{\omega}_2^{(c-1)+pa})\big)$ and parallel shape $\beta\gamma\beta$ which appears in the third section of Table \ref{table: types}. Define $\tau'\defeq \omega_2^{(b+1)+p(b+1)}\oplus\omega_2^{(c-1)+pa}\oplus\omega_2^{a+p(c-1)}$ which has the orientation $s_0'=(12)$ and $s_1'=(132)$.  The special point of $(\overline{\fM}, \iota) \in Y^{\mu, \tau}(\F)$ of shape $\beta \gamma \beta$ has  
$$
A^{(j)}_{\beta \gamma \beta} = \begin{pmatrix}0&\overline{c}_{12}^{*}&0\\ \overline{c}_{21}^{*}v^2&0&0\\0&0& \overline{c}_{33}^{*}v\end{pmatrix}
$$ 
We are only interested in the restriction to inertia and so we can forget $\iota$ and set the constants to 1.  Let $\overline{\cM} \defeq \overline{\fM} \otimes_{ \F[\![u]\!]} \F(\!(u)\!)\in \Phi\text{-}\Mod^{\text{\'et}}_{dd}(\F)$ be the associated \'etale $\phz$-module with descent datum. By Proposition \ref{prop:phif}, the \'etale $(\phz^2,k'\otimes_{\Fp}\F(\!(v)\!))$-module $\overline{\cM}_0 \defeq \varepsilon_0\big((\overline{\cM})^{\Delta'=\Id}\big)$ is described, in an appropriate basis $\mathfrak{f}\defeq (e_1,e_2,e_3)$, as follows:
\begin{eqnarray*}
\Mat(\phi_{\cM})&=&
\left(s_0'A^{(1)}\text{\tiny$\begin{pmatrix}v^{c-1}&0&0\\0&v^{b+1}&0\\0&0&v^a\end{pmatrix}$}s_0'\right)\cdot
\phz\left(s_1'A^{(0)}\text{\tiny$\begin{pmatrix}v^{c-1}&0&0\\0&v^{b+1}&0\\0&0&v^a\end{pmatrix}$}(s_1')^{-1}\right)\\
&=&\text{\tiny$\begin{pmatrix}0&v^{(c+1)+p(a+1)}&0\\0&0&v^{(b+1)+p(c+1)}\\v^{(a+1)+p(b+1)}&0&0\end{pmatrix}$}
\end{eqnarray*}
up to constants. In particular, we see that the $\phz^{6}$-action on $e_1$ is described by $$e_1\mapsto v^{(p^3+1)((c+1)+p(a+1)+p^2(b+1))}e_1.$$ We conclude that $\rhobar := T_{dd'}^*(\overline{\fM})$ is tame with
\begin{eqnarray*}
\rhobar|_{I_{\Qp}} &\cong& \omega_3^{
(a+1)+p(b+1)+p^2(c+1)} \oplus \omega_3^{
(b+1)+p(c+1)+p^2(a+1)} \oplus \omega_3^{
(c+1)+p(a+1)+p^2(b+1)}.
\end{eqnarray*}
We conclude that (up to unramified twist) $\rhobar = \Ind_{G_{\Qpf{3}}}^{G_{\Qp}}\omega_3^{(a+1)+p(b+1)+p^2(c+1)}$ and so $\bf{w}(\rhobar, \tau) = \beta \gamma \beta$, confirming the 9th row of Table \ref{table: types}. 
Each semisimple $\rhobar$ arises in this way from exactly three types. Comparing $\JH(\sigmabar(\tau))$ and $W^{?}(\rhobar|_{I_{\Qp}})$ is a tedious but not difficult computation.  We see that in each case the intersection $\JH(\sigmabar(\tau))\cap W^{?}(\rhobar|_{I_{\Qp}})$ is exactly a lower alcove weight together with its shadow.

A similar computation can be done for the other shapes and cuspidal types using Proposition \ref{prop:phif}.
Briefly, regarding parts (1), (3), and (4), for each tame type $\tau$, there are six shapes of length 4, six non-shadow shapes of length 3, and six shapes of length 2.  For each semisimple $\rhobar$, there are exactly six types for which $\bf{w}(\rhobar, \tau)$ has length 4, six types for which $\bf{w}(\rhobar, \tau)$ has non-shadow length 3 and six types for which $\bf{w}(\rhobar, \tau)$ has length 2.  For each obvious weight, there is a unique type satisfying (1) and two types satisfying (3).  For each shadow weight, there are two types which work for part (4).
\end{proof}

\begin{proof}[Proof of Theorem \ref{thm:serreweightconj}]  For generic tame types $\tau_v$ and $\rhobar_v := \rbar|_{G_{F_{\widetilde{v}}}}$, if $\bf{w}(\rhobar_v, \tau_v)$ has length greater than or equal 2 for all $v \mid p$, then by Theorem \ref{thm:dringBC} and Table \ref{table:withmon},  $R_{\infty}(\tau)$ has connected generic fiber and so if $M_{\infty}(\sigma(\tau)^\circ)$ is nonzero, then it has full support.  By Proposition \ref{upperbound}, we know the Hilbert-Samuel multiplicity of $M_{\infty}(\overline{\sigma}(\tau)^\circ)$.  The strategy is to compute $e(M_{\infty}(\otimes_{v \in \Sigma^+_p} V_v))$ by varying the tame type.  In fact, we show that $e(M_{\infty}(\otimes_{v \in \Sigma^+_p} V_v)) = 1$ whenever $V_v \circ \iota^{-1}_{\tld{v}}\in W^{?}(\rbar|_{I_{F_{\tld{v}}}})$ for all $v$.  

Let $W^?(\rhobar_v)$ be $W^?(\rbar|_{I_{F_{\tld{v}}}}) \circ \iota_{\tld{v}}$.
As usual, this does not depend on the choice of place $\tld{v}|v$.
Similarly define $W_{\mathrm{obv}}(\rhobar_v)$.
Our assumptions imply that if $\rbar$ is modular of a reachable weight $\otimes_{v \in  \Sigma^+_p} V_v$ then $V_v\in W^?(\rhobar_v)$ for all $v\in\Sigma_p^+$ (cf. Remark \ref{rem:WE}), and all weights in $W^?(\rhobar_v)$ are reachable. 

First, assume that $\rbar$ is modular of a reachable weight $\otimes_{v \in  \Sigma^+_p} V_v$ such that all $V_v$ are lower alcove.  Under these hypotheses, the modularity of the obvious weights is known by \cite[Theorem 5.1.4]{BLGG} (alternatively, one could deduce the modularity of the obvious weights using the arguments two paragraphs below and Proposition \ref{prop:intersec}(3)).
Choose $V_v \in W^{?}(\rhobar_v)$ for each $v \in  \Sigma^+_p$.  Let $S^+_p \subset \Sigma^+_p$ be the set of places for which $V_v$ is a shadow weight.  For each $v \in S^+_p$, let $V'_v$ denote the lower alcove weight corresponding to $V_v$. We induct on the size of $S^+_p$.  If $S^+_p$ is empty, then by Proposition \ref{prop:intersec}(1), we can choose types $\tau_v$ for each $v  \in \Sigma^+_p$ such that $\JH(\sigma(\tau_v))\cap W^{?}(\rhobar_v)= \{V_v\}$.  In this case, $M_{\infty}(\otimes_{v \in \Sigma^+_p} \sigmabar(\tau_v)^\circ) = M_{\infty}(\otimes_{v \in \Sigma^+_p} V_v)$ is nonzero by Definition \ref{minimalpatching}(2) and modularity of the obvious weights, and so $e(M_{\infty}(\otimes_{v \in \Sigma^+_p} V_v)) = e(M_{\infty}(\otimes_{v \in \Sigma^+_p} \sigmabar(\tau_v)^\circ)) = 1$ since the corresponding deformation ring is a power-series ring. 

In general, for each $v \in \Sigma^+_p \backslash S^+_p$, we choose $\tau_v$ as in \ref{prop:intersec}(1).  For each $v \in S^+_p$, we choose $\tau_v$ as in \ref{prop:intersec}(2) to contain exactly $V_v$ and $V'_v$. Consider 
$
M_{\infty}(\overline{\sigma}(\tau)) 
$
over $\overline{R}_{\infty}(\tau)$.  The deformation rings for $v \notin S^+_p$ are again formally smooth.   By the previous paragraph, $\rbar$ is modular of all the obvious weights and so  $M_{\infty}(\overline{\sigma}(\tau))$ is nonzero.   By Table \ref{table:withmon} and Proposition \ref{upperbound}, we deduce that 
$$
e(M_{\infty}(\overline{\sigma}(\tau))) = e(\overline{R}_{\infty}(\tau)) = 2^{|S^+_p|}.
$$
By the inductive hypothesis, the contribution to $e(M_{\infty}(\overline{\sigma}(\tau)))$ of any Serre weight in $W^?(\rbar) \cap \JH(\sigmabar(\tau))$ not equal to $\otimes_{v \in \Sigma^+_p} V_v$ is 1.  We deduce that $e(M_{\infty}(\otimes_{v \in \Sigma^+_p} V_v)) = 1$.

Finally, we show that if $\rbar$ is modular of any reachable weight then it is modular of a reachable weight $\otimes_{v \in  \Sigma^+_p} V_v$ where $V_v$ lower alcove for each $v$.  
Assume $\rbar$ is modular of a reachable weight $\otimes_{v} V_v$.  As above, let $S^+_p \subset \Sigma^+_p$ be the set of places for which $V_v$ is a shadow weight.   Assume $|S^{+}_p|$ is minimal among such weights and is nonzero.   Choosing types $\tau_v$ as above, we conclude that $e(M_{\infty}(\overline{\sigma}(\tau)))  = 2^{|S^+_p|}$ and so 
$$
e(M_{\infty}(\otimes_{v \in \Sigma^+_p} V_v)) \leq 2^{|S^+_p|}. 
$$
Choose now a place $v_0 \in S^+_p$, and replace $\tau_{v_0}$  by a tame type $\tau'_{v_0}$ for the shadow weight $V_{v_0}$ as in Proposition \ref{prop:intersec}(4). (Note that since $\rhobar_v$ is $8$-generic, then $\tau'_v$ is $6$-generic; this implies that weights in $\JH(\sigmabar(\tau'_v))$ are reachable.) For the new type $\tau'$, we have
$$
 e(M_{\infty}(\overline{\sigma}(\tau'))) = 2^{|S^+_p| + 1}.
 $$ 
 By comparing multiplicities and using that $M_\infty$ vanishes on reachable weights $\otimes_{v} V'_v$ if $V'_v\notin W^?(\rhobar_v)$ for some $v\in\Sigma_p^+$, we conclude that $M_{\infty}(\otimes_{v} V'_v) \neq 0$ for some Serre weight $\otimes_{v} V'_v$ of $\rbar$ such that $V'_v \in W_{\mathrm{obv}}(\rhobar_v)$ if $V_v \in W_{\mathrm{obv}}(\rhobar_v)$ for all $v$ and $V'_{v_0}\in W_{\mathrm{obv}}(\rhobar_{v_0})$.
This contradicts the minimality of $|S^+_p|$.

To prove that $\rbar$ is modular of a weight $\otimes_{v \in \Sigma^+_p} V'_v$ with $V'_v$ reachable and lower alcove for all $v$, one repeats the argument of the last paragraph beginning with a modular weight $\otimes_{v \in \Sigma^+_p} V_v$ with $V_v$ reachable and obvious for all $v$, defining $S^+_p \subset \Sigma^+_p$ using upper alcove in place of shadow, and defining $\tau$ and $\tau'$ using Proposition \ref{prop:intersec}(1) and (3).
\end{proof}


We provide an example of a calculation of modular reachable Serre weights for a non-semisimple $\rhobar$. This will be needed in \S \ref{sbsb:nsemi} for the proof of Theorem \ref{thm:conngenfib}.

\begin{prop}\label{Prop-non-ss} Let $\tau$ be a strongly generic tame type. Let $\rhobar$ be a non-semisimple representation such that $\bf{w}(\rhobar, \tau) = \alpha$ $($i.e., $\overline{c}'_{22} \neq 0$ in Table $\ref{table shapes mod p})$. Then there is a subset $W^{?}(\rhobar) \subset W^{?}(\rhobar^{ss})$ consisting of six weights 
such that the following holds: 

There exists a CM field $F$ and a continuous Galois representation $\rbar: G_F\ra\GL_3(\F)$ as in the statement of Theorem \ref{thm:serreweightconj}, except that $\rbar|_{G_{F_{\tld{v}}}}\cong \rhobar$ at all places $v\in \Sigma_p^+$ and such that 
\[
\underset{v\in \Sigma_p^+}{\bigotimes}F_{\un{a}_v}\in W(\rbar)\Longleftrightarrow F_{\un{a}_v}\circ\iota_{\tld{v}}^{-1}\in W^{?}(\rhobar) \quad\text{for all}\quad v\in \Sigma_p^+.
\]
\end{prop}  
\begin{proof} We focus on the case of principal series type as the other cases are similar.  Let $\tau = \omega^{-a} \oplus \omega^{-b} \oplus \omega^{-c}$ with $a- b, b-c >3$.  Let $\rhobar$  be the unique non-split representation of $G_{\Qp}$ of the form 
$$
\rhobar \cong \begin{pmatrix}
\chi_1 \omega^{a+1}&*&0\\
0&\chi_2 \omega^{b+1}&0\\
0&0&\chi_3 \omega^{c+1}
\end{pmatrix}
$$
for fixed unramified characters $\chi_i$. Then $\bf{w}(\rhobar, \tau) = \alpha$. To see this, consider a Kisin module $\fM\in Y^{\mu,\tau}(\F)$ such that $\Mat_{\overline{\beta}}(\phi_{\fM,\omega^c})=A_{\alpha}$. Then by Lemma \ref{Frowoper} which allows row operations, there is another eigenbasis $\overline{\beta}'$ for $\overline{\fM}$ such that $\Mat_{\overline{\beta}'}(\phi_{\fM,\omega^c})$ is given by
\[
A_{\alpha}'=
\begin{pmatrix}
-v \overline{c}_{21}^*\frac{\overline{c}_{12}^*}{\overline{c}'_{22}}&0&0\\
v \overline{c}_{21}^*&v \overline{c}'_{22}&0\\
0&0&v \overline{c}_{33}^*
\end{pmatrix}
\]
It can be seen, for instance via \cite[Lemma 2.2.7]{HLM}, that such Kisin module comes from a Fontaine-Laffaille module with Hodge-Tate weights $(a+1,b+1,c+1)$ and the corresponding Galois representation must be $\rhobar$ (up to unramified twist).

Note that $\rhobar$ admits a Fontaine-Laffaille and hence a potentially diagonalizable lift with Hodge-Tate weights $(a+1, b+1, c+1)$.
We can now apply Lemma A.5 and Proposition A.6 of \cite{EG} to produce an $\rbar:G_{F} \ra \GL_3(\overline{\F})$ such that $\rbar|_{G_{F_{\tld{v}}}}\cong \rhobar$, and such that the setting of \S \ref{sbsec:ModLift}  and Theorem \ref{thm:serreweightconj} holds. Furthermore we can arrange so that $\rbar$ has an automorphic lift $r$ such that $r|_{G_{F_{\tld{v}}}}$ is Fontaine-Laffaille with Hodge-Tate weights $(a+1,b+1,c+1)$ (this is achieved by choosing the potential diagonalizable lift of $r|_{G_{F_{\tld{v}}}}$ featuring in the proof of \cite{EG}, Lemma A.5---which refers back to Theorem 4.3.1 in \cite{BLGGT}---to be Fontaine-Laffaille with Hodge-Tate weights $(a+1,b+1,c+1)$).

Let 
$$
W^{?}(\rhobar) = \left \{ \begin{matrix}  F(a-1, b, c+1), F(p-2+ c, a, b+1), F(a-1, c, b-p+2), \\ F(p-1+b, a, c), F(c+p-1, b, a-p+1), F(a, c, b - p +1) \end{matrix} \right \} \subset W^{?}(\rhobar^{ss}).
$$  
The top row are obvious weights for $\rhobar^{ss}$; the bottom row are shadow weights.  Define $W_{\mathrm{obv}}(\rhobar)\defeq W_{\mathrm{obv}}(\rhobar^{ss})\cap W^?(\rhobar)$.

To prepare for the proof, we observe:
\begin{enumerate}
		\item \label{it:new:pf:0}  There exists a generic tame type $\tau'$ such that $\JH(\sigmabar(\tau'))\cap W^?(\rhobar)=\{F(a-1, b, c+1)\}$ and $\bf{w}(\rhobar,\tau')$ is a length 4 shape;
	\item\label{it:new:pf:1} There exists a generic tame type $\tau'$ such that $\JH(\sigmabar(\tau'))\cap W^?(\rhobar)=\{F(a-1, b, c+1),\,F(a-1,c,b-p+2)\}$ and $\bf{w}(\rhobar,\tau')$ is a length 3 non-shadow shape;
		\item\label{it:new:pf:2} There exists a generic tame type $\tau'$ such that $\JH(\sigmabar(\tau'))\cap W^?(\rhobar)=\{F(c+p-2, a, b+1),\,F(a-1,c,b-p+2)\}$ and $\bf{w}(\rhobar,\tau')$ is a length 3 non-shadow shape;
		\item\label{it:new:pf:3} For each $F\in W_{\mathrm{obv}}(\rhobar)$ such that $F$ is lower alcove with shadow $F'\in W^?(\rhobar)$, there is a generic tame type $\tau'$ such that $\bf{w}(\rhobar,\tau')$ is a length 3 shadow shape;
	\item\label{it:new:pf:4}	For the type $\tau' \defeq \omega_3^{-(a+1) - p (b-1) - p^2 c} \oplus \omega_3^{-(b-1)- p c - p^2 (a+1)} \oplus \omega_3^{-c - p (a+1) - p^2 (b-1)}$, we have that $\JH(\sigmabar(\tau')) \cap W^{?}(\rhobar) = \{F(c+p-2, a, b+1),  F(a, c, b-p+1) \}$ and $\bf{w}(\rhobar,\tau')$ is a length 3 non-shadow shape.

\end{enumerate}
For items (\ref{it:new:pf:0})-(\ref{it:new:pf:3}), we pick the type $\tau'$ which satisfies the analogous properties for $\rhobar^{ss}$ (cf. Proposition \ref{prop:intersec}).
What needs to be checked are the claims about the shapes of $\rhobar$ with respect to $\tau'$.

We will do it for item (\ref{it:new:pf:4}), the other cases will be similar. Note that by Proposition \ref{basechange2}, the orientation for the type $\tau'_v$ is $s_0' = (123), s_1' = (123)^2$ and $s_2' = \mathrm{id}$. We claim that $\bf{w}(\rhobar,\tau')=\alpha\beta\gamma$.
To do this, it suffices to show that there exists a matrix $A\in\Iw(\F)\alpha\beta\gamma\Iw(\F)$ such that the \'etale $\varphi$-module $\cM'$ with $\Mat(\phz^3_{\cM'})$ given by
\[
s'_0A{s'_0}^{-1}s'_1\begin{pmatrix}v^{a+1}&0&0\\
0&v^{b-1}&0\\ 0&0&v^{c}\end{pmatrix}
\phz(A){s'_1}^{-1}s'_2\begin{pmatrix}v^{p(a+1)}&0&0\\
0&v^{p(b-1)}&0\\ 0&0&v^{pc}\end{pmatrix}
\phz^2(A){s'_2}^{-1}\begin{pmatrix}v^{p^2c}&0&0\\
0&v^{p^2(a+1)}&0\\ 0&0&v^{p^2(b-1)}\end{pmatrix}
\]
(cf. Proposition \ref{prop:phif}) is isomorphic to the \'etale $\phz^3$-module $\cM$ associated to $\rhobar|_{G_{\Q_{p^3,\infty}}}$, which has matrix $\Mat(\phz^3_{\cM})$ given by
\[
A_\alpha\begin{pmatrix}v^{a}&0&0\\
0&v^{b}&0\\ 0&0&v^{c}\end{pmatrix}
\phz(A_\alpha)\begin{pmatrix}v^{pa}&0&0\\
0&v^{pb}&0\\ 0&0&v^{pc}\end{pmatrix}
\phz^2(A_\alpha)\begin{pmatrix}v^{p^2a}&0&0\\
0&v^{p^2b}&0\\ 0&0&v^{p^2c}\end{pmatrix}.
\]
Indeed we can set $A\defeq A_\alpha\text{\tiny$\begin{pmatrix}v^{-1}&0&0\\
0&v&0\\ 0&0&1\end{pmatrix}$}{s_1'}^{-1}s_0'$, and this has the required properties.

We now address the proof of the proposition.
We pick a weak minimal patching functor for $\rbar$ (it exists since $\rbar$ and $F/F^+$ satisfy the setup of \S \ref{sbsec:ModLift}). Note that $M_\infty(\otimes_{v\in\Sigma_p^+}F(a-1,b,c+1)\circ \iota_{\tld{v}})\neq 0$ in our situation.
First note that if $\otimes_{v}V_v\in W(\rbar)$ then $V_v \in W^{?}(\rhobar)$ for all $v\in\Sigma_p^+$ by Theorem 1.1 in \cite{MP}.
To finish the proof, it suffices to show that $e(M_\infty(\otimes_{v}V_v))=1$ if $V_v\in W^{?}(\rhobar)$ for all $v$.

We first show this statement in the case where $V_v\in \{F(a-1, b, c+1),\,F(a-1,c,b-p+2)\}$ for all $v$. 
Let $S$ be the set of $v\in\Sigma_p^+$ such that $V_v=F(a-1,c,b-p+2)$.
We induct on the size of $S$.
Consider the type $\otimes_{v\in\Sigma_p^+}\sigma(\tau'_v)$ where $\tau'_v$ is as in item (\ref{it:new:pf:0}).
Then $M_\infty(\otimes_{v\in\Sigma_p^+}\sigmabar(\tau'_v))=M_\infty(\otimes_{v\in\Sigma_p^+}F(a-1,b,c+1)\circ \iota_{\tld{v}})$. By Table \ref{table:withmon}, this gives the base case for the induction.
In general, we pick the type $\otimes_{v\in\Sigma_p^+}\sigma(\tau'_v)$ where $\tau'_v$ is as in item (\ref{it:new:pf:0}) if $v\notin S$ and $\tau'_v$ is as in item (\ref{it:new:pf:1}) otherwise.
By Table \ref{table:withmon}, we have $e(M_\infty(\otimes_{v\in\Sigma_p^+}\sigmabar(\tau'_v)))=2^{|S|}$.
Every $\otimes_{v\in\Sigma_p^+}V'_{v}\in \JH(\otimes_{v\in\Sigma_p^+}\sigmabar(\tau'_v))\cap W(\rbar)$ 
we have $V'_v\in \{F(a-1, b, c+1),\,F(a-1,c,b-p+2)\}$ for all $v$.
All such factors other than $\otimes_{v\in\Sigma_p^+}V_{v}$ will have $V'_v=F(a-1,c,b-p+2)$ for strictly less than $|S|$ embeddings $v$. From the inductive hypothesis we deduce that $e(M_{\infty}(\otimes_{v} V_v)) = 1$.
A similar argument using the type $\tau'$ as in item (\ref{it:new:pf:2}) deals with the case where  $V_v\in W_{\mathrm{obv}}(\rhobar)$.

Finally we deal with the shadow weights in $W^?(\rhobar)$ exactly as in Theorem \ref{thm:serreweightconj}, using the type $\tau'$ as in item (\ref{it:new:pf:4}) to deal with the weight $F(a, c, b-p+1)$ and the types as in item (\ref{it:new:pf:3}) for the remaining shadows.
\end{proof} 

We conclude this section with a counterexample to \cite[Conjecture 4.3.2]{gee-annalen} (the \emph{crystalline conjecture}).

\begin{prop} \label{geecounter}
Let $F=F(a-1,b,c+1)$ be a lower alcove weight such that the triple $(a,b,c)$ is generic and write $F_\textrm{Sh}\defeq F(c+p-1,b,a-p+1)$.
There exists a CM field $F$ where $p$ splits completely and a continuous automorphic Galois representation $\rbar:G_F\ra \GL_3(\F)$ such that $\otimes_{v\in\Sigma_p^+}F\circ\iota_{\tld{v}}\in W(\rbar)$ but $(\otimes_{v\in\Sigma_p^+\setminus S}F_{v}\circ\iota_{\tld{v}})\otimes (\otimes_{v\in S}F_\textrm{Sh}\circ\iota_{\tld{v}})\notin W(\rbar)$ for any $\emptyset \neq S\subseteq \Sigma_p^+$ and $F_v$ any weight at $v\in \Sigma_p^+\setminus S$.
Moreover, for any $v\in\Sigma_p^+$ the representation $\rbar|_{G_{F_{\tld{v}}}}$ has a crystalline lift of weight $(c+p+1,b+1,a-p+1)$.
\end{prop}
\begin{proof}
We consider the generic type $\tau\defeq \omega^{-a}\oplus\omega^{-b}\oplus\omega^{-c}$. It satisfies $F,\, F_\textrm{Sh}\in\JH(\sigmabar(\tau))$.

Then there exists $\overline{\fM}\in Y^{\mu,\tau}(\F)$ over $\Fp\otimes_{\Fp}\F[\![u]\!]$ endowed with an eigenbasis $\overline{\beta}$ such that $A\defeq \Mat_{\beta}(\phi_{\fM,\omega^c})$ is given by the matrix in row $\alpha\beta\alpha$ in Table \ref{table shapes mod p}, satisfying $(a-b) \overline{c}_{23} \overline{c}_{32} -(a-c) \overline{c}_{22}^* \overline{c}'_{33} \neq 0$, $\overline{c}_{23} \overline{c}_{32} -\overline{c}_{22}^* \overline{c}'_{33} \neq 0$, $ \overline{c}_{23} \overline{c}_{32}\neq 0$, and $\overline{c}'_{33}\neq 0$. 
Moreover, $T_{dd}(\overline{\fM})\cong \rhobar|_{G_{\Q_{p,\infty}}}$ for a continuous Galois representation $\rhobar:G_{\Qp}\ra\GL_3(\F)$. 
This is checked as in the proof of Proposition \ref{Prop-non-ss} by showing that the \'etale $\phz$-module associated to $\overline{\fM}$ comes from a Fontaine-Laffaille module $M$ over $\Fp\otimes_{\Fp}\F$ with Hodge-Tate weights $(a+1,b+1,c+1)$. 
In fact it can be shown, using the Fontaine-Laffaille module $M$, that the associated $\rhobar$ is maximally non-split with  $\omega^{a+1}\subseteq \rhobar$ and $\rhobar\twoheadrightarrow\omega^{c+1}$, $\rhobar^{ss}\cong \omega^{a+1}\oplus\omega^{b+1}\oplus\omega^{c+1}$ and Fontaine-Laffaille parameter $\mathrm{FL}(\rhobar)\neq 0,\, \infty$: cf. ~\cite[Definition 2.1.10, Corollary 2.1.8 and Lemma 2.2.7]{HLM}.
Note that $\bf{w}(\rhobar,\tau)=\alpha\beta\alpha$ by construction.

By the same argument as in the proof of Proposition \ref{Prop-non-ss}, we find a CM field $F/F^+$ which is unramified at all finite places and an automorphic Galois representation $\rbar:G_F\ra\GL_3(\F)$ satisfying the Taylor-Wiles conditions such that $\rbar|_{G_{F_w}}\cong \rhobar$ for all $w\in \Sigma_p$ (note that $\rhobar$ admits a crystalline lift of weight $(a+1,b+1,c+1)$). 

Let $M_\infty$ be a weak minimal patching functor for $\rbar$ (which exists by Proposition \ref{prop:WMPF}). Let $\sigma(\tau)^\circ$ be the unique lattice in $\sigma(\tau)$ such that $\overline{\sigma}(\tau)^{\circ}$ has irreducible cosocle isomorphic to $\otimes_{v\in\Sigma_p^+}F\circ\iota_{\tld{v}}$. Then the kernel $\overline{N}$ of the surjection $\overline{\sigma}(\tau)^{\circ}\twoheadrightarrow \otimes_{v\in\Sigma_p^+}F\circ\iota_{\tld{v}}$ contains all the weights of the form $(\otimes_{v\in\Sigma_p^+\setminus S}F\circ\iota_{\tld{v}})\otimes (\otimes_{v\in S}F_\textrm{Sh}\circ\iota_{\tld{v}})$ where $\emptyset \neq S\subseteq \Sigma_p^+$.

By \cite[Theorem D]{HLM}, if a weight $\otimes_{v\in \Sigma_p^+}F_{v}\circ\iota_{\tld{v}}$ is modular, then necessarily $F_{v}\in\{F,\, F_\textrm{Sh}\}$ for all $v$. Thus for the first claim of the Proposition, it suffices to show $M_\infty(\overline{N})=0$.
Indeed on one hand, we have $e\left(M_\infty(\otimes_{v\in\Sigma_p^+}F\circ\iota_{\tld{v}})\right)\geq 1$ (since $\rbar$ is modular of the lower alcove weight $\otimes_{v\in\Sigma_p^+}F\circ\iota_{\tld{v}})$.
On the other hand, we have $e\left(M_\infty(\overline{\sigma}(\tau)^{\circ})\right)=e\left(\overline{R}_\infty(\tau)\right)=1$ since $R_\infty(\tau)$ is formally smooth over $\cO$ (cf. row 6 in Table \ref{table:withmon}).

By exactness of $M_\infty$, we conclude that $e(M_\infty(\overline{N}))$, and hence $M_\infty(\overline{N})$, is $0$ as required.

As for the last part of the statement, let $W\defeq W(c+p-1,b,a-p+1)$ denote the $\GL_3(\Fp)$-representation over $\F$ obtained by taking the $\GL_3(\Fp)$-rational points of the Weyl module of highest weight $(c+p-1,b,a-p+1)$, and extending the coefficient field to $\F$. Then $\JH(W)=\{F,\, F_{\mathrm{Sh}}\}$.

For any $\emptyset\neq S\subseteq \Sigma_p^+$, we have $M_\infty\left((\otimes_{v\in\Sigma_p^+\setminus S}F\circ\iota_{\tld{v}})\otimes (\otimes_{v\in S}W\circ\iota_{\tld{v}})\right)\neq 0$ and hence by classical local-global compatibility $\rbar|_{G_{F_{\tld{v}}}}$ admits a crystalline lift of weight $(c+p+1,b+1,a-p+1)$ for $v\in S$.
\end{proof}

\section{The $\alpha$ and $\mathrm{id}$ shapes}
\label{sec:badcases}

The aim of this section is complete the proof of Theorems \ref{thm:modularity} and \ref{thm:conngenfib} by studying the deformation rings in the most complicated cases when the shape has length 0 or 1. To recall the assumptions, $\tau$ will be a generic tame inertial type and $\rhobar:G_{\Qp} \ra \GL_3(\F)$ satisfies $\bf{w}(\rhobar, \tau) \in \{ \alpha, \mathrm{id}\}$.  
There are three cases to consider: the identity shape $\bf{w}(\rhobar, \tau)=\mathrm{id}$ (\S ~\ref{sbs:Id}), the case where $\bf{w}(\rhobar, \tau)=\alpha$ and $\rhobar$ is semisimple (\S ~\ref{sbsb:semi}) and the case where $\bf{w}(\rhobar, \tau)=\alpha$ and $\rhobar$ is non-semisimple (\S ~\ref{sbsb:nsemi})
In the three cases, the cardinality of the intersection $\JH(\sigmabar(\tau)) \cap W^{?}(\rhobar)$ is 6, 6 and 5 respectively (see Proposition \ref{Prop-non-ss} when $\rhobar$ is non-semisimple). 

The difficulty of these cases is that the monodromy equations in characteristic 0 become too complicated to manipulate, and in particular it is hard to see exactly the effect of $p$-saturation. Nevertheless, one can guess an explicit candidate for the mod $p$ fiber of the deformation ring, because the monodromy equations (and the relations implied by $p$-saturation) become much simpler mod $p$. A priori, this candidate could be strictly larger than the mod $p$ fiber of the deformation ring, but we then invoke global arguments in the form of the Serre weight conjectures (Theorem \ref{thm:serreweightconj} and Proposition \ref{Prop-non-ss}) to show that this does not happen.

The genericity condition is used to guarantee that the error terms in the monodromy equations are divisible by a large enough power of $p$ so that they can be ignored in our manipulations after reducing modulo $p$.


With respect to the notations of \S~\ref{KM with dd}, we have $f=1$ and set $a\defeq a_{s_0(1),0}$, $b\defeq a_{s_0(2),0}$ and $c\defeq a_{s_0(3),0}$.

\subsection{The identity shape} 
\label{sbs:Id}
We now assume that $\tau$ is $n$-generic with $n\geq 4$ and that $\bf{w}(\rhobar,\tau)=\Id$.  Let
$$
A = \begin{pmatrix}
c_{11}+c_{11}^*(v+p)&c_{12}&c_{13}\\
vc_{21}&c_{22}+ c_{22}^*(v+p) &c_{23}\\
vc_{31}&vc_{32}&c_{33}+c_{33}^*(v+p)
\end{pmatrix}.
$$
In this section, we work over the ring $R^{\mathrm{aux}}$ which is defined to be the $p$-saturation of the quotient of $\cO[c_{ij},\, (c^{*}_{ii})^{\pm 1},\, 1\leq i,j\leq 3]$ by the following relations:
all 2 by 2 minors of 
$$
A|_{v = -p} = \begin{pmatrix}
c_{11} &c_{12}&c_{13}\\
-pc_{21}&c_{22} &c_{23}\\
-pc_{31}&-pc_{32}& c_{33}
\end{pmatrix}
$$
vanish and the determinant condition 
$$
c_{11}c^*_{22}c^*_{33}+c_{22}c^*_{33}c^*_{11}+c_{33}c^*_{11}c^*_{22}-c_{11}^*c_{23}c_{32}-c_{22}^*c_{13}c_{31}-c_{33}^*c_{12}c_{21}+c_{21}c_{13}c_{32}=0.
$$
We have
\begin{equation*}
\resizebox{1.0 \textwidth}{!}
{$
-A^{\dagger}|_{v = -p} = \begin{pmatrix}
-pec^*_{11} &(a-b)c_{12}& (a- c)c_{13}\\
-p (e + b - a)c_{21}& -pec^*_{22} &(b-c)c_{23}\\
-p (e + c - a)c_{31}&-p(e + c - b)c_{32}& -pec^*_{33}
\end{pmatrix}, \frac{\det(A)}{P(v)} A^{-1} \mid_{v = -p} = \begin{pmatrix}
c_{22}^* c_{33} + c_{33}^* c_{22} - c_{23} c_{32}  & c_{13} c_{32} - c^*_{33} c_{12} & -c_{22}^* c_{13}\\
p c_{33}^* c_{21} & c_{11}^* c_{33} + c_{33}^* c_{11} - c_{13} c_{31}&c_{13} c_{21} - c_{11}^*c_{23}\\
-p c_{21} c_{32} + p c^*_{22} c_{31}  &  pc_{11}^* c_{32} & c_{11}^* c_{22} + c_{22}^* c_{11} - c_{12} c_{21}
\end{pmatrix}$.
}
\end{equation*}
In the following lemmas, let us abbreviate $U\defeq A|_{v = -p}$ and $V\defeq \frac{\det(A)}{P(v)} A^{-1} \mid_{v = -p}$.
\begin{lemma}
\label{lem:invertible}
The ring $R^{\mathrm{aux}}$ is a domain and $c_{ij}\neq 0$ in $R^{\mathrm{aux}}$ for all $1\leq i,j\leq 3$.
\end{lemma}
\begin{proof}
It suffices to prove the statements for $R^{\mathrm{aux}}[1/p]$. Let $\mathrm{Det}$ be the determinantal variety over $E$ on the entries of $U$ obtained by imposing that the 2 by 2 minors of $U$ vanish.
Then the map $\Spec(R^{\mathrm{aux}}[1/p])\ra \mathrm{Det}$ obtained by forgetting $c_{ii}^*$ is smooth and surjective.
This gives both statements of the lemma.
\end{proof}

\begin{lemma}
\label{lem:fix:ref}
Keep the setting and notation above. For any matrix $Z=(Z_{ij})_{1\leq i,j\leq 3}$ of formal variables $Z_{ij}$, we have the following equality
\[
\left(A|_{v = -p} \right)Z \left(\frac{\det(A)}{P(v)} A^{-1} \mid_{v = -p}\right)=(U_{ij}X_j)_{1\leq i,j\leq 3}
\]
with $pX_j\in R^{\mathrm{aux}}[Z_{ij},\, 1\leq i,j\leq 3]$.
\end{lemma}
\begin{proof}
It suffices to show that for any $1\leq i,k,l,j\leq 3$, one has
\begin{equation}
U_{ik} V_{lj}=U_{ij}X_{jkl}
\end{equation}
for some $X_{jkl}\in\frac{1}{p}\cO[c_{ij},\, c_{ii}^*,\, 1\leq i,j,\leq 3]$ independent of $i$.

The fact that the $(U_{ik} V_{lj})\in\frac{c_{ij}}{p}\cO[c_{ij},\, c_{ii}^*,\, 1\leq i,j,\leq 3]$ follows from the fact that $V_{lj}$ is a $R^{\mathrm{aux}}$-linear combination of $c_{mj}$'s (for example, $c^*_{22}c_{33} + c^*_{33}c_{22} - c_{23} c_{32}=-(c^*_{22}c^*_{33}c_{11}+c^*_{22}c_{13}c_{31}+c^*_{33}c_{12}c_{21}-c_{21}c_{13}c_{32})/c^*_{11}$), and the minor conditions allows us to convert from $c_{ij}c_{kl}$ to $c_{il}c_{kj}$, at possibly a cost of a $p$ in the denominator.

The fact that $X_{jkl}$ is independent of $i$ can also be checked on the locus (using Lemma \ref{lem:invertible}) where all the $U_{ij}$'s are invertible, where it follows immediately from the 2 by 2 minor condition ($\frac{U_{ik}}{U_{ij}}=\frac{U_{i'k}}{U_{i'j}}$).
\end{proof}

We now study the deformation ring.
We observe that the finite height ring $R_{\overline{\fM}}^{\tau, \overline{\beta}, \Box}$ is an $R^{\mathrm{aux}}$-algebra in an obvious way. 

Recall the $\cO$-algebra $R^{\tau, \overline{\beta}, \Box}_{\overline{\fM}, \rhobar}$ from Definition \ref{defprob} (2). 
By (\ref{defdiagram}) and Theorem \ref{gaugeunique}, it is isomorphic to a power series ring in 3 variables over the potentially crystalline ring $R^{\mu, \tau}_{\rhobar}$, and hence has relative dimension 14 over $\cO$.
Recall also (cf. (\ref{defdiagram})) that we have a surjection  $\pi: R_{\overline{\fM}}^{\tau, \overline{\beta}, \Box}\twoheadrightarrow R^{\tau, \overline{\beta}, \Box}_{\overline{\fM}, \rhobar}$, where we quotient out by the monodromy equations (and take the reduced and $p$-flat quotient of the result). 

The leading term for the monodromy is given by $(A^{\dagger} P(v)^2 A^{-1})_{v = -p}$ (cf. Definition \ref{leadingterm}). Define
\begin{align*}
\mathrm{Mon}_1=&(e-a+c)c^*_{22}c_{33}+(e-a+b)c_{22}c^*_{33}-(e-a+c)c_{23}c_{32}+pec^*_{22}c^*_{33}
\\
\mathrm{Mon}_2=&(a-b)c^*_{33}c_{11}+(e-b+c)c_{33}c^*_{11}-(a-b)c_{13}c_{31}+pec^*_{33}c^*_{11}
\\
\mathrm{Mon}_3=&(b-c)c^*_{11}c_{22}+(a-c)c_{11}c^*_{22}-(b-c)c_{12}c_{21}+pec^*_{11}c^*_{22}.
\end{align*}


A direct computation shows that in $R^{\mathrm{aux}}$ we have the equality
\[
(A^{\dagger} P(v)^2 A^{-1})_{v = -p}=A|_{v = -p}\cdot \begin{pmatrix}\mathrm{Mon}_1&0&0\\
0&\mathrm{Mon}_2&0\\0&0&\mathrm{Mon}_3\end{pmatrix}
\]
Thus Theorem \ref{moncond} and Lemma \ref{lem:fix:ref} show that in $R^{\tau, \overline{\beta}, \Box}_{\overline{\fM}, \rhobar}$ we have the equation 
\[
A|_{v = -p}\cdot \begin{pmatrix}\mathrm{Mon}_1+O_1(p^{n-2})&0&0\\
0&\mathrm{Mon}_2+O_2(p^{n-2})&0\\0&0&\mathrm{Mon}_3+O_3(p^{n-2})\end{pmatrix}=0
 \]
where $O_i(p^{n-2})$  stands for an error term which is divisible by $p^{n-2}$.

The following proposition refines Theorem \ref{thm:factors} in the present situation: 
\begin{prop} \label{refinedfactor}
The surjection $\pi$ factors through the quotient of $R_{\overline{\fM}}^{\tau, \overline{\beta}, \Box}$ by the relations
\begin{equation}
\label{rel:id}
\mathrm{Mon}_i+O_i(p^2)=0, \forall i\in \{1,2,3\}.
\end{equation}
\end{prop}
\begin{proof} We already know that $(\mathrm{Mon}_j+O_j(p^2))c_{ij}=0$. The refinement will come from the fact that $R^{\mu,\tau}_{\overline{\rho}}$ classifies representations with Hodge-Tate weights \textit{exactly} $(2,1,0)$ instead of just being in $[0,2]$.

Since $R^{\mu,\tau}_{\overline{\rho}}$ is flat and  $R^{\mu,\tau}_{\overline{\rho}}[\frac{1}{p}]$ is regular, it suffices to show that the equations (\ref{rel:id}) hold in the $p$-flat closure of each connected component of $R^{\tau, \overline{\beta}, \Box}_{\overline{\fM}, \rhobar}[\frac{1}{p}]$. Let $R$ denote the $p$-flat closure of a connected component of $R^{\tau, \overline{\beta}, \Box}_{\overline{\fM}, \rhobar}[\frac{1}{p}]$ and assume that we have $\mathrm{Mon}_1+O_1(p^2)\neq 0$ in $R$. 
Since $R$ is a domain, we conclude that $c_{i1}=0$ in $R$ for all $i$. 
On the other hand, $c_{ij}$ cannot be $0$ in $R$ for all $i,j$ since the Hodge--Tate weights are exactly $(2,1,0)$, so at least one of the equations $\mathrm{Mon}_j+O_j(p^2)=0$ for $j=2,3$ holds in $R$. 
Assume without loss of generality that $\mathrm{Mon}_2+O_2(p^2)=0$. This implies that $c_{33}\in pR^{\times}$ and in particular $c_{33}\neq 0$. Thus $\mathrm{Mon}_3+O_3(p^2)=0$ and hence $c_{22}\in p R^\times$.
Now the relation $c_{12}c_{33}=-pc_{13}c_{32}$ gives $c_{12}=\frac{-pc_{13}c_{32}}{c_{33}}$ in $R$ and $c_{23}c_{32}=-\frac{1}{p}c_{22}c_{33}=0$ in $R/\varpi$. Thus we see $R/\varpi$ is a quotient of a power series ring in $8$ variables over $\F[\![c_{13},\, c_{32},\, c_{23}, c_{ii}^*-[\overline{c}_{ii}^*],\, i=1,\,2,\,3]\!]/(c_{23}c_{32})$. This shows that $R/\varpi$ has dimension at most 13, a contradiction.
\end{proof}
\begin{cor} 
\label{cor:explicit}
Let $\widetilde{R}$ be the quotient of $\F[\![c_{ij},\ 1\leq i,j,\leq 3]\!]$  by the relations:
\begin{align*}
&c_{ii}c_{jj}=0,&&\text{for $i\neq j$};\\
&c_{11}c_{23}=0;&&
c_{31}c_{22}=0;
&& c_{12}c_{23}=c_{22}c_{13}; 
&& c_{11}c_{32}=c_{12}c_{31}; 
&& c_{21}c_{33}=c_{31}c_{23};
\end{align*}
\begin{align*}
& (e-a+c)c_{33}+(e-a+b)c_{22}-(e-a+c)c_{23}c_{32}=0;
\\ 
&(b-c)c_{22}+(a-c)c_{11}-(b-c)c_{12}c_{21}=0;
\\
& (a-b)c_{11}+(e-b+c)c_{33}-(a-b)c_{13}c_{31}=0;
\\
&c_{11}+c_{22}+c_{33}-c_{12}c_{21}-c_{13}c_{31}-c_{23}c_{32}+c_{21}c_{13}c_{32}=0.
\end{align*}
Then the ring $R^{\tau, \overline{\beta}, \Box}_{\overline{\fM}, \rhobar}/\varpi$ is a power series ring over a quotient of $\widetilde{R}$.
\end{cor}
\begin{proof} This is a direct consequence of Proposition \ref{refinedfactor} and the observation that replacing $c_{ij}$ by $\frac{c_{ij}}{c^*_{jj}}$ eliminates the $c^*_{ii}$ from all the equations.
\end{proof}
The following Proposition gives basic structural information about $\widetilde{R}$:
\begin{prop} \label{id explicit}
The ring $\widetilde{R}$ is a reduced 3-dimensional Cohen-Macaulay ring. It has 6 minimal primes and each irreducible component is formally smooth over $\F$. Thus $e(\widetilde{R})=6$.
\end{prop}
\begin{proof} Observe that the relations defining $\tilde{R}$ are actually polynomials instead of genuine power series, so $\tilde{R}$ can be viewed as a completion of a quotient of a polynomial ring by the ideal $I$ generated by the relations above. We use some standard terminology from the theory of Gr\"obner bases \cite{Eisenbud:CA}. We pick the monomial order on $\F[c_{ij}]\defeq \F[c_{ij},\ 1\leq i,j\leq 3]$ given by $c_{11}>c_{12}>c_{13}>c_{21}>c_{22}>c_{23}>c_{31}>c_{32}>c_{33}$ and write $I$ for the ideal generated by the relations of Proposition \ref{cor:explicit}.
An easy but tedious calculation (using Buchberger's algorithm, for example) shows that the ideal $I$ has the following Gr\"obner basis with our choice of monomial order:
\begin{align*}
&c_{11}-c_{13}c_{31}+\frac{e-b+c}{a-b}c_{33},
\\
&c_{12}c_{21}-\frac{a-c}{b-c}c_{13}c_{31}-\frac{e-a+c}{e-a+b}c_{23}c_{32}+\Big(\frac{(a-c)(e-b+c)}{(b-c)(a-b)}-\frac{e-a+c}{e-a+b}\Big)c_{33},
\\
&c_{12}c_{23}-\frac{e-a+c}{e-a+b}c_{13}c_{33},
\\
&c_{12}c_{33},\  c_{13}c_{21}c_{32}-\frac{a-c}{b-c}c_{13}c_{31}-c_{23}c_{32}+\frac{e}{b-c}c_{33},
\\
&c_{12}c_{23}c_{31}-\frac{e-b+c}{a-b}c_{23}c_{33},\  c_{13}c_{31}c_{33}-\frac{e-b+c}{a-b}c^2_{33},
\\
&c_{21}c_{33}-c_{23}c_{31},\  c_{22}-\frac{e-a+c}{e-a+b}c_{23}c_{32}+\frac{e-a+c}{e-a+b}c_{33},
\\
&c_{23}c_{31}c_{32}-c_{31}c_{33},\  \textrm{ and } \ c_{23}c_{32}c_{33}-c^2_{33}.
\end{align*}
In each of the above polynomials, the leading monomial is exactly the left-most term. Thus we see that the initial ideal $\mathrm{in}(I)$ of $I$ is generated by square-free monomials. This implies that $I$ is radical: Suppose $f^k\in I$, then $\mathrm{in}(f)^k\in \mathrm{in}(I)$, so $\mathrm{in}(f)\in \mathrm{in}(I)$. But then we can divide $f$ by elements in $I$ and get some $f'<f$ with $f'^k\in I$. Continuing this way, we see that $f\in I$.

Furthermore, as in \cite[\S 15.8]{Eisenbud:CA}, we can realize $\F[c_{ij}]/I$ as the fiber $\cF_t$ (for any $t \neq 0$) of a flat family $\cF$ over $\F[t]$, such that the fiber $\cF_0$ is $\F[c_{ij}]/\mathrm{in}(I)$. Since this quotient is Cohen-Macaulay by an explicit check, and the Cohen-Macaulay locus is open in $\cF$, we conclude that $\cF$ is Cohen-Macaulay at some closed point of the form $c_{ij}=0$ for all $i$ and $j$, and $t=t_0\neq 0$. But then $(t-t_0)$ is a regular element in the localization of $\cF$ at this point, and hence the localization of $\F[c_{ij}]/I$ at the closed point $c_{ij}=0$ for all $i$ and $j$ is Cohen-Macaulay.

The computation of the irreducible components is left as an easy exercise to the reader.
\end{proof} 

For the following Proposition, we need to assume that $\tau$ is \emph{strongly generic}.

\begin{prop}  
\label{prop:explicit2}
Assume that $\tau$ is strongly generic.
Then $R^{\tau, \overline{\beta}, \Box}_{\overline{\fM}, \rhobar}/\varpi$ is isomorphic to a power series ring over $\widetilde{R}$.
\end{prop}
\begin{proof}
We globalize $\rhobar$ to a $\overline{r}: G_F \to \GL_3(\F)$ such that the following conditions hold:
\begin{itemize}
\item The assumptions of Theorem \ref{thm:serreweightconj} are satisfied;
\item $\overline{r}$ is unramified away from $p$;
\item $p$ splits completely in $F$. Make a choice $\tilde{v}$ above each place $v|p$ in $F^+$;
\item For each $\tilde{v}|p$, there is an isomorphism $F_{\tilde{v}}\cong \mathbb{Q}_p$ and $\overline{r}|_{G_{F_{\tilde{v}}}}\cong \rhobar$
\end{itemize}
(cf. the proof of Proposition \ref{Prop-non-ss}).
With this global setting, we can choose a weak minimal patching functor $M_\infty$.
We now choose a tame type $\otimes_{v\in\Sigma_p^+}\tau_v$ such that $\overline{r}|_{G_{F_{\tilde{v}}}}$ has a shape of length 4 with respect to $\tau_v$ for all but one place $v_0|p$, while  $\overline{r}|_{G_{F_{\tilde{v}_0}}}$ has $\mathrm{id}$ shape with respect to $\tau_{v_0}\defeq \tau$. 
Since $\tau$ is strongly generic, we deduce from Lemma \ref{lem:gen-2} below that $\rhobar$ is $8$-generic. Moreover $\JH(\sigmabar(\tau))$ consists of reachable weights.
With this choice, the intersection
\[
\JH(\otimes_{v\in\Sigma_p^+}\sigmabar(\tau_v))\cap W^?(\otimes_{v\in\Sigma_p^+}\rbar|_{G_{F_{\tilde{v}}}})
\] 
(with obvious notation) consists of exactly 6 weights. 
By Theorem \ref{thm:serreweightconj}, we conclude that if $W$ is any Serre weight for $\cG$ in this intersection we have $M_\infty(W)\neq 0$. 
Thus we have 
$$
e(M_\infty(\overline{\sigma}(\otimes_{v\in\Sigma_p^+}\tau_v))) \geq 6.
$$
On the other hand, by our choice of $\tau_v$ and the knowledge of Galois deformation rings for length 4 shapes (Corollary \ref{dring} and Table \ref{table:withmon}), $\overline{R}_\infty(\otimes_{v\in\Sigma_p^+}\tau_v)$ is isomorphic to a power series ring over $R^{\mu,\tau_{v_0}}_{\rhobar}$. It follows that $e( R^{\mu,\tau}_{\rhobar}/\varpi)\geq 6$.
The Proposition now follows from Lemma \ref{lem:ring} below, the fact that $R^{\tau, \overline{\beta}, \Box}_{\overline{\fM}, \rhobar}/\varpi$ receives a surjection from a power series ring over $\widetilde{R}$ (Corollary \ref{cor:explicit}) and has the same Hilbert-Samuel multiplicity as $R^{\mu,\tau}_{\rhobar}/\varpi$.
\end{proof}

\begin{lemma}
\label{lem:gen-2}
Let $n\geq 2$.
Let $\tau_0$ be an $n$-generic inertial type for $I_{\Qp}$ and $\rhobar_0:G_{\Qp} \ra \GL_3(\F)$ be a Galois representation such that there exists $\cM \in \Phi\text{-}\Mod^{\text{\'et}}_{dd}(\F)$ with $Y^{\mu,\tau_0}_{\cM}(\F)\neq \emptyset$ $($and thus a single point by Theorem \ref{Kisinvariety}$)$ and $\bV_{dd}^*(\cM) \cong \rhobar_0|_{G_{\Q_{p,\infty}}}$.
Then $\rhobar_0$ is an $(n-2)$-generic continuous Galois representation.
\end{lemma}
\begin{proof}
This follows from a direct computation using Proposition \ref{prop:phif}.
\end{proof}

\begin{lemma}
\label{lem:ring} 
Suppose $R$, $S$ are complete Noetherian local rings over $\F$, with $R\twoheadrightarrow S$. Assume that $R$ is reduced, that $R$ and $S$ are equidimensional with $\dim R=\dim S$ and that $e(R)=e(S)$.
Then $R\cong S$.
\end{lemma}
\begin{proof}
Let $I$ be the kernel of $R\twoheadrightarrow  S$.
Because $e(R)=e(S)$ and $\dim R=\dim S$, the support of $I$ as an $R$-module does not contain any minimal primes, hence $I_{\mathfrak{p}}= 0$ for all minimal primes $\mathfrak{p}$ of $R$. But this implies that $I$ is inside the intersection of all the minimal primes of $R$, hence $I=0$ because $R$ is reduced.
\end{proof}

\begin{cor} Assume that $\tau$ is strongly generic. The ring $R^{\mu,\tau}_{\rhobar}$ is normal, Cohen-Macaulay, and $R^{\mu,\tau}_{\rhobar}[\frac{1}{p}]$ is a domain.
\end{cor}
\begin{proof} 
By Propositions \ref{id explicit}, \ref{prop:explicit2} above, $R^{\tau,\overline{\beta},\Box}_{\overline{\fM},\rhobar}/\varpi$ is reduced and Cohen-Macaulay, hence $R^{\mu,\tau}_{\rhobar}/\varpi$ inherits those properties by formal smoothness (cf. (\ref{defdiagram})).
This implies Cohen-Macaulayness.

  Since  $R^{\mu,\tau}_{\rhobar}[\frac{1}{p}]$ is regular, to show it is a domain it suffices to show it has no non-trivial idempotent. Suppose $e$ is a non-trivial idempotent. Then there is a maximal $k\in \Z$ such that $\varpi^{-k}e\in R^{\mu,\tau}_{\rhobar}$. By maximality and $e=e^2\in \varpi^{2k}R^{\mu,\tau}_{\rhobar}$, we have $k\geq 2k$. On the other hand $(\varpi^{-k}e)^2=\varpi^{-k}(\varpi^{-k}e)$ and $\varpi^{-k}e\neq 0$ mod $\varpi$, hence we must have $k=0$ since $R^{\mu,\tau}_{\rhobar}/\varpi$ is reduced. But then $e$ is an idempotent of the local ring $R^{\mu,\tau}_{\rhobar}$, hence it must be a trivial idempotent.

Finally, since $R^{\mu,\tau}_{\rhobar}[\frac{1}{p}]$ is regular and $R^{\mu,\tau}_{\rhobar}/\varpi$ is reduced, $R^{\mu,\tau}_{\rhobar}$ satisfies conditions $R_1$ and $S_2$, hence is normal. 
\end{proof}

\begin{rmk}
The reason for which we need $\tau$ to be $10$-generic in Proposition \ref{prop:explicit2} is due to Lemma \ref{lem:gen-2} and the $8$-genericity assumption on $\rhobar$ appearing in Theorem \ref{thm:serreweightconj}.
By Remark \ref{rem:WE} an improvement on weight elimination for a niveau three $\rhobar$ could relax the genericity assumption on $\tau$. (An inspection of the proof of Theorem \ref{thm:serreweightconj} and Table \ref{table: types} shows however that any improvement of Theorem \ref{thm:serreweightconj} based on stronger weight elimination results will require $\rhobar$ to be at least $7$-generic: in order to obtain a further relaxation on the genericity hypotheses on $\rhobar$, hence on $\tau$, one should be able to perform the explicit computations in \S~\ref{Explicit deformation rings} with types which are $m$-generic with $m\leq 4$).

We also observe that in the specific situation of Proposition \ref{prop:explicit2} we could have avoided Lemma \ref{lem:gen-2}, noting that $\bf{w}(\rhobar,\tau)=\Id$ implies that $\rhobar$ and $\tau$ have the same degree of genericity. However, this is no longer true when $\bf{w}(\rhobar,\tau)=\gamma$, in which case we need to invoke Lemma \ref{lem:gen-2} to have the analogous statement of Proposition \ref{prop:explicit2} for shape $\gamma$.
\end{rmk}

\subsection{The $\alpha$ shape} 
We assume that $\tau$ is $n$-generic with $n\geq 4$ and that $\bf{w}(\rhobar,\tau)=\alpha$.
The universal family of shape $\alpha$ is given by 
$$
A = \begin{pmatrix}
c_{11}&c_{12}+ (v+p)c_{12}^*&c_{13}\\
c_{21}^*v&c_{22}+ (v+p)c_{22}' &c_{23}\\
c_{31}v&c_{32}v&\left(c_{33}+(v+p) c_{33}^*\right)
\end{pmatrix}
$$
subject to the condition that all 2 by 2 minors of 
$$
A|_{v = -p} = \begin{pmatrix}
c_{11} &c_{12}&c_{13}\\
-pc^*_{21}&c_{22} &c_{23}\\
-pc_{31}&-pc_{32}& c_{33}
\end{pmatrix}
$$
vanish and the determinant condition 
$$
c_{11}c_{22}'c_{33}^*+c_{13}c_{21}^*c_{32}-c_{13}c_{22}'c_{31}-c_{12}c_{21}^*c_{33}^*+pc_{21}^*c_{12}^*c_{33}^*=0.
$$
We have
\begin{equation*}
\resizebox{1.0 \textwidth}{!}
{$
-A^{\dagger}_{v = -p} = \begin{pmatrix}
0 & (a-b) c_{12} - e p c^*_{12} & (a-c) c_{13} \\
-p(e - a + b) c_{21}^* & -e p c'_{22}& (b-c) c_{23} \\
-p(e - a + c) c_{31} & -p(e -b + c)c_{32} & - p e c^*_{33} 
\end{pmatrix}, \frac{\det(A)}{P(v)} A^{-1} \mid_{v = -p} = \begin{pmatrix}
-c_{23} c_{32} + c_{22} c^*_{33} + c_{33} c'_{22} & c_{13} c_{32} - c_{12} c^*_{33} - c^*_{12} c_{33}& c_{12}^* c_{23} - c_{13} c'_{22} \\
p c^*_{33} c^*_{21} & - c_{13} c_{31} + c_{11} c^*_{33} &  c^*_{21} c_{13} \\
- p c^*_{21} c_{32} + p c'_{22} c_{31} & p c_{31} c^*_{12} &  - c_{21}^* c_{12} + c_{11} c'_{22}  + p c^*_{12} c^*_{21} 
\end{pmatrix}
$}
\end{equation*}
Define $\widetilde{c}_{32}\defeq c_{32}-\frac{c'_{22}c_{31}}{c^*_{21}}$.
By looking at the $(1,1),(2,1),(3,3)$ entries of $(A^{\dagger} P(v)^2 A^{-1})_{v = -p}$ (the leading term for monodromy, cf. Definition \ref{leadingterm}) we get the following monodromy equations: 
\begin{align*}
&(a-b) c_{12} c^*_{33}  - (a-c)c_{13}\widetilde{c}_{32} = pe c^*_{12} c^*_{33} +O(p^2) 
\\
&(e-a + c)  c_{23}\widetilde{c}_{32}  - (e - a +b) c_{22} c^*_{33}  = ep c'_{22} c^*_{33}+O(p^2)  
\\
&\text{\Small$(e-a+c)c_{31}c_{23}c^*_{12}-(e-a+c)c_{31}c_{13}c'_{22}+(e-b+c)c_{32}c_{13}c^*_{21}-ec_{12}c^*_{33}c^*_{21}+ec_{11}c'_{22}c^*_{33}+pec^*_{12}c^*_{21}c^*_{33}=O(p^2)$}
\end{align*}

\subsubsection{The semisimple case.}
\label{sbsb:semi}
Let us first consider the case when $\overline{c}'_{22}=0$, that is $\rhobar$ is semisimple.
\begin{prop}
\label{expl alpha1}
Let $\widetilde{R}$ be the quotient of $\F[\![c_{11},c_{13},c_{23},c_{31},\widetilde{c}_{32},c'_{22}]\!]$ by the relations:
\begin{align*}
&c_{11}c_{23}=0; && c_{11}\widetilde{c}_{32}=c_{13}c_{31}\widetilde{c}_{32}; && c_{11}c'_{22}=\frac{b-c}{a-b}c_{13}\widetilde{c}_{32};&& c_{13}c_{23}\widetilde{c}_{32}=0; && c_{23}c_{31}\widetilde{c}_{32}=0;
\end{align*}
\begin{align*}
(a-b)c_{13}c_{31}c'_{22}+(c-b)c_{13}\widetilde{c}_{32}+(e-a+c)c_{23}c_{31}=0.
\end{align*}
Then the ring  $R^{\tau, \overline{\beta}, \Box}_{\overline{\fM}, \rhobar}/\varpi$ is a power series ring over a quotient of $\widetilde{R}$.
\end{prop}
\begin{proof} We need to check that the relations defining $\tilde{R}$ are satisfied in $R^{\tau, \overline{\beta}, \Box}_{\overline{\fM}, \rhobar}/\varpi$. Throughout the proof, we work modulo $\varpi$. First, observe that replacing $c_{i1}$ with $\frac{c_{i1}}{c^*_{21}}$, $c_{i2}$ with $\frac{c_{i2}}{c^*_{12}}$ and $c_{i3}$ with $\frac{c_{i3}}{c^*_{33}}$, we eliminate $c^*_{12}, c^*_{21},c^*_{33}$ from all equations, so we can assume $c^*_{12}=c^*_{21}=c^*_{33}=1$ in what follows. The monodromy equations mod $\varpi$ solves $c_{12}$, $c_{22}$ in terms of the remaining variables:
\begin{align}
\label{rel alpha1}
c_{12}=&\frac{a-c}{a-b}c_{13}\widetilde{c}_{32},\\
\label{rel alpha2}
c_{22}=&\frac{e-a+c}{e-a+b}c_{23}\widetilde{c}_{32}.
\end{align}
The determinant condition thus gives
\begin{align}
\label{rel alpha3}
c_{11}c'_{22}=&\frac{b-c}{a-b}c_{13}\widetilde{c}_{32}.
\end{align}
From the relation $c_{11}c_{32}=c_{12}c_{31}$, using (\ref{rel alpha1}), the definition of $\widetilde{c}_{32}$  and (\ref{rel alpha3}), we obtain:
\begin{align*}
c_{11}\widetilde{c}_{32}=&c_{13}c_{31}\widetilde{c}_{32}.
\end{align*}
Multiplying (\ref{rel alpha2}) by $c_{13}$, (\ref{rel alpha1}) by $c_{23}$ and using the relation $c_{12}c_{23}=c_{13}c_{22}$, we obtain:
\begin{align*}
c_{13}c_{23}\widetilde{c}_{32}=&0.
\end{align*}
Using $c_{22}c_{31}=0$ and (\ref{rel alpha2}), we get
\begin{align*}
c_{23}c_{31}\widetilde{c}_{32}=&0.
\end{align*}
Finally, using the third monodromy equation and the previous relations, we obtain
\begin{align*}
(a-b)c_{13}c_{31}c'_{22}+(c-b)c_{13}\widetilde{c}_{32}+(e-a+c)c_{23}c_{31}=0.
\end{align*}
\end{proof}
\begin{prop} 
\label{expl alpha2}
The ring $\widetilde{R}$ is a reduced 3-dimensional Cohen-Macaulay ring. It has 6 minimal primes, and each irreducible component is formally smooth over $\F$. Thus $e(\widetilde{R})=6$. Furthermore, the minimal primes of $\widetilde{R}$ are exactly
\begin{align*}
&(c_{11}-c_{13}c_{31},c_{23},(a-b)c_{31}c'_{22}+(c-b)\widetilde{c}_{32}); && (c_{11},(a-b)c_{13}c'_{22}+(e-a+c)c_{23},\widetilde{c}_{32});
\end{align*}
\begin{align*}
&(c_{11},c_{13},c_{23}); && (c_{11},c_{13},c_{31}); &&(c_{11},c_{31},\widetilde{c}_{32}); && (c_{23},\widetilde{c}_{32},c'_{22}).
\end{align*}
\end{prop}
\begin{proof} The proof is very similar to the proof of Proposition \ref{id explicit}, so we will only sketch it. The relation ideal $I$ defining $\widetilde{R}$ consists of polynomials, and indeed form a Gr\"obner basis with respect to the monomial order $c_{11}>c_{13}>c_{21}>c_{23}>c_{31}>\widetilde{c}_{32}>c'_{22}$. Since the initial ideal of $I$ is generated by square-free monomials, we get reducedness. Cohen-Macaulayness follows as in Proposition \ref{id explicit}.
\end{proof}

\subsubsection{The non-semisimple case.}
\label{sbsb:nsemi}

Finally, we handle the case where $\rhobar$ is non-semisimple. Then we have that $c'_{22}$ is a unit instead of a topologically nilpotent element. The relations of Propositions \ref{expl alpha1}, \ref{expl alpha2} continue to hold, the only difference is that $\widetilde{R}$ is not a quotient of 
$\F[\![c_{11},c_{13},c_{21},c_{23},c_{31},\widetilde{c}_{32},c'_{22}]\!]$, but rather a quotient of 
$\F[\![c_{11},c_{13},c_{21},c_{23},c_{31},\widetilde{c}_{32},c'_{22}-[\overline{c}'_{22}]]\!]$. The effect of $c'_{22}$ being a unit is that $\widetilde{R}$ only has $5$ minimal primes instead of $6$ (the minimal prime $(c_{23},c_{32},c'_{22})$ is no longer present). 

The proofs of the following results for the $\alpha$ shape are exactly the same as the proofs for the $\mathrm{id}$ shape, so we will not repeat them.
\begin{prop}  Assume that $\tau$ is strongly generic. Then $R^{\tau, \overline{\beta}, \Box}_{\overline{\fM}, \rhobar}/\varpi$ is isomorphic to a power series ring over $\tilde{R}$.
\end{prop}
\begin{cor} Assume that $\tau$ is strongly generic. Then the ring $R^{\mu,\tau}_{\rhobar}$ is normal, Cohen-Macaulay, and $R^{\mu,\tau}_{\overline{\fM}}[\frac{1}{p}]$ is a domain.
\end{cor}

This completes the proof of Theorem \ref{thm:conngenfib}.

\section{Appendix: Tables}

\begin{table}[h]
\caption{\textbf{The $(2,1,0)$-admissible elements}} \label{Table admissible elements} 
\begin{center}
\adjustbox{max width=\textwidth}{%
\begin{tabular}{| c | c |}
\hline
& \\
Length 4 & $\begin{matrix}
\alpha \beta \alpha \gamma=t_{(2,1,0)},&\beta \gamma \alpha \gamma=t_{(1,2,0)},& 
\beta \gamma \beta \alpha=t_{(0,2,1)},\\
\gamma \alpha \beta  \alpha=t_{(0,1,2)},& \alpha \gamma \alpha \beta=t_{(1,0,2)},& 
\alpha \beta \gamma \beta=t_{(2,0, 1)}
\end{matrix}$\\ 
\hline
Length 3 (ordinary) & $\begin{matrix}\gamma \alpha \beta,& \alpha \gamma \beta,& \alpha \beta \gamma & \beta \alpha \gamma,&
\beta \gamma \alpha,& \gamma \beta \alpha&\end{matrix}$ \\
\hline
Length 3 (shadow) & $\begin{matrix}\gamma \alpha \gamma,&\alpha \beta \alpha,&
\beta \gamma \beta&\end{matrix}$ \\
\hline
Length 2 & $\begin{matrix}\gamma \alpha,&\alpha \gamma,&\beta \alpha,&
 \alpha \beta,&\beta \gamma,&\gamma \beta\end{matrix}$ \\
\hline
Length 1 & $\begin{matrix}\alpha,&\beta,&\gamma\end{matrix}$ \\
\hline
Length 0 & $\begin{matrix}\Id\end{matrix}$\\
\hline
\end{tabular}}
\end{center}
\caption*{\Small{There are 25 different $(2,1,0)$-admissible elements. For simplicity, we label them by the corresponding element in the affine Weyl group of $\SL_3$, e.g. $\alpha\beta \alpha \gamma$ corresponds to $v (\alpha \beta \alpha \gamma)$ in $\widetilde{W}$.  If $(x,y,z) \in X_*(T) \cong \Z^3$ is a  cocharacter, we write $t_{(x,y,z)}$ for the image of translation by $(x,y,z)$ in $\widetilde{W}$.}}
\end{table}

\begin{table}[h]
\caption{\textbf{Inertial local Langlands}} \label{ILL} 
\begin{center}
\adjustbox{max width=\textwidth}{%
\begin{tabular}{| c | c |}
\hline
$\tau$ & $\sigma(\tau)$ \\
\hline
\hline
& \\
${\omega}_{f}^{-\bf{a}_1^{(0)}}\oplus \omega_f^{-\bf{a}_2^{(0)}}\oplus\omega_f^{-\bf{a}_3^{(0)}}$& $\Ind_{\B(k)}^{\GL_3(k)}(\tld{\omega}_{f}^{\bf{a}_1^{(0)}}\otimes \tld{\omega}_f^{\bf{a}_2^{(0)}}\otimes\tld{\omega}_f^{\bf{a}_3^{(0)}})$ \\
& \\
\hline
& \\
$\omega_{f}^{-\bf{a}_1^{(0)}} \oplus \omega_{2f}^{-\bf{a}_2^{(0)}-p^f\bf{a}_3^{(0)}}\oplus \omega_{2f}^{-\bf{a}_3^{(0)} - p^f\bf{a}_2^{(0)}}$& $\Ind_{P_2(k)}^{\GL_3(k)}\tld{\omega}_f^{\bf{a}_{1}^{(0)}}\otimes
\Theta(\tld{\omega}_{2f}^{\bf{a}_{2}^{(0)}+p^f\bf{a}_{3}^{(0)}})$ \\
& \\
\hline
& \\
$\omega_{3f}^{-\bf{a}_1^{(0)}-p^f\bf{a}_2^{(0)}-p^{2f}\bf{a}_3^{(0)}}\oplus \omega_{3f}^{-\bf{a}_2^{(0)}-p^f\bf{a}_3^{(0)}-p^{2f}\bf{a}_1^{(0)}}\oplus \omega_{3f}^{-\bf{a}_3^{(0)}-p^f\bf{a}_1^{(0)}-p^{2f}\bf{a}_2^{(0)}}$ & $\Theta(\tld{\omega}_{3f}^{\bf{a}_{1}^{(0)}+p^f\bf{a}_{2}^{(0)}+p^{2f}\bf{a}_{3}^{(0)}})$ \\
& \\
\hline
\end{tabular}}
\end{center}
\caption*{\Small{
In the table above, we set $P_2\defeq\text{\tiny{$\maq{*}{*}{*}{}{*}{*}{}{*}{*}$}}$ and write $\Theta(\psi)$ for the cuspidal representation of $\GL_{r}(k)$ associated to a $k'$-primitive character $\psi:(k')^\times\rightarrow E^{\times}$ as in \cite[Lemma 4.4]{herzig-duke}.}
}
\end{table}

\begin{table}[h]
\caption{\textbf{Jordan--H\"older factors of Deligne--Lusztig $\GL_3(\F_p)$-representations}} \label{table:JH} 
\begin{center}
\adjustbox{max width=\textwidth}{%
\begin{tabular}{| c | c |}
\hline
$\sigma(\tau)$ & $\JH(\sigmabar(\tau))$ \\
\hline
\hline
& \\
$\Ind_{\B(\F_p)}^{\GL_3(\F_p)}(\tld{\omega}_{1}^a\otimes \tld{\omega}_1^b\otimes\tld{\omega}_1^c)$& $\begin{matrix}
F(a,b,c),\\ F(c+p-1,a,b),\\ F(b,c,a-p+1),\\
F(a, c, b-p+1),\\ F(c+p-1,b,a-p+1),\\ F(b+p-1,a,c),\\
F(a-1,b,c+1),\\ F(c+p-2,a,b+1),\\ F(b-1,c,a-p+2)
\end{matrix}$ \\
& \\
\hline
& \\
$\Ind_{P_2(\F_p)}^{\GL_3(\F_p)}\tld{\omega}_1^a\otimes
\Theta(\tld{\omega}_2^{b+pc})$& $\begin{matrix}
F(a,b-1,c+1),\\ F(c+p-2,a,b+1),\\ F(b-1,c+1,a-p+1),\\
F(a, c, b-p+1),\\ F(c+p-1,b,a-p+1),\\ F(b+p-2,a,c+1),\\
F(a-1,b,c+1),\\ F(c+p-1,a,b),\\ F(b-1,c,a-p+2)
\end{matrix}$ \\
& \\
\hline
& \\
$\Theta(\tld{\omega}_3^{a+pb+p^2c})$ & $\begin{matrix}
F(a-2,b+1,c+1),\\ F(c+p-2,a,b+1),\\ F(b-1,c+1,a-p+1),\\
F(a-1, c, b-p+2),\\ F(c+p-1,b,a-p+1),\\ F(b+p-1,a-1,c+1),\\
F(a-1,b,c+1),\\ F(c+p-1,a-1,b+1),\\ F(b,c,a-p+1)
\end{matrix}$ \\
& \\
\hline
\end{tabular}}
\end{center}
\caption*{\Small{
In the table above, the triple $(a,b,c)$ is assumed to be weakly generic, we set $P_2\defeq\text{\tiny{$\maq{*}{*}{*}{}{*}{*}{}{*}{*}$}}$, and we write $\Theta(\psi)$ for the cuspidal representation of $\GL_{r}(\F_p)$ associated to a $\F_{p^r}$-primitive character $\psi:\F_{p^r}^\times\rightarrow E^{\times}$ as in \cite[Lemma 4.4]{herzig-duke}. See also \cite[Theorem 5.1]{florian-thesis}.}
}
\end{table}

\begin{table}[h]
\caption{\textbf{Shapes of Kisin modules over $\F$}}\label{table shapes mod p}
\centering
\begin{tabular}{| c | c || c | c |}
\hline
& & &\\
$\widetilde{w}_j $ & $\overline{A}^{(j)}_{\widetilde{w}_j}$ & $\widetilde{w}_j $ & $\overline{A}^{(j)}_{\widetilde{w}_j}$\\ 
& & &\\
\hline
& & &\\
$\alpha \beta \alpha \gamma$ &  $\begin{pmatrix} v^2\overline{c}_{11}^* & 0 & 0 \\ v^2\overline{c}_{21} & v \overline{c}_{22}^* & 0 \\ \overline{c}_{31}v+\overline{c}'_{31}v^2 & v\overline{c}_{32} & \overline{c}_{33}^*\end{pmatrix}$  & $\beta \gamma \alpha \gamma$  &$\begin{pmatrix} v\overline{c}_{11}^* & v\overline{c}_{12} & 0 \\ 0 & v^2\overline{c}_{22}^* & 0 \\ v\overline{c}_{31} & \overline{c}_{32}v+\overline{c}'_{32}v^2 & \overline{c}_{33}^*\end{pmatrix}$\\
& & &\\
\hline
& & &\\
$\beta \alpha \gamma$ & $\begin{pmatrix} 0&v \overline{c}_{12}^*&0\\v^2\overline{c}_{21}^*&0&0\\  v\overline{c}_{31}+v^2\overline{c}'_{31}&v\overline{c}_{32}&\overline{c}_{33}^* \end{pmatrix}$ &$\alpha \beta \gamma$ &  $\begin{pmatrix} v^2\overline{c}_{11}^*&0&0\\v\overline{c}_{21}+v^2\overline{c}'_{21}&0& \overline{c}_{23}^* \\v^2\overline{c}_{31}&v\overline{c}_{32}^*&0 \end{pmatrix}$\\
& & &\\
\hline
& & &\\
$\alpha \beta \alpha$ & $\begin{pmatrix} 0&0&v \overline{c}_{13}^* \\0&v\overline{c}_{22}^*&v\overline{c}_{23}\\v\overline{c}_{31}^*&v\overline{c}_{32}&v\overline{c}_{33} \end{pmatrix}$ &   & \\
& & &\\
\hline \hline
& & &\\
$\alpha \beta$ & $\begin{pmatrix} 0&0&v \overline{c}_{13}^* \\v \overline{c}_{21}^* & 0 &v\overline{c}_{23}\\0 &v\overline{c}_{32}^*&v\overline{c}_{33} \end{pmatrix}$ & $\beta \alpha$ & $\begin{pmatrix} 0&v \overline{c}_{12}^* & 0 \\ 0 & 0 &v  \overline{c}_{23}^*\\v\overline{c}_{31}^*& v \overline{c}_{32} & v\overline{c}_{33} \end{pmatrix}$ \\
& & &\\
\hline
& & &\\
$\alpha$ & $\begin{pmatrix} 0&v \overline{c}_{12}^* & 0 \\ v\overline{c}_{21}^*  & v \overline{c}_{22} & 0 \\ 0 & 0 &  v\overline{c}_{33}^* \end{pmatrix}$ &
$\mathrm{id}$ & $\begin{pmatrix} v \overline{c}_{11}^* & 0 & 0 \\ 0 & v \overline{c}_{22}^* &  0\\ 0& 0  & v \overline{c}_{33}^*& \end{pmatrix}$
 \\
& & &\\
\hline
\end{tabular}
\caption*{In the above, we have $\overline{c}_{ik}, \overline{c}'_{ik} \in \overline{\F}$ and $\overline{c}_{ik}^* \in \overline{\F}^{\times}$.   }
\end{table}

\begin{table}[h]
\caption{\textbf{Deforming $\overline{\fM}$ by shape (without monodromy)}}
\label{table:lifts}
\centering
\adjustbox{max width=\textwidth}{
\begin{tabular}{| c || c | c | c |}
\hline
& & &\\
$\widetilde{w}_j$ & $\overline{A}^{(j)}_{\widetilde{w}_j}$ & $\deg(\tld{A}^{(j)}_{\widetilde{w}_j})$ & $\tld{A}^{(j)}_{\widetilde{w}_j}$ with height/det conditions\\ 
& & &\\
\hline
& & &\\
$\alpha \beta \alpha \gamma$ &  $\begin{pmatrix} v^2\overline{c}_{11}^* & 0 & 0 \\ v^2\overline{c}_{21} & v \overline{c}_{22}^* & 0 \\ v\overline{c}_{31}+v^2\overline{c}'_{31} & v\overline{c}_{32} & \overline{c}_{33}^*\end{pmatrix}$  & $\begin{pmatrix}2^* & \leq 0 & -\infty \\ v(\leq 1) & 1^* & -\infty \\ v(\leq 1) & v(\leq 0) & 0^* \end{pmatrix}$ & $\begin{pmatrix} (v + p)^2 c_{11}^* & 0 & 0 \\ v(v+p) c_{21} & (v + p) c_{22}^* & 0 \\ v (c_{31}+ (v+p) c'_{31}) & v c_{32} & c_{33}^*\end{pmatrix}$\\
& & &\\
\hline
& & &\\
$\beta \gamma \alpha \gamma$ &  $\begin{pmatrix} v\overline{c}_{11}^* & v\overline{c}_{12} & 0 \\ 0 & v^2\overline{c}_{22}^* & 0 \\ v\overline{c}_{31} & v \overline{c}_{32}+v^2 \overline{c}'_{32} & \overline{c}_{33}^*\end{pmatrix}$ & $\begin{pmatrix} 1^* & \leq 1 & -\infty \\ v( \leq 0) & 2^* & -\infty \\  v(\leq 0) & v(\leq 1) & 0^* \end{pmatrix}$ & $\begin{pmatrix} (v + p) c_{11}^* & (v + p)c_{12} & 0 \\ 0 & (v+p)^2 c_{22}^* & 0 \\ v c_{31} & v (c_{32}+(v+p) c'_{32})& c_{33}^*\end{pmatrix}$ \\
& & &\\
\hline
\hline
& & &\\
$\beta \alpha \gamma$ & $\begin{pmatrix} 0&v \overline{c}^*_{12}&0\\v^2\overline{c}_{21}^*&0&0\\v\overline{c}_{31}+v^2\overline{c}_{31}'&v\overline{c}_{32}&\overline{c}_{33}^* \end{pmatrix}$ & $\begin{pmatrix} \leq 1 &  1^* & -\infty \\  v(1^*) & \leq 1 & -\infty \\ v(\leq 1) & v(\leq 0) & 0^* \end{pmatrix}$
 & $\begin{matrix} \begin{pmatrix} (v + p) c_{11} & (v + p) c_{12}^* & 0\\ v(v + p) c_{21}^*& (v + p) c_{22} &0\\v (c_{31}+(v + p)  c_{31}')&v c_{32}& c_{33}^* \end{pmatrix} \\ c_{11} c_{22} = -p c_{12}^* c_{21}^* \end{matrix}$ \\
& & &\\
\hline
& & &\\
$\alpha \beta \gamma$ &  $\begin{pmatrix} v^2\overline{c}_{11}^*&0&0\\v\overline{c}_{21}+v^2\overline{c}_{21}'&0&\overline{c}_{23}^* \\v^2\overline{c}'_{31}&v\overline{c}_{32}^*&0 \end{pmatrix}$ &  $\begin{pmatrix} 2^* &  \leq 0 & -\infty \\  v(\leq 1) & \leq 0 & 0^* \\ v(\leq 1) & v(0^*) & \leq 0 \end{pmatrix}$
& $\begin{matrix} \begin{pmatrix} (v + p)^2 c_{11}^* &0&0\\v(c_{21} + (v+p) c_{21}') & c_{22}& c_{23}^*\\v(c_{21} c_{33} (c_{23}^*)^{-1}+ (v+p)c_{31}')&vc_{32}^*&c_{33} \end{pmatrix} \\ c_{22} c_{33} = - p c_{32}^* c_{23}^* \end{matrix} $\\
& & &\\
\hline \hline
& & &\\
$\alpha \beta \alpha$ & $\begin{pmatrix} 0&0&v \overline{c}^*_{13}\\0&v\overline{c}_{22}^*&v\overline{c}_{23}\\v\overline{c}_{31}^*&v\overline{c}_{32}&v\overline{c}'_{33} \end{pmatrix}$ & $\begin{pmatrix} \leq 0  &  \leq 0 & 1^* \\  -\infty & 1^* & \leq 1 \\ v(0^*) & v(\leq 0) & \leq 1 \end{pmatrix}$ & $\begin{matrix} \begin{pmatrix} c_{11} & c_{11} c_{32} (c_{31}^*)^{-1}& c_{13} + (v + p) c_{13}^* \\ 0  & (v + p)c_{22}^*& (v + p)c_{23}\\v c_{31}^* & vc_{32}&c_{33} + (v+p)c_{33}' \end{pmatrix} \\ c_{11} c_{33} = - p c_{13} c_{31}^* \\ c_{11} c_{33}' - c_{13} c_{31}^* + p c_{13}^*c_{31}^* = 0 \end{matrix} $\\
& & &\\
\hline 
\end{tabular}}
\end{table}
\begin{table}[p]
\adjustbox{max width=\textwidth}{%
\begin{tabular}{| c || c | c | c |}
\hline
& & &\\
$\alpha \beta$ & $\begin{pmatrix} 0&0&v \overline{c}_{13}^*\\v \overline{c}_{21}^* & 0 &v\overline{c}'_{23}\\0 &v\overline{c}_{32}^*&v\overline{c}'_{33} \end{pmatrix}$ & $\begin{pmatrix} \leq 0  &  \leq 0 & 1^* \\  v(0^*) & \leq 0 & \leq 1 \\ v(\leq 0) & v(0^*) & \leq 1 \end{pmatrix}$ & 
$\begin{matrix}
\begin{pmatrix}
c_{31}c_{12}(c_{32}^*)^{-1}&c_{12}&c_{13}+(v+p) c_{13}^*\\
vc_{21}^*&c_{22}&c_{23}+(v+p) c_{23}'\\
vc_{31}&vc_{32}^*&\left(c_{31}c_{23}(c_{21}^*)^{-1}+(v+p)c_{33}'\right)
\end{pmatrix}
\\
c_{22}c_{31}=-pc_{21}^*c_{32}^*\\
c_{12}c_{23}= c_{22} c_{13}\\
c_{21}^*c_{32}^*c_{13}-pc_{21}^*c_{32}^*c_{13}^*-c'_{33}c_{21}^*c_{12}=0
\end{matrix}$
\\
& & &\\
\hline
& & &\\
$\beta \alpha$ & $\begin{pmatrix} 0&v \overline{c}_{12}^* & 0 \\ 0 & 0 &v  \overline{c}_{23}^*\\v\overline{c}_{31}^*& v \overline{c}_{32} & v\overline{c}'_{33} \end{pmatrix}$ & $\begin{pmatrix} \leq 0 &  1^* & \leq 0 \\  - \infty & \leq 1 & 1^* \\ v(0^*) & v(\leq 0) & \leq 1 \end{pmatrix}$  & 
$\begin{matrix}
\begin{pmatrix}
c_{11}&\left((c_{31}^*)^{-1}c_{11}c_{32}+ (v+p)c_{12}^*\right)&c_{13}\\
0& (v+p) c_{22}'&  (v+p)c_{23}^*\\
c_{31}^*v&c_{32}v&c_{33}+ (v+p)c_{33}'
\end{pmatrix}
\\
c_{11}c_{33}=-pc_{31}^*c_{13}\\
c_{22}'(c_{11} c_{33}' - c_{13} c^*_{31}) = pc_{23}^*c_{12}^*c_{31}^*
\end{matrix}$

\\
& & &\\
\hline \hline
& & &\\
$\alpha$ & $\begin{pmatrix} 0&v \overline{c}_{12}^*  & 0 \\ v\overline{c}_{21}^*  & v \overline{c}'_{22} & 0 \\ 0 & 0 &  v\overline{c}_{33}^* \end{pmatrix}$ & $\begin{pmatrix} \leq 0  &  1^* & \leq 0\\  v(0^*) & \leq 1 & \leq 0 \\ v(\leq 0) & v(\leq 0) & 1^* \end{pmatrix}$ & 
$\begin{matrix}
\begin{pmatrix}
c_{11}&c_{12}+ (v+p)c_{12}^*&c_{13}\\
c_{21}^*v&c_{22}+ (v+p)c_{22}' &c_{23}\\
c_{31}v&c_{32}v&\left(c_{33}+(v+p) c_{33}^*\right)
\end{pmatrix}
\\
\text{all $2\times2$ minors of }\tld{A}_{\alpha}^{(j)}|_{v = -p} \text{vanish,}\\
c_{11}c_{22}'c_{33}^*+c_{13}c_{21}^*c_{32}-c_{13}c_{22}'c_{31}-c_{12}c_{21}^*c_{33}^*+pc_{21}^*c_{12}^*c_{33}^*=0
\end{matrix}$
\\
& & &\\
\hline \hline
& & &\\
id & $\begin{pmatrix} v \overline{c}_{11}^* & 0 & 0 \\ 0 & v \overline{c}_{22}^* &  0\\ 0& 0  & v \overline{c}_{33}^*  \end{pmatrix}$ & $\begin{pmatrix} 1^*  &  \leq 0 & \leq 0 \\  v(\leq 0) & 1^* & \leq 0 \\ v(\leq 0) & v(\leq 0) & 1^* \end{pmatrix}$ &  
$\begin{matrix}
\begin{pmatrix}
c_{11}+c_{11}^*(v+p)&c_{12}&c_{13}\\
vc_{21}&c_{22}+ c_{22}^*(v+p) &c_{23}\\
vc_{31}&vc_{32}&c_{33}+c_{33}^*(v+p)
\end{pmatrix}
\\
\text{all $2\times2$ minors of }\tld{A}_{\Id}^{(j)}|_{v = -p} \text{vanish,}\\
c_{11}c^*_{22}c^*_{33}+c_{22}c^*_{33}c^*_{11}+c_{33}c^*_{11}c^*_{22}-c_{11}^*c_{23}c_{32}-c_{22}^*c_{13}c_{31}-c_{33}^*c_{12}c_{21}+c_{21}c_{13}c_{32}=0 
\end{matrix}$
\\
& & &\\
\hline 
\end{tabular}}
\caption*{
\textbf{Explanation of the table:} $\deg(\tld{A}_{\tld{w}_j}^{(j)})$ is the degree of the polynomial in each entry.  We write $k^*$ to indicate an entry polynomial of degree $k$ whose leading coefficient is a unit. We use $c^*$ to indicate an entry which is a unit in $R$. Each entry is also subject to the condition that the reduction modulo $m_R$ gives $\overline{A}^{(j)}_{\widetilde{w}_j}$. The third column is further explained in Remark \ref{degreeintable}.   
In the fourth column, we describe $\Mat_{\beta}(\phi^{(j)}_{\fM,s_{j+1}(3)})$ with  finite height and determinant conditions imposed, and performing some obvious $p$-saturation. However, we do not claim that we have performed all the $p$-saturation in the case of the last 5 rows.
}
\end{table}

\begin{table}[h]
\caption{\textbf{Monodromy equations}}\label{table3}
\adjustbox{max width=\textwidth}{%
\begin{tabular}{| c || c | c |}
\hline
& &\\
$\widetilde{w}_j$ & $R^{\expl}_{\widetilde{w}_j}$ & Leading term  \\ 
& & \\
\hline
& & \\
$\alpha \beta \alpha \gamma$ & $ \cO [\![ x_{11}^*, x_{22}^*, x_{33}^*, x_{21}, x_{31}, x'_{31}, x_{32} ]\!]$ &$ 
\begin{matrix}
pc^*_{33}\big((e-(a-c))c^*_{22}c_{31}+p(e-(b-c))c_{21}c_{32}c^*_{33}-pec^*_{22}c'_{31}\big)+O(p^{n-1})\\
\text{(using the $(3,1)$ entry of the leading term)}
\end{matrix}
 $ \\
& & \\
\hline
& & \\
$\beta \gamma \alpha \gamma$ &   $ \cO [\![ x_{11}^*, x_{22}^*, x_{33}^*, x_{12}, x_{31}, x_{32}, x'_{32} ]\!]$ & $
\begin{matrix}
pc^*_{33}\big((e-(b-c)) c_{32} c_{11}^* - (e - (a-c)) c_{12} c_{31}+pec_{11}c'_{32}\big)+O(p^{n-1})\\
\text{(using the $(3,2)$ entry of the leading term)}
\end{matrix}
$ \\
& &\\
\hline
\hline
& & \\
$\beta \alpha \gamma$ & $\begin{matrix} \cO [\![ x_{11}, x_{12}^*, x_{21}^*, x_{33}^*, x_{22}, x_{31}, x'_{31}, x_{32} ]\!] \\
 c_{11} c_{22} =  -p c_{12}^* c_{21}^* \end{matrix} $   & 
$
\begin{matrix}
pc_{33}^*\big((e-(a-c))c_{12}^*c_{31}-pec_{12}^*c_{31}'-(e-(b-c))c_{32}c_{11}\big)+O(p^{n-1})
\\
\text{(using the $(3,2)$ entry of the leading term)}
\end{matrix}$ 
\\
& & \\
\hline
& & \\
$\alpha \beta \gamma$ & $\begin{matrix}  \cO [\![x_{11}^*, x_{21}, x_{21}', x_{22}, x_{23}^*, x'_{31}, x_{32}^*, x_{33} ]\!] \\ c_{22} c_{33} = - p c_{32}^* c_{23}^* \end{matrix} $ & 
$\begin{matrix}
pc_{23}^*\big((e-(a-c))c_{32}^*c_{21}+(b-c)c_{22}c'_{31}-p(e-(b-c))c_{32}^*c_{21}'\big)+O(p^{n-1})\\
\text{(using the $(2,1)$ entry of the leading term)}
\end{matrix}$
\\
& & \\
\hline \hline
& & \\
$\alpha \beta \alpha$ & $\begin{matrix} \cO[\![x_{11}, x_{32}, x_{23}, x_{13}, x_{33}, x_{33}', x_{31}^*, x_{22}^*, x_{13}^*]\!] \\ c_{11}c_{33}=-pc_{13}c_{31}^*\\
c_{11}c_{33}'-c_{13}c_{31}^*+pc_{13}^*c^*_{31}=0 \end{matrix} $  &
$\begin{matrix}
-pc_{31}^*\big((e-(a-c))c_{33}c_{22}^*-p(a-b)c_{23}c_{32}+pec_{22}^*c'_{33})\big)+O(p^{n-1})
\\
\text{(using the $(3,1)$ entry of the leading term)}
\end{matrix}
$
\\
& & \\
\hline \hline
& & \\
$\alpha \beta$ & $\begin{matrix}
\cO[\![x_{31}, x_{22}, x_{12}, x_{13}, x_{23}, x_{23}', x_{33}', x_{21}^*, x_{13}^*, x_{32}^*]\!]
\\
c_{22}c_{31}=-pc_{21}^*c_{32}^*\\
c_{12}c_{23}= c_{22} c_{13}\\
c_{32}^*c_{13}-pc_{32}^*c_{13}^*-c_{33}'c_{12}=0
\end{matrix}$ & 
$\begin{matrix}
pc^*_{32}\big((e-(a-c))c_{31}c_{23}+p(e-(a-b))c^*_{21}c'_{33}+p(a-b)c_{31}c'_{23}\big)+O(p^{n-1})\\
\text{(using the $(3,1)$ entry of the leading term)}
\end{matrix}
$

\\
& & \\
\hline
& & \\
$\beta \alpha$  & 
$\begin{matrix}
\cO[\![x_{11}, x_{22}', x_{32}, x_{13}, x_{33}, x_{33}', x_{31}^*, x_{12}^*, x_{23}^*]\!]
\\
c_{11}c_{33}=-pc_{31}^*c_{13}\\
c_{22}'(c_{11} c_{33}' - c_{13} c^*_{31}) = pc_{23}^*c_{12}^*c_{31}^*
\end{matrix}$ &
$\begin{matrix}
pc^*_{31}\big(-(e-(a-c))c_{33}c'_{22}+p(a-b)c_{32}c^*_{23}-pec'_{22}c'_{33}\big)+O(p^{n-1})\\
\text{(using the $(3,1)$ entry of the leading term)}
\end{matrix}
$
\\
&  &\\
\hline \hline
&  &\\
$\alpha$ &  
See \S \ref{sec:badcases}.
&
$\begin{matrix}
p\big((e-a + c)  (c_{23} c_{32} c_{21}^* - c'_{22} c_{23} c_{31}) - (e - a +b) c_{22} c^*_{33}  c_{21}^*  - ep c'_{22} c^*_{33}  c_{21}^* \big)+O(p^{n-1})\\
\text{(using the $(2,1)$ entry of the leading term)}
\end{matrix}
$
\\
& &\\
\hline \hline
&  &\\
$\Id$  & 
See \S \ref{sec:badcases}.&See \S \ref{sec:badcases}.
\\
&  &\\
\hline \hline
\end{tabular}}
\caption*{We list the generator of $I_{\mathrm{mon}}^{(j+1)}[1/p]$ according to the procedure of Proposition \ref{onequation}. We take $a\defeq a_{s_{j+1}(1)}^{(j+1)},\ b\defeq a_{s_{j+1}(2)}^{(j+1)},\ c\defeq a_{s_{j+1}(3)}^{(j+1)}$. We define the variables $x_{ij}^{\bullet}\defeq c_{ij}^{\bullet}-[\overline{c}_{ij}^{\bullet}]$ where $\bullet=*,\ \prime$ or $\bullet\in\emptyset$ and $[\cdot]$ is the Teichm\"uller lift.}
\end{table}
\begin{table}[p]
\caption{\textbf{Deformation rings with monodromy}}\label{table:withmon}
\adjustbox{max width=\textwidth}{%
\begin{tabular}{| c || c | c |}
\hline
& & \\
$\widetilde{w}_j$ & Condition on $\overline{\fM}$ & $R^{\expl, \nabla}_{\overline{\fM}, \widetilde{w}_j}$ \\ 
& & \\
\hline
& & \\
$\alpha \beta \alpha \gamma$ & $\overline{c}_{31} = 0$  &  $ \cO [\![ x_{11}^*, x_{22}^*, x_{33}^*, x_{21}, x'_{31}, x_{32} ]\!]$ \\
& & \\
\hline
& & \\
$\beta \gamma \alpha \gamma$ & $(e-b+c) \overline{c}_{32} \overline{c}_{11}^* = (e - a+c) \overline{c}_{12} \overline{c}_{31}$ & $\cO [\![ x_{11}^*, x_{22}^*, x_{33}^*, x_{12}, x_{31}, x'_{32} ]\!]$\\
& & \\
\hline
\hline
& & \\
$\beta \alpha \gamma$ & $\overline{c}_{31} = 0$  & $\cO [\![ y_{11}, y_{22}, x_{12}^*, x_{21}^*, x_{33}^*, x'_{31}, x_{32} ]\!]/ (y_{11} y_{22} - p) $\\
& & \\
\hline
& & \\
$\alpha \beta \gamma$ & $\overline{c}_{21}= 0$ & $\cO [\![y_{22}, y_{33}, x_{11}^*, x_{21}', x_{23}^*, x'_{31}, x_{32}^*]\!]/(y_{22} y_{33} - p) $  \\
& & \\
\hline \hline
& & \\
$\alpha \beta \alpha$ & $(a-b) \overline{c}_{23} \overline{c}_{32} -(a-c) \overline{c}_{22}^* \overline{c}'_{33} \neq 0$ &$\cO[\![x_{32}, x_{23}, x_{33}', x_{31}^*, x_{22}^*, x_{13}^*]\!] $
\\
& & \\
\hline
& & \\
$\alpha \beta \alpha$ & $(a-b) \overline{c}_{23} \overline{c}_{32} -(a-c) \overline{c}_{22}^* \overline{c}'_{33} = 0$ & $\cO[\![x_{11}, x_{32}, x_{23}, y_{33}', x_{31}^*, x_{22}^*, x_{13}^*]\!]/(x_{11}y_{33}'-p)$
\\
& & \\
\hline \hline
& & \\
$\alpha \beta$ & $\overline{c}'_{33} \neq 0 $&$\cO[\![y_{31}, x_{22}, x_{23}', x_{33}', x_{21}^*, x_{13}^*, x_{32}^*]\!]/(y_{31} x_{22} - p)$\\
& & \\
\hline
& & \\
$\alpha \beta$ & $\overline{c}'_{33}= 0$ & $\cO[\![y_{31}, x_{22}, x_{12}, x_{23}', y_{33}', x_{21}^*, x_{13}^*, x_{32}^*]\!]/(y_{31} x_{22} - p, x_{12}y_{33}'-p)$\\
& &
\\
\hline
& & \\

$\beta \alpha$  & $\overline{c}_{32} \neq 0$ &$\cO[\![x_{22}', y_{13}, x_{32},  x_{33}', x_{31}^*, x_{12}^*, x_{23}^*]\!]/( x_{22}' y_{13} - p)$
\\
& &
\\
\hline
& & \\
$\beta \alpha$  & $\overline{c}_{32} = 0$ & $\cO[\![x_{11}, x_{22}', y_{32}, y_{13}, x_{33}', x_{31}^*, x_{12}^*, x_{23}^*]\!]/( x_{22}' y_{13} - p, x_{11}y_{32}-p)$
\\
&  &\\
\hline
\end{tabular}}
\caption*{The condition imposed by monodromy on the coefficients of $\Mat_{\overline{\beta}}\big(\phi^{(j)}_{\overline{\fM},s_{j+1}(3)}\big)$, according to the shape of $\overline{\fM}$ (cf. Table \ref{table shapes mod p}). 
We take $a\defeq a_{s_{j+1}(1)}^{(j+1)},\ b\defeq a_{s_{j+1}(2)}^{(j+1)},\ c\defeq a_{s_{j+1}(3)}^{(j+1)}$.
Note that the table above covers all the shapes of length $\geq 2$; the shapes of length $\leq1$ are more delicate and treated in detail in \S \ref{sec:badcases}.}
\end{table}

\begin{table}[h]
\caption{\textbf{Types with Weyl intersection in the proof of Proposition \ref{prop:intersec}(2)}}\label{table: types}
\adjustbox{max width=\textwidth}{%
\begin{tabular}{| c | c | c | c|}
\hline
&&&\\
$\rhobar$& $ \sigma(\tau) $ & $\JH(\sigma(\tau))\cap W^{?}(\rhobar)$& Shape $\bf{w}(\rhobar, \tau)$ \\ 
&&&\\
\hline
&&&\\
\multirow{3}{*}{$\big(\omega^{a}\oplus\omega^b\oplus\omega^{c}\big)\otimes\omega$}&$\Ind_{P_2(\Fp)}^{\GL_3(\Fp)}(\teich{\omega}^b\otimes\Theta(\teich{\omega}_2^{c+pa}))$ & 
$\begin{matrix}F(a-1,b,c+1),\\ F(c+p-1,b,a-p+1)\\
\end{matrix}$&$\alpha\beta\alpha$\\
&&&\\
\cline{2-4}
&&&\\
& $\Ind_{P_2(\Fp)}^{\GL_3(\Fp)}(\teich{\omega}^c\otimes\Theta(\teich{\omega}_2^{(a+1)+p(b-1)}))$
& 
$\begin{matrix}F(b-1,c,a-p+2),\\ F(a,c,b-p+1)\\
\end{matrix}$&$\beta\gamma\beta$\\
&&&\\
\cline{2-4}
&&&\\
&$
\Ind_{P_2(\Fp)}^{\GL_3(\Fp)}(\teich{\omega}^a\otimes\Theta(\teich{\omega}_2^{(b+1)+p(c-1)}))$& 
$\begin{matrix}F(c+p-2,a,b+1),\\ F(b+p-1,a,c)\\
\end{matrix}$&$\gamma\alpha\gamma$\\
&&&\\
\hline
\hline
&&&\\
\multirow{3}{*}{$\big(\omega^{a}\oplus \Ind_{G_{\Qpf{2}}}^{G_{\Qp}}\omega_2^{b+pc}\big)\otimes\omega$}&$\Theta(\teich{\omega}_3^{a+pb+p^2c})$
& 
$\begin{matrix}F(a-1,b,c+1),\\ F(c+p-1,b,a-p+1)\\
\end{matrix}$&$\alpha\beta\alpha$\\
&&&\\
\cline{2-4}
&&&\\
&$\Theta(\teich{\omega}_3^{(c-1)+pb+p^2(a+1)})$
& 
$\begin{matrix}F(b-1,c,a-p+2),\\ F(a,c,b-p+1)\\
\end{matrix}$&$\beta\gamma\beta$\\
&&&\\
\cline{2-4}
&&&\\
&$\Ind_{\B(\Fp)}^{\GL_3(\Fp)}\big(\omega^{a}\otimes\omega^{b-1}\otimes\omega^{c+1}\big)$
& 
$\begin{matrix}F(c+p-1,a,b),\\ F(b+p-2,a,c+1)\\
\end{matrix}$&$\gamma\alpha\gamma$\\
&&&\\
\hline\hline
&&&\\
\multirow{3}{*}{$\big(\Ind_{G_{\Qpf{3}}}^{G_{\Qp}}\omega_3^{a+pb+p^2c}\big)\otimes\omega$}&$\Ind_{P_2(\Fp)}^{\GL_3(\Fp)}\big(\omega^a\otimes \Theta(\teich{\omega}_2^{b+pc})\big)$
& 
$\begin{matrix}F(a-1,b,c+1),\\ F(c+p-1,b,a-p+1)\\
\end{matrix}$&$\alpha\beta\alpha$\\
&&&\\
\cline{2-4}
&&&\\
&$\Ind_{P_2(\Fp)}^{\GL_3(\Fp)}\big(\omega^{c+1}\otimes \Theta(\teich{\omega}_2^{(a-1)+pb})\big)$
& 
$\begin{matrix}F(c+p-1,a-1,b+1),\\ F(b+p-1,a-1,c+1)\\
\end{matrix}$&$\gamma\alpha\gamma$\\
&&&\\
\cline{2-4}
&&&\\
&$\Ind_{P_2(\Fp)}^{\GL_3(\Fp)}\big(\omega^{b+1}\otimes \Theta(\teich{\omega}_2^{(c-1)+pa})\big)$
& 
$\begin{matrix}F(b,c,a-p+1),\\ F(a-1,c,b-p+2)\\
\end{matrix}$&$\beta\gamma\beta$\\
&&&\\
\hline
\end{tabular}}
\end{table}

\begin{table}[h]
\adjustbox{max width=\textwidth}{%
\begin{tabular}{| c | c | c | c|}
\hline
&&&\\
\multirow{3}{*}{$\big(\omega^{b}\oplus \Ind_{G_{\Qpf{2}}}^{G_{\Qp}}\omega_2^{a+pc}\big)\otimes\omega$}&$
\Ind_{\B(\Fp)}^{\GL_3(\Fp)}\big(\omega^{a}\otimes\omega^{b}\otimes\omega^{c}\big)$
& 
$\begin{matrix}F(a-1,b,c+1),\\ F(c+p-1,b,a-p+1)\\
\end{matrix}$&$\alpha\beta\alpha$\\
&&&\\
\cline{2-4}
&&&\\
&$\Theta\big(\teich{\omega}_3^{c+p(b+1)+p^2(a-1)}\big)$
& 
$\begin{matrix}F(c+p-1,a-1,b+1),\\ F(b+p-1,a-1,c+1)\\
\end{matrix}$&$\gamma\alpha\gamma$\\
&&&\\
\cline{2-4}
&&&\\
&$\Theta\big(\teich{\omega}_3^{a+p(b-1)+p^2(c+1)}\big)$
& 
$\begin{matrix}F(b-1,c+1,a-p+1),\\ F(a-1,c+1,b-p+1)\\
\end{matrix}$&$\beta\gamma\beta$\\
&&&\\
\hline
\end{tabular}}
\end{table}

\begin{table}[h]
\caption{\textbf{Serre weights for semisimple $\rhobar$}}\label{table:SW}
\adjustbox{max width=\textwidth}{%
\begin{tabular}{| c | c | c |}
\hline
&&\\
$\rhobar$& $W_{\mathrm{obv}}(\rhobar)$ & $W^?(\rhobar)\setminus W_{\mathrm{obv}}(\rhobar)$ \\ 
&&\\
\hline
&&\\
$\big(\omega^{a}\oplus\omega^b\oplus\omega^{c}\big)\otimes\omega$&
$\begin{matrix}F(a-1,b,c+1),\\ F(b-1,c,a-p+2),\\ F(c+p-2,a,b+1),\\
F(a-1, c, b-p+2),\\ F(b+p-2,a,c+1),\\ F(c+p-2,b,a-p+2)\end{matrix}$
&
$\begin{matrix}F(c+p-1,b,a-p+1),\\
 F(b+p-1,a,c),\\ F(a,c,b-p+1)\end{matrix} $
\\
&&\\
\hline
&&\\
$\big(\omega^{a}\oplus \Ind_{G_{\Qpf{2}}}^{G_{\Qp}}\omega_2^{b+pc}\big)\otimes\omega$&
$\begin{matrix}
F(a-1,b,c+1),\\ F(b-1,c,a-p+2),\\ F(c+p-1,a,b),\\
F(a-1, c+1, b-p+1),\\ F(c+p-1,b-1,a-p+2),\\ F(b+p-1,a,c)
\end{matrix}$
&
$\begin{matrix}
F(c+p-1,b,a-p+1),\\ F(a,c,b-p+1),\\F(b+p-2,a,c+1)
\end{matrix}$\\
&&\\
\hline
&&\\
$\big(\Ind_{G_{\Qpf{3}}}^{G_{\Qp}}\omega_3^{a+pb+p^2c}\big)\otimes\omega$&
$\begin{matrix}
F(a-1,b,c+1),\\ F(c+p-1,a-1,b+1),\\ F(b,c,a-p+1),\\
F(a-1, c+1, b-p+1),\\ F(c+p-1,b+1,a-p),\\ F(b+p-1,a,c)
\end{matrix}$
&
$\begin{matrix}
F(c+p-1,b,a-p+1),\\ F(b+p-1,a-1,c+1),\\ F(a-1,c,b-p+2)
\end{matrix}$\\
&&\\
\hline
&&\\
$\big(\omega^{b}\oplus \Ind_{G_{\Qpf{2}}}^{G_{\Qp}}\omega_2^{a+pc}\big)\otimes\omega$
& 
$\begin{matrix}
F(a-1,b,c+1),\\ F(b-1,c+1,a-p+1),\\ F(c+p-1,a-1,b+1),\\
F(a-1, c, b-p+2),\\ F(c+p,b,a-p),\\ F(b+p-2,a,c+1)
\end{matrix}$
&
$\begin{matrix}
F(c+p-1,b,a-p+1),\\ F(b+p-1,a-1,c+1),\\ F(a-1,c+1,b-p+1)
\end{matrix}$\\
&&\\
\hline
\end{tabular}}
\caption*{
The triple $(a,b,c)\in \Z^3$ verifies $1<a-b,\ b-c <p-2$ and $a-c<p-2$. The table is deduced from \cite{herzig-duke}, Lemma 7.6 and Proposition 6.28; alternatively the obvious weights can be deduced from \cite{BLGG}, Lemma 5.1.2.}
\end{table}

\newpage
\bibliography{Biblio}
\bibliographystyle{amsalpha}

\end{document}